\documentclass[12pt]{amsart}
\usepackage[english]{babel}
\usepackage{graphicx}
\usepackage{color}
\usepackage{amsmath,amssymb,hyperref,srcltx}
\usepackage[shortlabels]{enumitem}
\usepackage{xifthen}
\DeclareMathOperator{\ord}{ord}

\newcommand{\cn}[1]{\mbox{(${\mathbb C}^{#1},0$)}}

\def\tidy{ordered}
\def\nice{separated}

 \newboolean{pedro}
\setboolean{pedro}{true}
\newcommand{\pedro}{\ifthenelse{\boolean{pedro}}{\color{magenta}
 \setboolean{pedro}{false}}{\color{black}\setboolean{pedro}{true}}}
\newcounter{margin}

 \newboolean{javier}
\setboolean{javier}{true}
\newcommand{\javier}{\ifthenelse{\boolean{javier}}{\color{red}
 \setboolean{javier}{false}}{\color{black}\setboolean{javier}{true}}} 

\newtheorem{pro}{Proposition}[section]
\newtheorem{teo}{Theorem}[section]
\newtheorem*{teo_o}{Theorem}

\newtheorem{cor}{Corollary}[section]

\newtheorem{lem}{Lemma}[section]

\theoremstyle{remark}
\newtheorem{rem}{Remark}[section]

\theoremstyle{definition}
\newtheorem{defi}{Definition}[section]

\usepackage{algorithm}
\usepackage{algpseudocode}

\title[Generic Component of the Analytic Moduli]{Explicit Computation of The Generic Component of the Analytic Moduli of a Plane Branch}

\author{Pedro Fortuny Ayuso and Javier Rib\'{o}n}
\email{fortunypedro@uniovi.es, jribon@id.uff.br}
\address{Universidad de Oviedo, Spain --- Universidade Federal Fluminense, Brazil}
\date{\today}

\begin{document}

\begin{abstract}
Let ${\mathcal C}$ be a fixed equisingularity class of irreducible germs of complex analytic
plane curves. We compute a basis of the ${\mathbb C}[[x]]$-module of K\"ahler differentials for generic 
$\Gamma \in {\mathcal C}$, algorithmically, and study its behaviour under blow-up.

As a first application, we give an algorithm providing the generic semimodule in an equisingularity class
in terms of its multiplicity and its Puiseux characteristic exponents. 
As another application,  we give an alternative proof for a formula of Genzmer, that provides  
the dimension of the moduli of analytic classes in the equisingularity class of $\Gamma$.
\end{abstract}

\maketitle

\tableofcontents

\bibliographystyle{plain}

\section{Introduction} \label{sec:intro}
What is the structure of the generic component of the analytic moduli of a singular plane branch? What is its behavior under blow-up? These are the two questions we address in this work, by means of explicitly and algorithmically computing a basis of the semi-module associated to a generic singularity in a topological class.

Let us explain our work in detail.
We are interested in problems regarding the analytic classification of
irreducible germs of complex analytic plane curves, the so called {\it
  plane branches}.  The topological classification of germs
of plane curves has a simple classical solution (cf. \cite[Theorem
8.4.21]{brieskorn2012plane}).  Consider, in $(\mathbb{C}^2,0)$ with
coordinates $(x,y)$, an irreducible germ of complex analytic plane
curve $\Gamma_{\circ}$.  Given a plane branch $\Gamma$, it is said to be topologically (resp. analytically) conjugated to
$\Gamma_{\circ}$, denoted
$\Gamma \sim_{\mathrm{top}} \Gamma_{\circ}$ (resp.
$\Gamma \sim_{\mathrm{ana}} \Gamma_{\circ}$) if there exists a germ of
homeomorphism (resp. biholomorphism) $\phi$ defined in a neighborhood
of the origin of $\mathbb{C}^2$ such that
$\phi (\Gamma_{\circ}) = \Gamma$.  It is well-known that the set
\[ \tilde{\mathcal C} :=
\{ \Gamma : \ \Gamma \ {\rm is \ a \ plane \ branch \ such \ that} \ \Gamma_{\circ} \sim_{\mathrm{top}} \Gamma \} \]
consists of the plane branches with the same graph of resolution of singularities as $\Gamma_{\circ}$ and hence it is called 
{\it the equisingularity class} of $\Gamma$. Alternatively, $\Gamma \in \tilde{\mathcal C}$ if and only if the semigroup of values ${\mathcal S} (\Gamma_{\circ})$ and ${\mathcal S} (\Gamma)$ coincide, where 
\[
  {\mathcal S} (\Gamma)=\left\{ \nu_{\Gamma} (f) := \ord_t(f(\Gamma(t)))\ :\ f(x,y)\in
    \mathbb{C}[[x,y]] \right\} 
\]
and $\Gamma (t)$ is an irreducible Puiseux parametrization of $\Gamma$. 

\strut

It is natural to refine this topological classification by considering the analytic classification, both 
providing a complete system of analytic invariants and describing the properties of the moduli space
$\tilde{\mathcal{C}}/ \sim_{\mathrm{ana}}$. The salience of this problem was noted by Zariski in his 1973 Course 
at l'\'{E}cole Polytechnique in Paris \cite{Zariski:moduli}, where he also described the moduli space for several topological 
classes. The problem of analytic classification of plane curves was solved by Hefez, Hernandes and Rodrigues Hernandes 
in a series of papers
\cite{Hefez-Hernandes-classification, Hefez-Hernandes-Rodrigues:classification, Hernandes-Rodrigues:classification}
starting in 2011. Let us focus in the case of plane branches \cite{Hefez-Hernandes-classification}.
The main invariant of the classification is the semimodule 
\[
 \Lambda_{\Gamma} = \big\{
 \nu_{\Gamma} (\omega) := \ord_t (\Gamma(t)^{\ast}\omega) + 1 \ |\ \big.
 \big.
 \omega \in  \hat{\Omega}(\mathbb{C}^2,0)
 \big\} ,
\]
where $ \hat{\Omega}(\mathbb{C}^2,0) := \mathbb{C}[[x,y]] dx + \mathbb{C}[[x,y]] dy$ is the set of formal $1$-forms and 
we define $\ord_t (h(t) dt) =  \ord_t (h(t))$ for $h(t) \in {\mathbb C}[[t]]$.
Thus, in order to continue the study of the properties 
of the moduli spaces $\tilde{\mathcal{C}}/ \sim_{\mathrm{ana}}$, 
it is interesting not only to describe the semimodule, but also to explicitly 
compute bases of $1$-forms generating it in some sense.

The following, by no means exhaustive, is a summary of references  in the literature. 
The moduli space 
 $\tilde{\mathcal{C}}/ \sim_{\mathrm{ana}}$ of a plane branch $\Gamma_{\circ}$
 can be identified with the set of orbits of a regular action of a connected solvable algebraic 
 group $G$ on an affine space ${\mathbb C}^{p}$
 by a theorem of Ebey  \cite{Ebey:moduli}.
 As a consequence, that moduli space can also be identified with a constructible subset of an affine space. 
 There exists a Zariski open subset $U$ of ${\mathbb C}^{p}$ consisting of orbits of $G$ such that all orbits of 
 $G$ in $U$ have the same dimension $d$ by a theorem of Rosenlicht \cite{Rosenlicht:basic}.
 Indeed, $U/G$ can be identified with a Zariski open subset of an algebraic variety $W$ of dimension $p-d$
 (cf. \cite[Chapter VI.1]{Zariski:moduli}) that represents 
 the generic component of the moduli space  $\tilde{\mathcal{C}}/ \sim_{\mathrm{ana}}$.
 %
 %
 %Indeed $U/G$ can be identified with an open Zariski subset $V$ of some variety $W$ and there is
 %a dominant rational map  $\tau: \mathbb{C}^{p} \mathrel{-\,}\rightarrow W$ defined in $U$, with $\tau (U) = V$,  and sending
 %any point to its class in $U/G$ by a theorem of Rosenlicht \cite{algo} (see also \cite[Chapter VI.1]{Zariski:moduli}).
 %
 The generic dimension $p-d$ of the moduli was studied in the case of one Puiseux characteristic exponent by 
 Delorme \cite{Delorme:moduli} and Brian\c{c}on, Granger and Maisonobe \cite{Briancon-Granger-Maisonobe:modules}
 who gave an algorithm for its calculation. Genzmer provided a general formula for 
 the generic dimension of the   moduli space of any 
 plane branch that depends on simple combinatorial data 
 associated to the topological class of $\Gamma_{\circ}$ \cite{Genzmer-moduli-2020}.
 
 Delorme studied the generation of the semimodule $\Lambda_{\Gamma_{\circ}}$ over the semigroup $\mathcal{S} (\Gamma_{\circ})$
 in the case of a unique Puiseux characteristic exponent, 
 revealing interesting properties for the generating $1$-forms  \cite{Delorme:moduli}.
 This approach was adopted by Cano, Corral and Senovilla-Sanz in \cite{Cano-Corral-Senovilla:semiroots},
where they describe geometrical properties of Delorme's generating $1$-forms from the viewpoint of foliation theory.
 
\strut

Consider, for simplicity, the subset $\mathcal{C}$ of the topological class $\tilde{\mathcal{C}}$ of $\Gamma_{\circ}$ of 
plane branches $\Gamma$ whose tangent cone is different from $x=0$. 
Clearly every analytic class of an element of  $\tilde{\mathcal{C}}$ intersects $\mathcal{C}$ and hence 
$\tilde{\mathcal{C}}/\sim_{\mathrm{ana}} = \mathcal{C}/\sim_{\mathrm{ana}}$.
Given $\Gamma \in \mathcal{C}$,  
set
${\mathcal I}_{\Gamma} = \{ \omega \in \hat{\Omega}(\mathbb{C}^2,0) : \Gamma(t)^{*} \omega \equiv 0\}$.
We shall compute a ``nice'' basis of the ${\mathbb C}[[x]]$-module of K\"ahler  differentials for generic $\Gamma \in \tilde{\mathcal{C}}$.
Specifically:
 \begin{teo} 
 \label{teo:nice}
 There is a \emph{nice} basis for ${\mathcal C}$:
 If $n$ is the multiplicity of $\Gamma_{\circ}$ at the origin, 
 there exists an open subset $U$ of $\tilde{\mathcal{C}}$ and germs of $1$-forms
$\Omega_{1,\Gamma}, \hdots, \Omega_{n, \Gamma}$, for $\Gamma \in U$, such that 
\begin{itemize}
\item $\hat{\Omega}(\mathbb{C}^2,0) = \mathbb{C}[[x]] \Omega_{1,\Gamma} \oplus \hdots \oplus
  \mathbb{C}[[x]] \Omega_{n,\Gamma} \oplus {\mathcal I}_{\Gamma}$;
\item The values $\nu_{\Gamma} (\Omega_{1, \Gamma}), \hdots, \nu_{\Gamma} (\Omega_{n, \Gamma})$ 
define an {\it Apery set} of $\Lambda_{\Gamma}$ \cite{Apery:branches}, 
i.e. the classes of $\nu_{\Gamma} (\Omega_{j, \Gamma})$ and  $\nu_{\Gamma} (\Omega_{k, \Gamma})$ 
modulo $n$ are different for all $1 \leq j < k \leq n$;
\item $\Omega_{j, \Gamma} \in \mathbb{C}^{*} \overline{\Omega}_{j} \oplus \oplus_{k=1}^{j-1} \mathbb{C}[[x]] \overline{\Omega}_{k}$ for any 
$1 \leq j \leq n$, where $\overline{\Omega}_{2k+1} = y^{k} dx$ and $\overline{\Omega}_{2k+2} = y^{k} dy$ for any $k \in \mathbb{Z}_{\geq 0}$;
\item $\nu_{\Gamma} (\Omega_{j, \Gamma}) < \nu_{\Gamma} (\Omega_{k, \Gamma})$
for all $1 \leq j < k \leq n$.
\end{itemize}
Moreover, $\Omega_{j, \Gamma}$ depends polynomially on $\Gamma \in U$ 
and $\nu_{\Gamma} (\Omega_{1, \Gamma}), \hdots, \nu_{\Gamma} (\Omega_{n, \Gamma})$  do not depend on $\Gamma \in U$.
 \end{teo} 
 The first two properties determine that $\Omega_{1,\Gamma}, \hdots, \Omega_{n, \Gamma}$ is an {\it Apery basis} \pedro{}\cite{Apery:branches}\pedro{} of $\Gamma$ for $\Gamma \in U$.
 The last two itens provide properties of {\bf minimality} and {\bf monotonicity} of the Apery basis respectively.
 Let us expand on this remark.
 By the properties of an Apery basis, it follows that 
$\oplus_{k=1}^{j} {\mathbb C}[[x]] \Omega_{k, \Gamma}$ is a free $\mathbb{C}[[x]]$-module of rank $j$.
Moreover, we have a filtration 
\[  \mathbb{C}[[x]] \overline{\Omega}_{1} \subset \mathbb{C}[[x]] \overline{\Omega}_{1} \oplus \mathbb{C}[[x]] \overline{\Omega}_{2}  
\subset \hdots \subset \oplus_{k=1}^{j} \mathbb{C}[[x]] \overline{\Omega}_{k} \subset \hdots \]
of $\mathbb{C}[[x]]$-modules of rank $1, 2, \hdots, j, \hdots$ The most economical scenario for the Apery basis with respect
to the filtration happens when 
\[ \Omega_{j, \Gamma} \in \oplus_{k=1}^{j} \mathbb{C}[[x]] \overline{\Omega}_{k} \setminus 
\oplus_{k=1}^{j-1} \mathbb{C}[[x]] \overline{\Omega}_{k} \] 
for any $1 \leq j \leq n$. This property derives from the \emph{minimality} property.
Thus, the filtration gives an order on the $1$-forms of an Apery basis. \emph{Monotonicity} implies that this is the same order that 
is given by the natural valuation. 

Let us remark that the minimality property is related to ``maximal contact"
properties in the spirit of Hironaka \cite{Hironaka} for the resolution of algebraic
singularities and of Aroca, Hironaka and Vicente \cite{aroca2018complex} for the complex
analytic case. Maximal contact varieties are essential for the resolution of singularities
in characteristic $0$. The $1$-forms $\Omega_{1,\Gamma}, \hdots, \Omega_{n, \Gamma}$
have maximal contact in the following sense:
$\nu_{\Gamma} (\Omega_{j,\Gamma})$ is the maximum of the 
$\Gamma$-orders of the $1$-forms in 
$\mathbb{C}^{*} \overline{\Omega}_{j} \oplus \oplus_{k=1}^{j-1} \mathbb{C}[[x]] \overline{\Omega}_{k}$ for any 
$1 \leq j \leq n$ (Proposition \ref{pro:omega_max}).

\strut

There are two other properties that will be key in our construction of an Apery basis for a 
generic $\Gamma \in \mathcal{C}$:
\begin{description}
\item[simplicity] We divide our algorithm of construction of the Apery basis in $g+1$ levels, where $g$ is the genus of $\Gamma_{\circ}$, i.e. its number of Puiseux 
characteristic exponents. We define the \emph{terminal} family of $1$-forms of level $0$ as the set consisting of the
elementary $1$-forms $dx$ and $dy$. Given a \emph{terminal} family of level $l$, the initial family of  level $l+1$, with $0 \leq l < g$, is obtained 
by combining that terminal family of level $l$ with data of topological flavor. %(subsection \ref{subsec:basic_step}). 
Then we describe a straightforward
algorithm that provides a terminal family of level $l+1$. At the end of level $g$ we
obtain the Apery basis provided by Theorem \ref{teo:nice}.

\item[richness] The algorithm of construction of an Apery basis preserves a very rich structure.   
Indeed, we profit from our constructions to  obtain that 
the $\Gamma$-values of all the $1$-forms involved in the construction of the Apery basis
can be calculated in terms of $n, \beta_{1}, \hdots, \beta_{g}$
%of the equisingularity class of  $\Gamma_{\circ}$ 
(Propositions \ref{pro:aux1} and \ref{pro:auxm}).  
\end{description} 
\begin{teo}
\label{teo:algorithm}
Algorithm  \ref{alg:orders} provides the generic semimodule
$\Lambda_{\Gamma}$, for $\Gamma \in U$, in terms of %the multiplicity 
$n$ and the Puiseux 
characteristic exponents $\beta_1,  \hdots, \beta_g$ of the equisingularity class $\mathcal{C}$.
\end{teo}  
Being more precise, 
the initial family of  level $l+1$, with $0 \leq l < g$, is obtained 
by combining the terminal family of level $l$ with the approximate root 
$f_{l+1, \Gamma}$ (cf. \cite{abhyankar-moh}) whose $\Gamma$-order is the generator 
$\overline{\beta}_{l+1}$ of the semigroup $\mathcal{S} (\Gamma)$
(sections \ref{subsec:roots} and \ref{subsec:ll+1}).
The main problem is that the $\Gamma$-orders 
of the $1$-forms in the initial family of level $l+1$ are not pairwise different modulo $n$.
Our 
algorithm is intended to provide a terminal family of level $l+1$ with that property, and the others above.
Philosophically, our method factors the contribution of the 
approximate roots
$f_{l, \Gamma}=0$, with $1 \leq l \leq g$, in the construction of the semimodule, 
one by one. Let us mention that, in a similar spirit, de Abreu and Hernandes, give 
some relations between
the semimodules of approximate roots of a plane branch $\Gamma$
and the semimodule of $\Gamma$
\cite{de_Abreu-Hernandes:semiroots}.

In order to understand the properties of the $1$-forms that appear along the steps
of the algorithm, we consider the contributions of the coefficients of the Puiseux 
expansion of $\Gamma$ one by one, for generic $\Gamma \in \mathcal{C}$. 
A similar idea was considered recently by Casas-Alvero \cite{Casas:class_1_exponent}.
This viewpoint leads us to define two new concepts of independent interest, namely 
the {\it leading variable} (or coefficient) and the \emph{$c$-sequence}
associated to some $\omega \in \hat{\Omega}(\mathbb{C}^2,0)$ (Definition \ref{def:max}). 
They provide sufficient control to calculate the generic semigroup 
$\Lambda_{\Gamma}$, for $\Gamma \in \mathcal{C}$, without calculating the $1$-forms in 
the Apery basis or any other $1$-forms along the process. 
%The existence of an Apery basis with the minimality and monotonicity properties,
%for generic $\Gamma \in \mathcal{C}$, is proved in Theorem \ref{teo:exists_nice}.
Simplicity and richness are natural byproducts of our algorithm 
(Propositions \ref{pro:aux1} and \ref{pro:auxm}). 
%Moreover,  the Apery bases given by Theorem \ref{teo:nice}
%depend polynomially on $\Gamma$ for generic $\Gamma \in \mathcal{C}$.

\subsection{Nice bases and blow-up}Our second aim is to relate the semimodule of $\Gamma$ 
with analogue concepts for 
the strict transform $\tilde{\Gamma}$ of $\Gamma$ by the 
blow-up of the origin. We 
are forced to consider a slightly more general setting and to study
the semimodule $\Lambda_{\Gamma}^{E}$  (cf. Definition \ref{def:semimodule}) 
of orders of K\"{a}hler differentials on $\Gamma$ associated to $1$-forms
that preserve $E$, where $E$ is a normal crossings divisor 
at the origin of $\mathbb{C}^{2}$.
 The construction is similar to the previous one, where $E=\emptyset$, and is studied
in section \ref{sec:E-non-empty}. Theorem \ref{teo:nice-E-basis} is the analogue of 
Theorem \ref{teo:nice} for this setting. Note that Algorithm  \ref{alg:orders}  (cf. Theorem \ref{teo:algorithm})
is generalized to  this case.
Minimality, monotonicity, simplicity and richness are all used to provide a simple 
comparison of the semimodules $\Lambda_{\Gamma}^{E}$ and $\Lambda_{\tilde{\Gamma}}^{\tilde{E}}$
(Proposition \ref{pro:blow-up-contacts}), where $\tilde{E}$ is the total transform of $E$.

%Let $\Gamma$ be a plane branch.
As an application of our construction,
consider the sequence $P_{0} = (0,0) \in \mathbb{C}^{2}, P_1, P_2 \hdots$ 
of infinitely near points of a plane branch $\Gamma$ (Remark \ref{rem:suitable_blow}).
Let $\pi_j$ be the blow-up of the point $P_{j-1}$.
Denote by $\Gamma_j$ the strict transform of $\Gamma$ by 
$\tau_{j}:=\pi_1 \circ \hdots \circ \pi_j$ where $\Gamma_{0} = \Gamma$
and $\tau_0 = \mathrm{id}$. Fix $j \in \mathbb{Z}_{\geq 0}$.
Now, let us consider actions provided by the flow of a vector field
on $\Gamma_j$. Actions of vector fields as a tool for the study of the 
analytic classification problem were introduced
in \cite{Fortuny:flows_moduli}. 
Consider a holomorphic vector field $X$
defined in a neighborhood of $P_j$ such that 
$E_{j} := \tau_{j}^{-1}(0,0)$ is $X$-invariant but $\Gamma_j$ is not.
We obtain a family of curves $\Gamma_{j}^{\epsilon} = \mathrm{exp} (\epsilon X)(\Gamma_j)$
defined for $\epsilon \in \mathbb{C}$ close to $0$. By blow-down by $\tau_j$, 
we get a family ${(\Gamma^{\epsilon})}_{\epsilon \in (\mathbb{C},0)}$ in the 
equisingularity class of $\Gamma$. The Puiseux expansion of $\Gamma^{\epsilon}$
is of the form 
\[
    \Gamma^{\epsilon} (t) =
    \left( t^n, \sum  a_{\beta}(\epsilon) t^{\beta} \right)
\]
where the map $(\epsilon, t) \mapsto \Gamma^{\epsilon} (t)$ is holomorphic. 
Let $\beta$ be the first exponent such that $a_{\beta}(\epsilon)$ is non-constant.
Any $\beta$ obtained this way, by varying $X$, is called a
\emph{flow-fanning exponents} at index $j$, and its set is
denoted by $\Theta_{j} (\Gamma)$ (Definition \ref{def:flow-fanning}).
We profit from this construction to reinterpret the formula for the generic dimension of the moduli of $\Gamma$ by Genzmer.
\begin{teo_o}[\cite{Genzmer-moduli-2020}] \label{teo:dim_gen}
  Let $\Gamma$ be a germ of irreducible curve defined in a neighborhood of
  $0 \in \mathbb{C}^{2}$. The dimension of the generic component of the analytic moduli of
  the equisingularity class of $\Gamma$ is equal to
  \[
    \sum_{0 \leq j < \tau} \sigma (\nu_{P_j} (\Gamma_j) + \delta_j),
  \]
  where $P_0, \hdots, P_{\tau}$ are infinitely near points of $\Gamma$, $\tau$ is the
  first index such that $\Gamma_{\tau}$ is smooth, and $\delta_{i}$ is the number of
  irreducible components of $E_i$ passing through $P_i$ (cf. Definition \ref{def:sigma}).
\end{teo_o}
\begin{teo}
\label{teo:dim_gen_geo}
Let $\Gamma$ be a germ of irreducible curve defined in a neighborhood of
  $0 \in \mathbb{C}^{2}$. The dimension of the generic component of the analytic moduli of
  the equisingularity class of $\Gamma$ is equal to
\[  \sum_{j \geq 1} \sharp (\Theta_{j} (\Gamma) \setminus \Theta_{j-1} (\Gamma)), \] 
and indeed $\sharp (\Theta_{j+1} (\Gamma) \setminus \Theta_{j} (\Gamma)) = \sigma (\nu_{P_j} (\Gamma_j) + \delta_j)$ for any $j \geq 0$,
%by Remark \ref{rem:gen_dim} and Proposition \ref{pro:calc_dim}, 
where we assume that $\Gamma$ is generic in its
equisingularity class $\mathcal{C}$ (otherwise we replace $\Gamma$ with a generic element). 
\end{teo}

We get a geometric interpretation of the generic dimension of the moduli space $\mathcal{C} / \sim_{\mathrm{ana}}$ in which
$\Theta_{j} (\Gamma^{}) \setminus \Theta_{j-1} (\Gamma^{})$ is the set of exponents
of a Puiseux parametrization that can be changed by the action of a vector field defined
in a neighborhood of $P_j$ but can not be changed by the action of a vector field 
defined in a neighborhood of any of the points $P_0, P_1, \hdots, P_{j-1}$.
Thus, in particular we obtain a geometric interpretation of Genzmer's formula.
 It is obtained by introducing 
a correspondence between exponents of the Puiseux expansion of $\Gamma$
and divisors of the desingularization process (Proposition \ref{pro:bij_break_fanning}).

It is possible to compute 
$\Theta_{j} (\Gamma) \setminus \Theta_{j-1} (\Gamma)$ in terms of the semimodules
$\Lambda_{\Gamma_{j-1}}^{E_{j-1}}$ and $\Lambda_{\Gamma_{j}}^{E_{j}}$  (Proposition \ref{pro:calc_dim})
by applying the results on the correspondence between orders of K\"ahler differentials 
on $\Gamma$ and the tangency orders of vector fields with $\Gamma$ described in 
\cite{Fortuny-Ribon-Canadian}. For the case where $\Gamma$ is generic in its equisingularity class, 
our construction of the Apery basis and its properties allow to calculate in parallel
$\Lambda_{\Gamma_{j-1}}^{E_{j-1}}$ and $\Lambda_{\Gamma_{j}}^{E_{j}}$ (Proposition \ref{pro:parallel}),  
to obtain $\sharp  (\Theta_{j} (\Gamma) \setminus \Theta_{j-1} (\Gamma))$ for $j \in \mathbb{Z}_{\geq 1}$ and then show 
that it is the
$j$th-term $\sigma (\nu_{P_{j-1}} (\Gamma_{j-1} \cup \tau_{j}^{-1}(0,0)))$ in Genzmer' formula  
(Proposition \ref{pro:blow-up-contacts}) (cf. Definitions \ref{def:sigma} and  \ref{def:mult}).

\strut

 We introduce the setting of the paper in section  \ref{sec:equisingularity}. 
 The construction of the Apery basis of a generic plane branch in 
 an equisingularity class is carried out in section    \ref{sec:E_empty}. 
 This is generalized for the case of K\"{a}hler differentials preserving 
 a normal crossings divisor $E$  in section \ref{sec:E-non-empty}, 
 where we also provide Algorithm  \ref{alg:orders}  that gives the generic semimodule in an equisingularity class. 
 Section \ref{sec:families} is devoted to introducing the correspondence between exponents of the Puiseux expansion and
 divisors of the desingularization process of a plane branch $\Gamma$ and to use such a setting to introduce the sets of 
 flow-fanning exponents and describe their properties. Section \ref{sec:contacts} gives, 
 as a direct application of our results, an alternative proof of Genzmer's formula.

\section{Equisingularity class of a pair $(\Gamma_{\circ}, E)$} \label{sec:equisingularity}
We consider, in $(\mathbb{C}^2,0)$ with coordinates $(x,y)$, an irreducible germ 
of complex analytic plane curve $\Gamma_{\circ}$ and a normal crossings divisor  $E \subset \{xy=0\}$. 
It is well-known that the equisingularity class $\tilde{\mathcal C}$ of $(\Gamma_{\circ}, E)$ is equivalent to the 
equisingularity class of $\Gamma_{\circ}$ together with  the intersection numbers
$(\Gamma_{\circ}, D)_{(0,0)}$ at the origin for each irreducible component $D$ of $E$
(cf. \cite[Theorem 8.4.21]{brieskorn2012plane}). 
\begin{defi} \label{def:mult}
We consider the notation $\nu_{P}(\Gamma)$ for the multiplicity of a germ of curve $\Gamma$ at $P$.
Sometimes we just write $\nu (\Gamma)$ when $P$ is implicit.
\end{defi}
\subsection{The set of exponents of an equisingularity class}
Define $n$ as the multiplicity $\nu (\Gamma_{\circ})$ of $\Gamma_{\circ}$ at the origin
if $\{x = 0\} \not \subset E$, otherwise we define $n$ as the intersection number $ (\Gamma_{\circ}, x=0)_{(0,0)}$.
We define 
\[ \beta_{0} = \max \{ (\overline{\Gamma}, y=0)_{(0,0)}  \ | \   (\overline{\Gamma}, E) \in \tilde{\mathcal C} \ \mathrm{and} \ (\overline{\Gamma}, x=0)_{(0,0)}=n \} . \]
For any curve $\overline{\Gamma} \in \tilde{\mathcal C}$,
 %in the equisingularity class  of $(\Gamma, E)$, 
 there exists an irreducible curve $\Gamma$ that is analytically
conjugated to $\overline{\Gamma}$ by a map preserving every irreducible component of $E$ and  admits a
parametrization of the form
\begin{equation}
\label{equ:param}
\Gamma (t) = \left( t^{n}, \sum_{\beta \geq \beta_{0}} a_{\beta, \Gamma} t^{\beta} \right)  .
\end{equation}
So, up to a change of coordinates, we suppose that $\Gamma_{\circ}$ has a Puiseux expansion of the form \eqref{equ:param}.
The calculation of the $\mathbb{C}[[x]]$-module of K\"ahler differentials on $\Gamma_{\circ}$ of $1$-forms
preserving $E$ is simple if $\Gamma_{\circ}$ is smooth and transverse to $x=0$.  Indeed, it is
generated by $dx$, $dy$ or $x dy$ depending on whether $\{y=0\} \not \subset E$,
$E=\{y=0\}$ or $E=\{xy=0\}$ respectively.  Thus, we will assume from now on $n >1$. 

 We can associate Puiseux characteristic exponents $\beta_1 < \hdots < \beta_g$ to the pair
$(\Gamma_{\circ}, E)$ by considering the natural numbers $\beta$ such that
\[
  \gcd( \{ n \} \cup \{ \tilde{\beta} : a_{\tilde{\beta}, \Gamma_{\circ}} \neq 0 \  \wedge \ 
  \tilde{\beta} < \beta\}) \neq
  \gcd( \{ n \} \cup \{ \tilde{\beta} : a_{\tilde{\beta}, \Gamma_{\circ}}
  \neq 0 \ \wedge \ \tilde{\beta} \leq \beta \}) .
\]
Indeed if $x=0$ is not the tangent cone of $\Gamma_{\circ}$ these numbers are the usual Puiseux
characteristic exponents of $\Gamma_{\circ}$.  Notice that either $\beta_0 \in (n)$ or
$\beta_0 = \beta_1$ \and that in particular $\beta_{0} = \beta_{1}$ holds if $\beta_{0} < n$.  Moreover, 
the Puiseux exponents of $(\Gamma_{\circ}, E)$ just depend on the pair
$(\Gamma_{\circ}, n)$.  
\begin{defi} 
We define the {\it set of exponents} ${\mathcal E}_{E}(\Gamma_{\circ})$ of
$(\Gamma_{\circ},E)$ as
\[
  {\mathcal E}_{E}(\Gamma_{\circ}) = \cup_{j=0}^{g}
  \{ m \in (n, \beta_0, \beta_1,\ldots,
  \beta_j) : m \geq \beta_j \} .
\] 
\end{defi}
Given any $\Gamma \in \tilde{\mathcal C}$ that admits a parametrization of the form \eqref{equ:param}, we have
$a_{\beta, \Gamma} = 0$ if $\beta \not \in {\mathcal E}_{E}(\Gamma_{\circ})$. Moreover
$a_{\beta_{0}, \Gamma} \hdots a_{\beta_{g}, \Gamma} \neq 0$ holds. 
In particular, the Puiseux exponents of
$(\Gamma,E)$ are 
$\beta_1 < \hdots < \beta_g$ and
${\mathcal E}_{E} (\Gamma) = {\mathcal E}_{E} (\Gamma_{\circ})$. 
In general $E$ is fixed and we just write ${\mathcal E}(\Gamma)$.

We consider   the subset 
\[
  {\mathcal C} = \left\{ \Gamma \ | \ \Gamma (t) = \left( t^{n}, 
      \sum_{\beta \in {\mathcal E}_{E} ( \Gamma_{\circ} )} a_{\beta, \Gamma} t^{\beta} \right) \
    \mathrm{and} \       a_{\beta_0, \Gamma} \hdots a_{\beta_g, \Gamma} \neq 0 
    \right\} 
\] 
of the equisingularity class $\tilde{\mathcal C}$ of $(\Gamma_{\circ}, E)$. 
We say that ${\mathcal C}$ is an {\it arranged equisingularity class} of curves. It
intersects any analytic class contained in $\tilde{\mathcal C}$. We just remove from
$\tilde{\mathcal C}$ the curves $\Gamma$ such that $n \neq (\Gamma, x=0)_{(0,0)}$ or 
$\beta_{0} \neq (\Gamma, y=0)_{(0,0)}$.
%or $\beta_0 \neq (\Gamma, y=0)_{(0,0)}$.  
For simplicity, we drop the term ``arranged" from
now on.
\begin{rem}
The above discussion allows us to define
${\mathcal E}_E({\mathcal C})=\mathcal{E}_E(\Gamma)$ for $\Gamma$ in $\mathcal{C}$, as ${\mathcal E}$ is constant in
${\mathcal C}$. This way, we can consider
${\{ a_{\beta} \}}_{\beta \in {\mathcal E}({\mathcal C}) }$ as variables that parametrize
the equisingularity class ${\mathcal C}$. These parameters are the main tool allowing us to attain our purpose. Thus, from now on, 
and unless explicitly said otherwise, 
we forget the subindex $_0$ and will work with $\Gamma$ generic in its equisingularity class.
\end{rem}
\subsection{The semigroup and the semi-module}
Recall that the main topological invariant associated to
$\Gamma$ is its \emph{semigroup}:
\[
  {\mathcal S} (\Gamma)=\left\{ \nu_{\Gamma} (f) := \ord_t(f(\Gamma(t)))\ :\ f(x,y)\in
    \mathbb{C}[[x,y]] \right\},
\]
whereas it is known since \cite{Hefez-Hernandes-classification} that the main analytic
invariant of $\Gamma$ is the set $\Lambda_{\Gamma}$ of contacts of holomorphic (or formal)
$1$-forms with $\Gamma$. 
\begin{defi} \label{def:semimodule}
We define
\[
 \Lambda_{\Gamma}^{E} = \big\{
 \nu_{\Gamma} (\omega) := \ord_t (\Gamma^{\ast}\omega) + 1 \ |\ \big.
 \big.
 \omega \in \hat{\Omega}_{E}(\mathbb{C}^2,0) 
 \big\} ,
\]
where
\[
  \hat{\Omega}_{E}(\mathbb{C}^2,0) = \left\{ A dx + B dy \ |\ \ A,B \in
    \mathbb{C}[[x,y]] \ \mathrm{and} \ E \ \mathrm{is} \ \omega-\mathrm{invariant}
  \right\}
\]
%so that if $\hat{\Omega}(\mathbb{C}^2,0)$ is the module of $1$-forms in $(\mathbb{C}^2,0)$, then 
%$\Lambda_{\Gamma}=\{ \nu_{\Gamma} (\omega) \ |\ \omega\in \hat{\Omega}(\mathbb{C}^2,0) \}$. 
is the $\mathbb{C}[[x,y]]$-module of formal $1$-forms in $(\mathbb{C}^2,0)$ preserving
$E$. We also define ${\Omega}_{E}(\mathbb{C}^2,0)$ as the subset of
$ \hat{\Omega}_{E}(\mathbb{C}^2,0)$ consisting of germs of analytic $1$-forms in
$(\mathbb{C}^2,0)$, and denote
\[
  \hat{\Omega}(\mathbb{C}^2,0) = \hat{\Omega}_{\emptyset}(\mathbb{C}^2,0), \
  {\Omega}(\mathbb{C}^2,0) = {\Omega}_{\emptyset}(\mathbb{C}^2,0) \ \mathrm{and} \
  \Lambda_{\Gamma} = \Lambda_{\Gamma}^{\emptyset}.
\]
%We call $\Lambda_{\mathcal{C}}^E$ the set $\Lambda_{\Gamma}^{E}$ for generic $(\Gamma, E)$
%in an equisingularity class ${\mathcal C}$.
\end{defi}
We intend to compute
$\Lambda_{\Gamma}^{E}$ for generic $(\Gamma, E)$ in a given equisingularity class
${\mathcal C}$: what we call $\Lambda_{\mathcal{C}}^E$.

The set $\Lambda_{\Gamma}^{E}$ is not a semigroup, but  it is a so called semimodule: 
if $\ell\in \Lambda_{\Gamma}^{E}$ and $s\in \mathcal{S}(\Gamma)$, then obviously
$\ell+s\in \Lambda_{\Gamma}^{E}$ (because $f\omega$ belongs to
$\hat{\Omega}_{E}(\mathbb{C}^2,0)$ if $\omega$ does).  However, we do not need such
generality. We shall only use the fact that if $\ell\in \Lambda_{\Gamma}^{E}$, then
$n+\ell\in \Lambda_{\Gamma}^{E}$.  If we write $[\ell]_m$ for the class modulo $m$ of
$\ell$, then: if $\mathbf{\Omega} = ( \Omega_1,\ldots, \Omega_n )$ is a family of
$1$-forms such that
%\[ \nu_{\Gamma} (\Omega_i) = \min \{ \nu_{\Gamma} (\Omega) \ | \ \Omega \in \hat{\Omega}(\mathbb{C}^2,0) \ \mathrm{and} \ 
%[\nu_{\Gamma} (\Omega_i)]_n = [\nu_{\Gamma} (\Omega)]_n \} \]
$\nu_{\Gamma} (\Omega_i)$ is minimum in
$[\nu_{\Gamma} (\Omega_i)]_n \cap \Lambda_{\Gamma}^{E}$ for $i=1,\ldots,n$ and
$[\nu_{\Gamma} (\Omega_i)]_n\neq [\nu_{\Gamma} (\Omega_j)]_n$ for $i\neq j$, then
certainly $\mathbf{\Omega}$ gives, by push-back to $\Gamma$, a basis of K\"ahler $1$-forms
over $\Gamma$ and preserving $E$, i.e.
\[
  \Gamma^{*} \hat{\Omega}_{E}(\mathbb{C}^2,0)
  = \Gamma^{*} ({\mathbb C}[[x]] \Omega_1 +
  \ldots + {\mathbb C}[[x]] \Omega_n).
\]
This will be, in essence, the aim of our
 construction: providing a simultaneous construction for bases of K\"ahler differentials
 on generic curves of a fixed equisingularity class ${\mathcal C}$ of a pair
 $(\Gamma_{\circ}, E)$.

\subsection{Combinatorial invariants of $(\Gamma, E)$} 
The following are the basic invariants of $(\Gamma, E)$ 
(or, equivalently, its equisingularity class ${\mathcal C}$), following
\cite{Zariski:moduli}.
\begin{defi}  
\label{def:inv}
We define the sequences $(e_j)_{j\geq 1}$, $(n_j)_{j\geq 1}$ and $(\nu_{j})_{j\geq 1}$ as
\begin{enumerate}
\item Let $e_0 =n$ and $e_{m} = \gcd (e_{m-1}, \beta_m)$ if $1 \leq m \leq g$.
\item Then $n_0=1$ and $n_m = e_{m-1}/e_m$ if $1 \leq m \leq g$.
\item Finally, $\nu_m = n_0 \hdots n_m = n/e_m$ for $0 \leq m \leq g$. 
\end{enumerate}
\end{defi} 
\subsection{Approximate roots of $(\Gamma, E)$} \label{subsec:roots}
In order to provide bases of K\"ahler differentials for generic curves in $\mathcal{C}$, 
we shall need to use both the elementary $1$-forms $dx$ and $dy$ and the
approximate roots of $\Gamma$ where $\Gamma \in {\mathcal C}$.

Fix $1 \leq j \leq g+1$. We define 
 \[
 f_j = \prod_{k=0}^{ \nu_{j-1}-1 } (y - \sum_{\beta < \beta_j}e^{\frac{2 \pi i k \beta}{n}}
 a_{\beta} x^{\beta/n}) .
 \]
 Since $f_j$ is invariant by the transformation $x \mapsto e^{\frac{2 \pi i}{n}} x$,
 it belongs to $ \mathbb{C}{[a_{\beta}]}_{\beta \in \mathcal{E}(\mathcal{C}) }[[x]][y]$.
 Given $\Gamma \in \mathcal{C}$, let
 $\Gamma_{<\beta_{j}}$ be the truncation 
\[
\Gamma_{<\beta_{j}} = \bigg( t^n, \sum_{ \beta < \beta_{j}}a_{\beta}t^{\beta} \bigg)
\] 
of $\Gamma$ up to (but not including) $\beta_j$. 
By convention, we denote $\beta_{g+1} = \infty$ and we obtain $\Gamma_{<\beta_{g+1}} = \Gamma$. By construction, 
 $f_{j, \Gamma} \circ \Gamma_{<\beta_{j}} \equiv 0$ holds 
for all $\Gamma \in \mathcal{C}$ and $1 \leq j \leq g+1$.

\begin{rem}
When $n = \nu (\Gamma)$, the polynomials $f_{1,\Gamma}, \hdots, f_{g, \Gamma}$ are
called, since \cite{abhyankar-moh}, the \emph{approximate roots of $\Gamma$}.  In our
context, we call them the \emph{approximate roots of the pair} $(\Gamma,E)$. The value 
$\overline{\beta}_j=\nu_{\Gamma}(f_{j, \Gamma})$ (i.e. $\overline{\beta}_j=\ord_t(f_{j,\Gamma}(\Gamma))$) 
is a topological invariant of $\Gamma$ and thus it does not depend on $\Gamma \in \mathcal{C}$ for $1 \leq j \leq g$.
 It is known that ${\mathcal S} (\Gamma)$ is the semigroup
$\langle n, \overline{\beta}_1,\ldots, \overline{\beta}_g\rangle$ generated by
$n, \overline{\beta}_1,\ldots, \overline{\beta}_g$; moreover, we
have $\overline{\beta}_1 = \beta_1$ and 
\begin{equation}
\label{equ:rec_beta}
\overline{\beta}_j = n_{j-1} \overline{\beta}_{j-1} - \beta_{j-1} + \beta_{j} 
\end{equation}
for $2 \leq j \leq g$ (cf. \cite[Chapter II, Theorem 3.9]{Zariski:moduli}). Notice that
$\overline{\beta}_j$ is the first element of the semigroup of $\Gamma$ that is not a
multiple of $e_{j-1}$. Moreover, we have $\overline{\beta}_j - {\beta}_j \in (e_{j-1})$.
\end{rem}
\begin{rem}\label{rem:semigroup-structure}
The $n$ expressions
\[ k_1 \overline{\beta}_1 + \hdots + k_g \overline{\beta}_g \] with $0 \leq k_1 < n_1$,
$\hdots$, $0 \leq k_g < n_g$ define $n$ different classes modulo $n$. The analogous
property holds if we replace $\overline{\beta}_j$ with $\beta_{j}$ for $1 \leq j \leq g$.
\end{rem}
Thus, using $f_{0, \Gamma},\ldots, f_{g, \Gamma}$, where $f_{0, \Gamma} =x$, we can compute the whole semigroup
$\mathcal{S}(\Gamma)$ as the set
 
\begin{equation}
\label{equ:semigroup}
  \{ \nu_{\Gamma}(f_{0, \Gamma}^{k_0}f_{1, \Gamma}^{k_1}\cdots f_{g, \Gamma}^{k_g}) :
  0 \leq k_0 ,k_1, \ldots, k_g \} .
\end{equation}

Paralleling the structure of the semigroup $\mathcal{S}(\Gamma)$,
which is generated by the $\Gamma$-orders of the $g$ approximate roots of
$\Gamma$, we shall divide the construction of a basis of K\"ahler
$1$-forms $\mathbf{\Omega}$ on $\Gamma$ into $g$ levels, one for each
Puiseux characteristic exponent, each level $l$ giving rise to a family
$\mathcal{T}_{l}$ of $1$-forms, so that
$\mathcal{T}_l\subset \mathcal{T}_{l+1}$ and 
$\mathbf{\Omega} \subset \mathcal{T}_g$. 
%\begin{defi}
% Given $l\in \{ 0,\ldots, g\}$, we shall denote by $\tilde{\nu}_l$ the number
% $2\nu_{l}$ if $l<g$, and $\tilde{\nu}_g=\nu_g=n$.
%\end{defi}
The value $2 {\nu}_l$ will be the cardinal of the set of $1$-forms of each level $l$ in
$\mathbf{\Omega}$. 
\subsection{Parametrized $1$-forms}
As we are going to use the approximate roots, which are of the form
$y^j + \cdots$, where the dots have exponents in $y$ of order less than $j$, we will need
to study a distinguished family of
$ \mathbb{C}{[a_{\beta}]}_{\beta \in \mathcal{E}(\mathcal{C}) } [[x]]$-modules.
\begin{defi} \label{def:eval_gamma}
Given $D \in  {{\mathbb C}[a_{\beta}]}_{ \beta  \in {\mathcal E} ({\mathcal C}) }$, we denote by 
$D(\Gamma)$ the value $D ({(a_{ \beta})}_{ \beta \in {\mathcal E} ({\mathcal C})})$.
Given $g = \sum D_{r} t^{r}$, where $D_{r} \in  {{\mathbb C}[a_{\beta}]}_{ \beta  \in {\mathcal E} ({\mathcal C}) }$
for any $r$, we define $g_{\Gamma} = \sum D_{r}(\Gamma) t^{r}$.
\end{defi}
\begin{defi}
\label{def:form_space}
We set: 
\[
  \hat{\Omega}_{\mathcal C}(\mathbb{C}^2,0) =
  \mathbb{C}{[a_{\beta}]}_{\beta \in \mathcal{E}(\mathcal{C}) } [[x,y]] dx +
  \mathbb{C}{[a_{\beta}]}_{\beta \in \mathcal{E}(\mathcal{C}) } [[x,y]] dy .
\] 
Given $\Omega \in \hat{\Omega}_{\mathcal C}(\mathbb{C}^2,0)$ and $\Gamma \in \mathcal{C}$
corresponding to the family ${(a_{\beta})}_{\beta \in \mathcal{E}(\mathcal{C})}$ of
complex coefficients, we define the form $\Omega_{\Gamma} \in \hat{\Omega} \cn{2}$
obtained by evaluating $\Omega$ in ${(a_{\beta, \Gamma})}_{\beta \in \mathcal{E}(\mathcal{C})}$
analogously as in Definition \ref{def:eval_gamma}.
We consider the set ${\Omega}_{\mathcal C}(\mathbb{C}^2,0)$ consisting of
$\Omega \in \hat{\Omega}_{\mathcal C}(\mathbb{C}^2,0) $ such that
$\Omega_{\Gamma} \in {\Omega} \cn{2}$ for any $\Gamma \in \mathcal{C}$, and denote
$\Omega_{E,\mathcal{C}}(\mathbb{C}^{2},0) = \Omega_{E}(\mathbb{C}^{2},0) \cap
\Omega_{\mathcal{C}}(\mathbb{C}^{2},0)$.
\end{defi}
So, we can consider the elements of $\hat{\Omega}_{\mathcal C}(\mathbb{C}^2,0)$ 
as formal $1$-forms parametrized by ${\mathcal C}$. We focus on the properties of 
$\Omega \in \hat{\Omega}_{\mathcal C}(\mathbb{C}^2,0)$ for generic  $\Gamma \in \mathcal{C}$. 
\begin{defi}
\label{def:class}
Given $\Omega \in \hat{\Omega}_{\mathcal C}(\mathbb{C}^2,0)$, define
\[
  \nu (\Omega) = \nu (\Omega_{\Gamma}), \ \ \nu_{\mathcal C} (\Omega) = \nu (t
  \Gamma^{*} \Omega_{\Gamma}) \ \ \mathrm{and} \ \ \Lambda_{\mathcal C}^{E} =
  \Lambda_{\Gamma}^{E}
\]
for a generic $\Gamma \in \mathcal{C}$. 
\end{defi}
The genericity
condition will become  clearer  when constructing the families $\mathcal{T}_l$.

\section{Case $E=\emptyset$} \label{sec:E_empty}
We divide our study into two cases: $E=\emptyset$, $E\neq \emptyset$. The former will in fact contain all the important definitions and arguments, whereas the latter will just consist in a generalization in which everything except the bootstrapping process carries over without modification.

The case $E=\emptyset$ is the foundational one. For the sake of simplicity, and because we are essentially doing the ground work, 
we omit $E$ everywhere in this section, as it is irrelevant.
\subsection{Distinguished modules of $1$-forms}
We shall make frequent use of the following $1$-forms and modules. 
\begin{defi}
\label{def:modules}
We define
\[ {\mathcal M}_{2 \nu_{l}} = \left(\bigoplus_{j=0}^{\nu_{l}-1}  \mathbb{C}[[x]] y^{j} dx\right)
  \bigoplus   \left(\bigoplus_{j=0}^{\nu_{l}-1}  \mathbb{C}[[x]] y^{j} dy\right)
\]
for $0 \leq l \leq g$. 
\end{defi}
Let
\[
  \overline{\Omega}_{2j+1} = y^{j} dx \quad
  \mathrm{and} \quad \overline{\Omega}_{2j+2} = y^{j} dy
\]
for $j \geq 0$. Thus, $(\overline{\Omega}_{1}, \ldots, \overline{\Omega}_{2 \nu_{l}})$ is a basis of the free 
${\mathbb C}[[x]]$-module ${\mathcal M}_{2 \nu_{l}}$. 
\begin{defi}
\label{def:ideal_g}
We define ${\mathcal I}_{\Gamma}$ as the set of $1$-forms $\Omega \in \hat{\Omega} (\mathbb{C}^2,0)$
such that $\Gamma^{*} \Omega \equiv 0$.  
\end{defi}
Now, let us focus on parametrized forms. 
Setting $M_0=\{ 0 \}$, we define, recursively,
\[\label{def:Mj}
  M_{j +1} = \mathbb{C}{[a_{\beta}]}_{\beta \in \mathcal{E}(\mathcal{C}) } [[x]]
  \overline{\Omega}_{j+1} +
  M_{j}, \ \ \hat{M}_{j +1} =
  ( \mathbb{C}{[a_{\beta}]}_{\beta \in \mathcal{E}(\mathcal{C}) } \setminus \{0\})
  \overline{\Omega}_{j+1} +M_{j},
\]
for $j \geq 0$. Clearly $\hat{M}_j \subset M_j$ for $j \geq 1$.  

The subspaces $M_j$ are relevant for our computations because they satisfy the following important properties:
\begin{lem}\label{lem:properties-Wj}
 Let 
 $P = \sum_{j=0}^{k} \alpha_{j} y^{j} \in
 \mathbb{C}{[a_{\beta}]}_{\beta \in \mathcal{E}(\mathcal{C})
 }[[x]][y]$ where
 $\alpha_j \in \mathbb{C}{[a_{\beta}]}_{\beta \in \mathcal{E}(\mathcal{C})
 }[[x]]$ for any $0 \leq j \leq k$. Then:
 \[
 P M_j \subset M_{j+2k} \ \ \mathrm{for \ any} \ j \geq 1 .
 \]
 Moreover, if $\alpha_k \in
 \mathbb{C}{[a_{\beta}]}_{\beta \in \mathcal{E}(\mathcal{C})} \setminus
 \{0\} $, we also have 
 \[
 P \hat{M}_j \subset \hat{M}_{j+2k} \ \ \mathrm{for \ any} \ j \geq 1 .
 \]
\end{lem}
As a matter of fact, the most important spaces are the ones corresponding to the indices $2 {\nu}_j$:
\begin{defi}\label{def:form-level-j}
 A $1$-form $\omega\in \hat{\Omega}_{\mathcal C}(\mathbb{C}^2,0)$ is \emph{of level $l$} if $\omega\in M_{2 {\nu}_l}$ for some $l\in \{0,\ldots, g\}$.
\end{defi}
All the $1$-forms we shall construct will be of some level $l$, as the reader will see. Certainly, $dx$ and $dy$ are both of level $0$.
\subsection{Ordered families}
Our aim is to find \emph{nice bases} for generic
curves $\Gamma$ in a given equisingularity class ${\mathcal C}$.
\begin{defi}
\label{def:nice}
Consider a family $\mathbf{\Omega} = (\Omega_1, \hdots, \Omega_{r} )$ in ${\Omega}_{\mathcal C}(\mathbb{C}^2,0)$ where $0 \leq l \leq g$ and 
 $r \in \{ 2 \nu_{l}, \min (2 \nu_{l}, n) \}$.
%We define $\tilde{\bf \Omega}=(\Omega_1, \hdots, \Omega_{r} )$ where $r=\min (n, 2 \nu_{l})$.
We say that $\mathbf{\Omega}$
is an {\it \tidy{} family of level} $l$ for $\mathcal{C}$ % , where $r \in \{ 2 \nu_{l}, \min (2 \nu_{l}, n) \}$, if
if
\begin{enumerate}[label=(\roman*)]
\item \label{cond1} $\Omega_j \in \hat{M}_j$ for any $1 \leq j \leq r$;
\item \label{cond3} $ \nu_{\mathcal{C}} (\Omega_{j}) \leq  \nu_{\mathcal{C}} (\Omega_{k})$ 
for all $1 \leq j \leq k \leq r$.
\end{enumerate}
We say that $\mathbf{\Omega}$ is a {\it\nice{} family of level} $l$ if it is  \tidy{} and
\begin{enumerate}[label=(\roman*)]
\setcounter{enumi}{2}
\item \label{cond4}
  $\nu_{\mathcal{C}} (\Omega_1), \ldots, \nu_{\mathcal{C}} (\Omega_{\min (2 \nu_{l},n)})$
  define pairwise different classes modulo $n$.
\end{enumerate}
For $l=g$ and $r=n$, we say that
$\mathbf{\Omega}= (\Omega_1, \hdots, \Omega_{n} )$, with
$\Omega_i\in {\Omega}_{\mathcal C}(\mathbb{C}^2,0)$, is a {\it nice basis} if it is a
\nice{} family and
\begin{enumerate}[label=(\roman*)]
\setcounter{enumi}{3}
\item \label{cond5}
  $\hat{\Omega}(\mathbb{C}^2,0)=\mathbb{C}[[x]] \Omega_{1,\Gamma} \oplus \hdots \oplus
  \mathbb{C}[[x]] \Omega_{n,\Gamma} \oplus {\mathcal I}_{\Gamma}$ for generic
  $\Gamma \in \mathcal{C}$.
\end{enumerate} 
\end{defi}
%\begin{rem}
%We have $\bf{\Omega} = \tilde{\bf \Omega}$ if $l < g$ and $\tilde{\bf \Omega}$ consists of the first half
%of elements of ${\bf \Omega}$ if $l=g$.
%\end{rem}
\begin{rem}
\label{rem:nice_suff}
Condition \ref{cond1} implies 
\begin{equation}
\label{equ:submodules}    
 {\mathcal M}_{2 \nu_{l}}   =\mathbb{C}[[x]] \Omega_{1,\Gamma} + \hdots +
  \mathbb{C}[[x]] \Omega_{2 \nu_{l},\Gamma} 
\end{equation}
 for generic $\Gamma \in \mathcal{C}$ if $r= 2 \nu_{l}$.  
 Thus, ${\bf \Omega}$ is a basis of the ${\mathbb C}[[x]]$-module ${\mathcal M}_{2 \nu_{l}}$.
 Moreover, it guarantees that the minimality property (cf. section \ref{sec:intro})
 is satisfied for generic $\Gamma \in \mathcal{C}$.
 A nice basis ${\bf \Omega}$ induces an Apery basis of $\Gamma$ (cf. section \ref{sec:intro}) 
 for generic $\Gamma \in \mathcal{C}$
 by Conditions  \ref{cond4} and  \ref{cond5}. 
 Finally, a nice basis satisfies the monotonicity condition (cf. section \ref{sec:intro}) 
 for generic $\Gamma \in \mathcal{C}$ by Conditions \ref{cond3} and  \ref{cond4}. 
\end{rem}
 At level $0$,  the first condition
will be accomplished by considering the 1-forms $dx$ and $d f_1$, that constitute a
\nice{} family. Then we describe a process in which $dx$ and $d f_1$ appear
alternatively and the approximate roots of $\Gamma$, with $\Gamma \in \mathcal{C}$,
play a prominent role. We obtain $f_1 = y$, since $\beta_0 = \beta_1$ for the case $E = \emptyset$. 
Anyway, we have $f_1 \neq y$ if $\beta_0 < \beta_1$. So, in order to have uniform notations for the cases 
$E= \emptyset$ and $E\neq \emptyset$, we prefer to use $f_1$ in all cases.

\subsection{The basic step --- leading variables} \label{subsec:basic_step}
Fix an equisingularity class ${\mathcal C}$ of germs of branches. Our construction
essentially follows this process: assume we have found a \emph{terminal} family (Definition \ref{def:level})
$\mathcal{T}_{l}= (\Omega_1,\ldots, \Omega_{2 {\nu}_l} )$ in  ${\Omega}_{\mathcal C}(\mathbb{C}^2,0)$
of level $l < g$. Such family is \nice{} but we do not delve now into its properties,
which will be introduced in their definition. For now, a terminal family of level $l$ can be considered just as
the product of our algorithm at the end of level $l$. 
We define $\mathcal{T}_0= (dx, d f_1)$ as the terminal family of level $0$.
We are going to construct, iteratively, a terminal family $\mathcal{T}_{l+1}$ of level $l+1$
 in ${\Omega}_{\mathcal{C}}(\mathbb{C}^2,0)$, 
with $2 {\nu}_{l+1}$ elements. In the end, we shall have $\mathcal{T}_g$ with $2 n$ elements 
whose first $n$ elements constitute a nice basis for $\mathcal{C}$. The construction of $\mathcal{T}_{l+1}$
requires several stages. The first one (stage $0$) is
elementary: we combine the $1$-forms in $\mathcal{T}_l$ with the approximate root $f_{l+1}$.
More precisely, we define 
\begin{equation}
\label{equ:stage}
 \Omega_{2 a \nu_{l}+j}^{0} = f_{l+1}^{a} \Omega_{j}
\end{equation}
for all  $0 \leq a < n_{l+1}$ and $1 \leq j \leq 2 \nu_{l}$
(recall that $f_{l+1, \Gamma}$ is the $l+1$-th approximate root of $\Gamma$). 
Notice the aforementioned alternate role of $dx$ and $d f_1$ since $\Omega_k^{0} \in \hat{M}_k$ for $k \geq 1$. 
\begin{defi}
 \label{def:set-stage0} 
 Fix a terminal family $\mathcal{T}_l= (\Omega_1, \hdots, \Omega_{2 {\nu}_{l}})$ . We say that 
 \begin{equation*}
        \mathcal{G}_{l+1}^0 = (\Omega^{0}_{2a\nu_l+j})_{0 \leq a < n_{l+1}, \ 1 \leq j \leq 2 {\nu}_{l} } =
 (\Omega^{0}_{k})_{1 \leq k \leq 2 {\nu}_{l+1}}
 \end{equation*}
 is a \emph{the initial family (of stage $0$) of level $l+1$}.  
 \end{defi}
  These $1$-forms belong to
${\Omega}_{\mathcal C}(\mathbb{C}^2,0)$, see equation (\ref{equ:pol_root}), and are $2 {\nu}_{l+1}$ in number. They will be called the
\emph{stage} 0 forms of level~$l+1$.
\begin{rem}
\label{rem:det_ord_in_fam}
Since $\nu_{\mathcal{C}} (\Omega_{2 a \nu_{l}+j}^{0} ) = a \overline{\beta}_{l+1} + \nu_{\mathcal{C}} (\Omega_{j})$, 
the $\mathcal{C}$-orders of elements of $\mathcal{T}_l$ determine the $\mathcal{C}$-orders of the elements of 
$\mathcal{G}_{l+1}^0$.
\end{rem}
\begin{defi}
\label{def:modulo}
Given $l \in {\mathbb Z}$ and $m \in {\mathbb Z}_{\geq 1}$, we denote by $[l]_{m}$ the class of $l$ modulo $m$. Moreover, given 
$S \subset {\mathbb Z}$, we define $[S]_{m} = \{ [l]_{m}  \ |\ \ l \in S \}$.
\end{defi}
Unfortunately, the stage 0 forms of level $l+1$ are ordered but usually not \nice{}
(condition \ref{cond4} in Definition \ref{def:nice}). 
We will construct families $\mathcal{G}_{l+1}^{0}, \mathcal{G}_{l+1}^{1}, \ldots$ 
in ${\Omega}_{\mathcal C}(\mathbb{C}^2,0)$
until we find a (\nice{}) family 
$\mathcal{G}_{l+1}^{r}$, which will be the terminal one.
\begin{defi}
\label{def:next}
Assume a family $\mathcal{G}_{l+1}^s=(\Omega_1^{s}, \ldots, \Omega_{2 {\nu}_{l+1}}^s)$ is
given, with $0\leq l < g$ and $s\geq 0$ (i.e. a family of level $l+1$ at some stage).
Assume that $\mathcal{G}_{l+1}^s$ is non-\nice{}.
 When this happens, we
consider the pairs $j,k$ with $1 \leq j < k \leq 2 {\nu}_{l+1}$ such that
$[\nu_{\mathcal{C}} (\Omega^{s}_j)]_n=[\nu_{\mathcal{C}} (\Omega^{s}_k)]_n$: whenever this
equality holds, there is some $d\in \mathbb{Z}_{\geq 0}$ with
$\nu_{\mathcal{C}} (\Omega^s_k)=\nu_{\mathcal{C}} (\Omega_{j}^s) + dn$.  Consider the
pull-backs 
\begin{equation*}
 \Gamma^{\ast}\Omega_{k,\Gamma}^s = (C_{k,\eta_{k}} (\Gamma)t^{\eta_k}+ \cdots) dt ,\
 \Gamma^{\ast}\Omega^s_{j,\Gamma} = (C_{j,\eta_{j}} (\Gamma) t^{\eta_{j}}+\cdots) dt
\end{equation*}
of $\Omega_k^s$ and $\Omega_j^s$, where 
$C_{k,\eta_k},C_{j,\eta_j}\in {{\mathbb C}[a_{\beta}]}_{ \beta \in {\mathcal E} ({\mathcal C}) } \setminus \{0\}$, 
for all curves in an equisingularity class. For any such a class, $C_{k,\eta_k}$ and
$C_{j,\eta_j}$ are 
%complex numbers that are 
both generically non-vanishing.
Then we replace $\Omega^s_k$ with
\begin{equation}\label{eq:main-substitution}
 \Omega^{s+1}_k := C_{j,\eta_j}\Omega^s_k - C_{k,\eta_k} x^{d} \Omega^s_{j} ; 
\end{equation}
Setting also $\Omega^{s+1}_j=\Omega^{s}_j$, for the other
$1$-form, we obtain the $2 {\nu}_{l+1}$ $1$-forms   in
${\Omega}_{\mathcal C}(\mathbb{C}^2,0)$ of stage $s+1$ (hence the exponent in
$\Omega^{s+1}_j$), in the same level $l+1$. 
\end{defi}
\begin{rem}
The families $ \mathcal{G}_{l+1}^{1}, \mathcal{G}_{l+1}^{2}, \ldots$ 
are well-defined (Proposition \ref{pro:well_def_fam}).
There are some points in the definition that need to be proved, for instance that $d$ is non-negative. 
Arithmetic arguments will show how
this construction can be iterated and at some stage it
gives a \nice{} family $\mathcal{G}_{l+1}^{r}$ of $1$-forms of level $l+1$.
\end{rem}
\begin{rem}
We have 
$\nu_{\mathcal{C}} (\Omega^{s+1}_k) >\nu_{ \mathcal{C}} (\Omega^{s}_k)$ and 
$[\nu_{ \mathcal{C}} (\Omega^{s+1}_k)]_n\neq [\nu_{ \mathcal{C}} (\Omega^{s}_k)]_n$ 
as we shall see (Propositions \ref{pro:aux1} and \ref{pro:auxm}). 
\end{rem}
%satisfying the first four conditions in Definition \ref{def:nice}.
\begin{rem}
The algorithm described above to obtain a \nice{} family  $\mathcal{G}_{l+1}^{r}$ was described as  {\bf simple} in 
section \ref{sec:intro}. One reason is that $\mathcal{G}_{l+1}^{0}$ is built from $\mathcal{T}_{l}$ and 
data related to the equisingularity class, namely the approximate root $f_{l+1}$. 
Another reason is that the algorithm consists in multiple applications of the elementary substitution 
\eqref{eq:main-substitution}. Indeed, the simplicity and expressiveness of Equation
\eqref{eq:main-substitution} will be essential to control the properties of our 
construction.
\end{rem} 
 In order to control $\nu_{\mathcal{C}} (\Omega^{s+1}_k)$ in substitution \eqref{eq:main-substitution}, 
we require the following %complicated 
definition, which
will allow us to explicitly compute that contact order. For the sake of simplicity, if
$\beta\in \mathcal{E}({\mathcal C})$, then ${\mathfrak p}(\beta)$ is the previous element in
$\mathcal{E}({\mathcal C})$ and ${\mathfrak n} (\beta)$ is the following one.   
 \begin{defi}
 \label{def:max}
 We say that 
 %a power series 
 $g(t)\in \mathbb{C}[a_{\beta_0},a_{{\mathfrak n} (\beta_0)},a_{{\mathfrak n}^{2} (\beta_0)},\ldots][[t]]$
 has \emph{leading variable $\beta$ (or $a_{\beta}$) of degree
 $m\in \mathbb{Z}_{\geq 0}$}, or equivalently, that $g(t)$ \emph{starts at $\beta$
 with degree $m$} if $\beta \in \mathcal{E} ({\mathcal C})$ and
\begin{itemize}
\item There exists $k\in \mathbb{Z}_{\geq 0}$ with
  \[
    g(t)= t^k \sum_{r=0}^{\infty} C_{{\mathfrak n}^{r}(\beta)} t^{{\mathfrak n}^{r}(\beta)}
  \]
  where
  $C_{{\mathfrak n}^r(\beta)} \in {{\mathbb C}[a_{\beta^{\prime}}]}_{ \beta^{\prime} \in {\mathcal E} ({\mathcal C}) }$
  for $r \geq 0$;
\item There exist $m\in \mathbb{Z}_{\geq 0}$, a polynomial
  $ R \in \mathbb{C}{[a_{\beta'}]}_{\beta_0 \leq \beta' < \beta} \setminus \{0\}$, a unit
  $\tilde{c} \in {\mathbb C}^{*}$ and a sequence
  $c_{\beta}=1, c_{{\mathfrak n} (\beta)}, c_{{\mathfrak n}^{2} (\beta)}, \hdots$ in ${\mathbb C}$ such that
\begin{equation}
\label{equ:maximal}
C_{{\mathfrak n}^{r}(\beta)} = \tilde{c} c_{{\mathfrak n}^{r}(\beta)} a_{{\mathfrak n}^{r}(\beta)} R(a_{\beta_0}, \hdots,
a_{{\mathfrak p}(\beta)}) a_{\beta}^{m} + Q_{{\mathfrak n}^{r}(\beta)}, 
\end{equation}
for all $r \geq 0$, where
$Q_{{\mathfrak n}^{r}(\beta)} \in \mathbb{C}{[a_{\beta'}]}_{\beta_0 \leq \beta' < {\mathfrak n}^{r}(\beta)}$.
\item If $\beta\not\in \left\{ \beta_1,\ldots,\beta_g \right\}$ (i.e. it is not a Puiseux
  characteristic exponent) and $\ell = \max \{ l \geq 0 : \beta_l \leq \beta\}$, then
  $m=0$ and $[k]_{e_\ell} = [0]_{e_\ell} $.
%(in our notation, $\beta_{\ell}<\beta<\beta_{\ell+1}$);
\item If $\beta=\beta_{\ell+1}$ (i.e. it is a Puiseux characteristic exponent), then
  $[k]_{e_\ell} = [m \beta]_{e_\ell}$.
%Moreover $Q_{{\mathfrak n}^{r}(\beta)} \equiv 0$ whenever 
%${\mathfrak n}^{r} (\beta) \in \{ \beta_0, \beta_1, \hdots, \beta_g \}$.
\end{itemize}
The sequence $(c_{{\mathfrak n}^{r}(\beta)})_{r \geq 0}$ will
be called the \emph{$c$-sequence of $g$}. 
\end{defi}
Notice that not all power series have a leading variable.
\begin{rem}
\label{exa:leading}
Consider the equisingularity class whose arranged representatives are of the form
$\Gamma (t) = (t^{6} , a_{9} t^{9} + a_{10} t^{10} + \hdots)$ where $a_{9} a_{10} \neq 0$. Its 
Puiseux characteristic exponents are $\beta_{1} = 9$ and $\beta_{2}=10$.
Let $f_1 = y$ and 
\[ f_2 = (y - a_{9} x^{3/2}) (y + a_{9} x^{3/2}) = y^{2} - a_{9}^{2} x^{3} \]
be its approximate roots.  
Moreover,  we have that 
\[ 
 t \frac{\partial (f_{1, \Gamma} \circ \Gamma (t))}{\partial t} = 
 9 a_{9, \Gamma} t^{9} + 10 a_{10, \Gamma} t^{10} +  11 a_{11, \Gamma} t^{11}  +\hdots \]
%\[ \Gamma^{*} d f_1 =
 %\frac{\partial (f_1 \circ \Gamma (t))}{\partial t} dt = \frac{9 a_{9} t^{9} + 10 a_{10} t^{10} +  11 a_{11} t^{11}  +\hdots}{t} dt \]
and 
\[  f_{2}^{2} \circ \Gamma (t) = {\left( a_{10} t^{10} + a_{11} t^{11} + \hdots \right)}^{2}  {\left( 2 a_{9} t^{9} + a_{10} t^{10} + a_{11} t^{11} + \hdots \right)}^{2}  \]
 belong to $\mathbb{C}[a_{9},a_{10},\ldots][[t]]$, where we skip the subindexes $\Gamma$ in the second expression.
 In the former case, $t \frac{\partial (f_1 \circ \Gamma (t))}{\partial t}$ has leading variable $a_{9}$ of degree $0$ and we have 
 $k=0$, $e_{0} =6$, $\tilde{c} = 9$, $R \equiv 1$, $Q_{9} \equiv Q_{10} \equiv \hdots \equiv 0$, $c_{9} =1$, $c_{10}=10/9$, $c_{11} = 11/9$ and so on. 
 
 In the latter case,  one can verify that $f_2^{2} \circ \Gamma (t)$ has leading variable $a_{10}$ of degree $1$. More precisely, we have
$k=28$, $e_{1} =3$, $\tilde{c}=1$, $R = 4 a_{9}^{2}$, $Q_{10} \equiv 0$, $Q_{11} = 4 a_{9} a_{10}^{3}$, 
$c_{10} =1$, $c_{11}=c_{12} = \hdots =2$.
\end{rem} 

\begin{rem}
 The nonzero polynomial $R(a_{\beta_{0}},\ldots, a_{{\mathfrak p}(\beta)})$ is the same for all the
 coefficients $C_{\beta}$ (this is what makes the definition of the $c-$sequence
 possible). Also, despite saying that $a_{\beta}$ is of ``degree $m$'', the first element
 $C_{\beta}$ is
 \begin{equation*}
 C_{\beta} = \tilde{c}a_{\beta}R(a_{\beta_0},\ldots,a_{{\mathfrak p}(\beta)})a_{\beta}^m + Q_{\beta}.
 \end{equation*}
 But the terminology reflects the fact that $a_{\beta}$ is the variable of maximal index in the
 coefficient of the least order (the ``leader''), and if $m>0$, then it subsequently
 appears explicitly \emph{always to the power $m$} in $C_{{\mathfrak n}^r(\beta)}$, for any $r>
 0$. If $\beta$ is not a Puiseux characteristic exponent, then $m=0$.  
 %it only appears in $C_{\beta}$.
\end{rem}

As in Remark \ref{exa:leading},  the  
notion of leading variable can be extended to power series
$f\in \mathbb{C}{[a_{\beta}]}_{\beta \in \mathcal{E}(\mathcal{C})} [[x,y]]$ just by
considering $g(t) = f \circ \Gamma (t)$, where $\Gamma \in \mathcal{C}$, since $g$ can be
interpreted as an element of
$\mathbb{C} {[a_{\beta}]}_{\beta \in \mathcal{E} (\mathcal{C})} [[t]]$.  Analogously, it
can be extended to $1$-forms $\Omega\in \hat{\Omega}_{\mathcal{C}}(\mathbb{C}^2,0)$ by
considering the power series $g(t)$, where $\Gamma \in \mathcal{C}$, satisfying
$t \Gamma^{*} \Omega_{\Gamma} = g(t) dt$.
 
\begin{rem} \label{rem:rich}
The use of leading variables will enable us to control $\nu_{\mathcal{C}} (\Omega^{s+1}_k)$
in substitution \eqref{eq:main-substitution}. As
a matter of fact, we will prove, in Propositions \ref{pro:aux1} and \ref{pro:auxm} that
\[ \nu_{\mathcal{C}} (\Omega_{k}^{s+1}) = \nu_{\mathcal{C}} (\Omega_{k}^{s}) + 
( {\mathfrak n}^{s+1}(\beta_{l+1}) - {\mathfrak n}^{s} (\beta_{l+1})) . \]
This property, together with Remark \ref{rem:det_ord_in_fam}, determine the $\mathcal{C}$-orders of 
all the $1$-forms appearing in the construction of a nice basis of $\mathcal{C}$. 
We obtain Algorithm \ref{alg:orders} that calculates such $\mathcal{C}$-orders without calculating the corresponding
$1$-forms. This is the property of {\bf rich}  structure of the construction mentioned in section \ref{sec:intro}.
%but unfortunately, this key equality requires the considerable technical apparatus which follows.
\end{rem}

\begin{defi}
The set of power series $g(t)$ starting at $\beta$ with degree $m$ will be denoted $\Delta^m_{\beta}$. We shall also
make use of
\begin{itemize}
\item The set $\Delta_{\beta,\leq}^{m}$ of power series $g(t)\in\Delta^m_{\beta}$
 whose $c$-sequence is increasing in ${\mathbb Q}^{+}$;
\item The set $\Delta^m_{\beta,<}$ of power series $g(t)\in \Delta^m_{\beta}$ whose
 $c$-sequence is strictly increasing in ${\mathbb Q}^{+}$. 
%\item The set $\overline{\Delta}^{m}_{\beta_l,\leq}$ of power series
% $g(t)\in\Delta^m_{\beta_l,\leq}$ (only for Puiseux characteristic exponents) such that
% $Q_{\beta_{l}}\equiv 0$.
\end{itemize}
Given $1 \leq j \leq g$, we define the set $\tilde{\Delta}_{\beta_j}$ consisting of $g \in \Delta_{\beta_j}$
such that $Q_{\beta_{j}} \equiv 0$. We denote 
$\tilde{\Delta}_{\beta_j,\leq}^{m} = {\Delta}_{\beta_j,\leq}^{m}  \cap \tilde{\Delta}_{\beta_j}$ and
$\tilde{\Delta}_{\beta_j,<}^{m} = {\Delta}_{\beta_j,<}^{m}  \cap \tilde{\Delta}_{\beta_j}$.
For each $\beta \in \mathcal{E} (\mathcal{C})$, the union for all degrees $m$ of each of those sets is also
important. Thus, we denote 
\[ \Delta_{\beta} = \cup_{m \geq 0} \Delta_{\beta}^{m}, \
\Delta_{\beta,\leq} = \cup_{m \geq 0} \Delta_{\beta,\leq}^{m} \ \mathrm{and} \ 
\Delta_{\beta, <} = \cup_{m \geq 0} \Delta_{\beta, <}^{m} . \]
We also denote $\tilde{\Delta}_{\beta_j,\leq} = \cup_{m \geq 0}  \tilde{\Delta}_{\beta_j,\leq}^{m}$ and  
$\tilde{\Delta}_{\beta_j,<} = \cup_{m \geq 0}  \tilde{\Delta}_{\beta_j,<}^{m}$.
\end{defi}
%We set $\Delta_{0,*} = \{ \lambda  t^{n} : \lambda \in \mathbb{C}^{*} \}$.

\begin{rem}
Clearly, $\Delta_{\beta} \cap \Delta_{\beta'} = \emptyset$ if $\beta \neq \beta'$ so that the leading variable of $g \in \cup \Delta_{\beta}$ is unique.
\end{rem} 
\begin{rem}
\label{exa2:leading}
The calculations in Remark \ref{exa:leading} imply  $d f_1 \in  \Delta_{9,<}^{0}$ and 
$f_{2}^{2} \in \Delta_{10,\leq}^{1} \setminus \Delta_{10,<}^{1}$.
\end{rem}

 We shall also use the same notation
$\Delta^{m}_{\beta}$, etc. for the corresponding subsets of $\mathbb{C}{[a_{\beta}]}_{\beta \in \mathcal{E}(\mathcal{C})} [[x,y]]$ and 
$\hat{\Omega}_{\mathcal{C}}(\mathbb{C}^2,0)$ because there is no possibility of confusion. 
For completeness, we add $dx$ to the set $\Delta_{n,<}$ in the case of $1$-forms.   
%set $\Delta_{0,<}=\{ dx \}$ in the case of $1$-forms.   
 
The next result will be critical to deal with $1$-forms whose leading variable is a Puiseux characteristic exponent.  
\begin{pro}
\label{pro:q0}
 Assume $\Omega\in \hat{\Omega}_{\mathcal{C}}(\mathbb{C}^2,0)$
 (resp.  $f \in \mathbb{C}{[a_{\beta}]}_{\beta \in \mathcal{E}(\mathcal{C})} [[x,y]]$)
 belongs to $\Delta_{\beta_{l}, \leq}^{m}$ for some $1 \leq l \leq g$ and $0 \leq m \leq n_{l}-2$.
 Then $\Omega \in \tilde{\Delta}_{\beta_{l}, \leq}^{m}$ (resp. $f \in \tilde{\Delta}_{\beta_{l}, \leq}^{m}$).
 \end{pro}
 \begin{proof}
   With the notations of Definition \ref{def:max}, consider
   $\Omega \in \hat{\Omega}_{\mathcal{C}}(\mathbb{C}^2,0)$. The proof for
   $f \in \mathbb{C}{[a_{\beta}]}_{\beta \in \mathcal{E}(\mathcal{C})} [[x,y]]$ is
   analogous but it can also be deduced from the case of $1$-forms by considering
   $df = \frac{\partial f}{\partial x} dx + \frac{\partial f}{\partial y} dy$.
 
 By definition, we know that $[k]_{e_{l-1}} = [m \beta_{l}]_{e_{l-1}}$. 
 Let $\Gamma$ be a generic curve  with irreducible parametrization 
 \[ \Gamma (t) = \left( t^{n/e_{l-1}}, \sum_{\beta_0 \leq \beta < \beta_l} a_{\beta,
       \Gamma} t^{\beta/e_{l-1}} \right) .
 \]
 Denote $\Gamma^{*} \Omega = h(t) dt / t$. Notice that, using Landau's notation,
 \[
   e_{l-1} h(t^{e_{l-1}}) \frac{dt}{t} = (\Gamma (t^{e_{l-1}}))^{*} \Omega = 
   \left( C_{\beta_l} t^{k + \beta_l} + O(t^{k + \beta_l +1}) \right)  \frac{dt}{t} .
 \]
 Assume, aiming at contradiction, that $Q_{\beta_l} \not \equiv 0$. This implies
 $C_{\beta_l} \neq 0$ for a generic curve as above. As the left hand side has
 order a multiple of $e_{l-1}$, we obtain
 $[0]_{e_{l - 1}} = [k +\beta_l]_{e_{l-1}} $ and thus $(m+1) \beta_{l}$ is a multiple of
 $e_{l-1}$, which is impossible because $m+1$ is not a multiple of $n_l$.
 \end{proof}
\subsection{Leading variables of products}
Because our argument hinges on the substitutions \eqref{eq:main-substitution} and their
iterations, we need to study how some elementary operations affect the leading variables
of power series and $1$-forms. This is the aim of the following technical lemmas.

We first prove that product of several (say $a$) power series starting at a Puiseux characteristic
exponent $\beta_l$ with degree $0$
%having $Q_{\beta_l}\equiv 0$, 
and $b$ power series starting at $\beta<\beta_l$
gives a power series with leading variable $\beta_l$ of degree $a-1$. 
%and $Q_{\beta_l}\equiv 0$. 
We can also control the $c$-sequence of such
product.
\begin{lem}
\label{lem:prod}
Fix $1 \leq l \leq g$. Suppose
\[
 g_1 (t), \hdots, g_a (t) \in \tilde{\Delta}_{\beta_l,\leq}^{0} \
 \mathrm{and} \ 
 g_{a+1} (t), \hdots, g_b (t) \in
 \cup_{\beta' < \beta_l} \Delta_{\beta'}
\]
with $a \geq 1$. Then $\prod_{j=1}^{b} g_{j} (t) \in \tilde{\Delta}_{\beta_l,\leq}^{a-1}$. 
Moreover, the $c$-sequence of $\prod_{j=1}^{b} g_{j} (t)$ is equal to the sum of the
$c$-sequences of $g_1, \hdots, g_a$ except at its first term, which is $1$.
\end{lem}
\begin{proof} 
 For each $j$, let $\beta(j)$ be such that
 $g_{j} \in \Delta_{\beta(j)}$ for $1 \leq j \leq b$. 
 Obviously, $\beta(j) = \beta_l$
 for $1 \leq j \leq a$. We have
 \begin{equation}
 \label{equ:gj}
 g_{j} (t) = t^{k_j}
 \sum_{r=0}^{\infty} C_{{\mathfrak n}^{r} (\beta (j)),j} t^{{\mathfrak n}^{r} (\beta (j))}, 
 \end{equation}
 for $1 \leq j \leq b$ by Definition \ref{def:max}.  We need to prove, first, that
 $\prod_{j=1}^{b} g_{j} (t) $ is of the form
 \begin{equation}
 \label{equ:prod}
 \prod_{j=1}^{b} g_{j} (t) =
 t^{k} \sum_{r=0}^{\infty} C_{{\mathfrak n}^{r}(\beta_l)} t^{{\mathfrak n}^{r} (\beta_l)} 
 \end{equation} 
 where $C_{{\mathfrak n}^{r}(\beta_l)} \in {\mathbb C}[a_{\beta_0}, a_{{\mathfrak n}(\beta_0)}, \hdots]$ for any
 $r \geq 0$. By equation \eqref{equ:gj}, if the coefficient of $t^{\ell}$ in
 $\prod_{j=1}^{b} g_{j} (t)$ is non-vanishing, then
 \begin{equation}
 \label{equ:monomials}
 \ell = (k_1 + \beta (1) + m_1) + (k_2 + \beta (2) + m_2) + \ldots +
 (k_b + \beta (b) + m_b), 
 \end{equation}
 where $\beta (j) + m_j \in {\mathcal E} ({\mathcal C})$ and $m_j \in \mathbb{Z}_{\geq 0}$
 for any $1 \leq j \leq b$ by the first property of Definition \ref{def:max}. Notice that
 \[
   k = \nu_{0} \left( \prod_{j=1}^{b} g_{j} (t) \right) - \beta_l =
   \sum_{j=1}^{b} (k_j + \beta (j)) - \beta_l \in [(a-1) \beta_l]_{e_{l-1}}
 \]
 holds by the first property of Definition \ref{def:max}. By subtracting $k$ from
 Expression \eqref{equ:monomials}, we deduce that in order to prove
 \eqref{equ:prod} we just need to show the following claim: 
 if $\beta (j) + m_j \in {\mathcal E} ({\mathcal C})$ for any $1 \leq j \leq b$ then
 $\beta_l + \sum_{j=1}^{b} m_j \in {\mathcal E} ({\mathcal C})$.   
 
   Let us prove the claim. 
 Let $\beta' = \max_{1 \leq j \leq b} (\beta(j) + m_j)$, and set
 \[
   e(\beta)=
   \gcd \{ \beta'' \in \mathcal{E} (\mathcal{C})\ :\
   \beta'' \leq \beta\}.
 \] 
 Notice that $\beta' + r e (\beta') \in {\mathcal E}(\mathcal{C})$ for all
 $\beta' \in \mathcal{E} (\mathcal{C})$ and $r \in {\mathbb Z}_{\geq 0}$ by Definition of
 $\mathcal{E} (\mathcal{C})$.  Moreover, $e (\beta')$ divides $m_1, \hdots, m_b$,
 $\beta' - \beta_{l}$ and hence it also divides
 $\sum_{j=1}^{b} m_j - (\beta' - \beta_{l})$. As a consequence,
 \[
 \beta_l + \sum_{j=1}^{b} m_j =
 \beta' + \sum_{j=1}^{b} m_j - (\beta' - \beta_{l})
 \in {\mathcal E}(\mathcal{C}),
 \]
 as desired.  

Let us study the polynomials $R$ and $Q_{{\mathfrak n}^{r} (\beta_{l})}$, for $r \geq 0$, in Definition \ref{def:max}. 
 Now, if the coefficient of $t^{\ell}$ in $\prod_{j=1}^{b} g_{j} (t)$ does not vanish then
 $\ell = k + {\mathfrak n}^{r} (\beta_l)$ for some $r \geq 0$  by the above claim. Moreover,  the coefficient of order $\ell$ is
 a sum of terms of the form
 \begin{equation}
 \label{equ:term}
 C_{\beta (1) + m_1,1}  C_{\beta (2) + m_2,2}  \hdots  C_{\beta  (b)
   + m_b,b}    
 \end{equation}
 where 
 \begin{equation}
 \label{equ:coefm}
 {\mathfrak n}^{r} (\beta_l) = \ell - k  = \beta_l + m_1 + \hdots + m_b.  
 \end{equation} 
It is straightforward to verify that the expression 
\begin{equation}
\label{equ:lead}
 D_{{\mathfrak n}^{r}(\beta_l)} := \sum_{s=1}^{a}
 \left( C_{{\mathfrak n}^{r}(\beta_l),s} \prod_{1 \leq j \leq b, \ j \neq s}
   C_{\beta (j),j} \right) 
 \end{equation} 
 is of the form \eqref{equ:maximal} setting
 $R = \prod_{j=1}^{a} R_j \prod_{j=a+1}^{b} C_{\beta (j), j}$,
 where $R_j$ is the
 corresponding factor in equation \eqref{equ:maximal} associated to $g_j$.  By direct
 computation, we obtain that the sum $t^{k} \sum_{r \geq 0}  D_{{\mathfrak n}^{r}(\beta_l)}  t^{{\mathfrak n}^{r}(\beta_l)}$
 belongs to $\tilde{\Delta}_{\beta_l,\leq} $ and its $c$-sequence $(c_{{\mathfrak n}^{r}(\beta_l)})_{r \geq 0}$
 satisfies $c_{\beta_l}=1$ and
 $c_{{\mathfrak n}^{r} (\beta_l)} = \sum_{j=1}^{a} c_{{\mathfrak n}^{r} (\beta_l)} ^{j}$ for any $r \geq 1$, where
 $(c_{{\mathfrak n}^{r}(\beta_l)}^{j})_{r \geq 0}$ is the $c$-sequence of $g_{j}$ for
 $1 \leq j \leq a$.
 
 The terms of the form \eqref{equ:term} that are not in \eqref{equ:lead} satisfy
 $\beta (j) + m_j < {\mathfrak n}^{r} (\beta_l) $ for any $1 \leq j \leq b$. Thus, the sum of such terms
 defines a polynomial $Q_{{\mathfrak n}^{r} (\beta_l)} \in \mathbb{C}{[a_{\beta'}]}_{\beta_0 \leq \beta' < {\mathfrak n}^{r}(\beta)}$.
 %Up to show $Q_{\beta^{\prime}} \equiv 0$ whenever $\beta^{\prime}$ is a Puiseux characteristic exponent, 
 We obtain that 
 $\prod_{j=1}^{b} g_{j} (t) \in {\Delta}_{\beta_l,\leq} $ and its $c$-sequence is still
 $(c_{{\mathfrak n}^{r}(\beta_l)})_{r \geq 0}$. Moreover, there are no such terms for $r=0$ and thus 
 $\prod_{j=1}^{b} g_{j} (t) \in \tilde{\Delta}_{\beta_l,\leq} $.
 %
 %
 %
 %Consider the case in which ${\mathfrak n}^{r} (\beta_l)$ is a Puiseux exponent $\beta_s$. We claim
 %that if \eqref{equ:coefm} holds then $\beta_l + m_j = \beta_s$ for some
 %$1 \leq j \leq a$.  Otherwise, $\beta (j) + m_j < \beta_s$ for any $1 \leq j \leq b$,
 %which implies that $m_j \in (e_{s-1})$ for any $1 \leq j \leq b$, and in particular
 %\[
 % \beta_s = \beta_l + m_1 + \hdots + m_b  \in (e_{s-1}),
 %\]
 %contradicting $\beta_s \not \in (e_{s-1})$. We conclude that $C_{{\mathfrak n}^{r} (\beta_l)}$ is
 %equal to the expression \eqref{equ:lead}.  Since
 %$C_{{\mathfrak n}^{r}(\beta_l),1}, \hdots, C_{{\mathfrak n}^{r}(\beta_l),a}$ are multiples of $a_{\beta_s}$ and
 %$C_{\beta_l,1}, \hdots, C_{\beta_l,a}$ are multiples of $a_{\beta_l}$ by hypothesis, we
 %deduce that the polynomial $Q_{{\mathfrak n}^{r}(\beta)}$ associated to $C_{{\mathfrak n}^{r}(\beta_l)}$ in
 %Definition \ref{def:max} is identically zero.
 \end{proof}
The first important consequence applies to powers of approximate roots of $\Gamma$:
\begin{cor}
 \label{cor:root}
 Let $1 \leq l \leq g$ and $m \geq 1$. Then $f_{l}^{m}$ belongs to
 $\tilde{\Delta}_{\beta_l,\leq}^{m-1}$ and its $c$-sequence is equal to
 $1, m, \hdots, m, \hdots$
%and $Q_{\beta_l} \equiv 0$.
\end{cor}
\begin{proof}
 The power series $f_l$ is given by
 \begin{equation}
 \label{equ:pol_root}
 f_l = \prod_{k=0}^{ \nu_{l-1}-1 } (y - \sum_{\beta < \beta_l}e^{\frac{2 \pi i k \beta}{n}}
 a_{\beta} x^{\beta/n}) .
 \end{equation}
 In particular $\Gamma^{\ast}f_{l}=f_{l} \circ \Gamma$ is equal to
 \begin{equation*}
 \Gamma^{\ast}f_{l} =g_0(t) g_1(t) \cdots g_{ \nu_{l-1}-1 }(t)
 \end{equation*}
 where
 \begin{equation*}
 g_0(t) = \sum_{\beta\geq \beta_l}a_{\beta}t^{\beta}\in \tilde{\Delta}^{0}_{\beta_l,\leq}
 \end{equation*}
 which has constant $c$-sequence equal to $1$. 
 Since 
 \[ 
 (e^{\frac{2 \pi i k \beta_1}{n}}, e^{\frac{2 \pi i k \beta_2}{n}}, \ldots, e^{\frac{2 \pi i k \beta_{l-1}}{n}}) = (1, 1, \ldots, 1) \] 
 implies that $k$ is a multiple of $n/ \gcd (n, \beta_1, \ldots, \beta_{l-1} )=\nu_{l-1}$, we conclude that
 \begin{equation*}
 g_i(t) \in
 \bigcup_{\beta^{\prime}<\beta_{l}}\Delta^0_{\beta^{\prime}}
 \
 \mathrm{for}\
 1\leq i \leq \nu_{l-1} -1.
 \end{equation*} 
 The result is, thus, a direct consequence of Lemma \ref{lem:prod}, setting $a=m$ and
 $b=(\nu_{l-1}-1)m$, and the multinomial formula.
\end{proof}  
\begin{rem}
Remarks \ref{exa:leading} and  \ref{exa2:leading} can be obtained as an application of Corollary \ref{cor:root} where 
$f_{2}^{2} \in {\Delta}_{\beta_2,\leq}^{1} ={\Delta}_{10,\leq}^{1} $.
\end{rem} 
Corollary \ref{cor:root} translates to $1$-forms which are the product of a power of an approximate root of $\Gamma$ and a $1$-form with lesser leading variable:
\begin{cor}
\label{cor:f_dominates_omega}
Let $1 \leq l \leq g$, $m \geq 1$ and $\Omega \in \Delta_{\beta}$ with
$\beta < \beta_{l}$. Then
$f_{l}^{m} \Omega \in \tilde{\Delta}_{\beta_{l},\leq}^{m-1} $ and the $c$-sequence of
$f_{l}^{m} \Omega$ is equal to $1, m, m, \hdots$.
\end{cor}
Finally, for products of powers of approximate roots and $1$-forms with the same leading variable with degree $0$, we can also control the outcome, and the $c$-sequence.
The next result is a consequence of Corollary \ref{cor:root}, Lemma \ref{lem:prod} and Proposition  \ref{pro:q0}.
\begin{cor}
\label{cor:f_omega}
Let $1 \leq l \leq g$, $m \geq 0$ and
$\Omega \in {\Delta}_{\beta_{l},\leq}^{0}$. Then
$f_{l}^{m} \Omega \in \tilde{\Delta}_{\beta_{l},\leq}^{m}$ and the $c$-sequence
$(c_{{\mathfrak n}^{r} (\beta_l)}')_{r \geq 0}$ of $f_{l}^{m} \Omega$ has $c_{\beta_l}'=1$ and
$c_{{\mathfrak n}^{r} (\beta_l)}' = c_{{\mathfrak n}^{r} (\beta_l)} + m$ for $r \geq 1$ where
$(c_{{\mathfrak n}^{r} (\beta_l)})_{r \geq 0}$ is the $c$-sequence of $\Omega$.
\end{cor}

All these results will be relevant for what will be \emph{stage $0$} of
each level $l+1$, in which we shall make use of both the $1$-forms
constructed in levels up to $l$, and the approximate root $f_{l+1}$.

%\subsection{Virtual multiplicity}
%Before introducing the procedure in detail, we need the following result which will guarantee that, at the initial stage of each level, our $1$-forms cannot have huge contacts 
%with $\Gamma$.

%\begin{defi}
%Given $\Delta$ a formal singular irreducible curve $\Delta$ that has $g$ Puiseux characteristic exponents.
%We define the {\it virtual multiplicity} $\mu (\Delta)$ of $\Delta$ as $\nu_{g-1}$ (cf. Definition \ref{def:inv}).
%\end{defi}
% Let us remind the reader that $\nu_{g-1}$ is the multiplicity at the origin of the penultimate approximate root of $\Delta$. 
%\begin{teo}
%\label{teo:virtual}
%Suppose $\Delta$ is a formal singular irreducible curve that has $g$ Puiseux characteristic exponents and
%is invariant by a $1$-form $\Omega \in \hat{\Omega} \cn{2}$. Then the
%multiplicity of $\Omega$ at $(0,0)$ is at least $\mu (\Delta)$; that is,
%$\nu_{(0,0)} (\Omega) \geq \mu (\Delta)$.
%\end{teo}
%\begin{proof}
 %The result holds for foliations, i.e. when $\Omega = a(x,y) dx + b(x,y) dy$ satifies
 %$\gcd (a,b)=1$ (see \cite{Cano-Fortuny-Ribon:irreducible}). Because of this, it
 %suffices to prove it for $\Omega$ of the form $f \Omega'$ where the coefficients of
 %$\Omega'$ are coprime. In this case, we have either $f \circ \Delta \equiv 0$ or
 %$\Delta^{*} \Omega' \equiv 0$. We get 
 %\[
 %\nu_{(0,0)} (f) \geq \nu_{(0,0)} (\Delta) > \mu (\Delta)
 %\quad
 %\mathrm{or}
 %\quad \nu (\Omega') \geq \mu (\Delta),
 %\]
 %respectively. In
 %any case, $\nu ( f \Omega') \geq \nu_{g-1}$ holds.
%\end{proof}

\subsection{Terminal families}
\label{sec:terminal}
As described in  subsection \ref{subsec:basic_step}, given $0 \leq l <g$, the initial family 
$\mathcal{G}_{l+1}^{0}$ of level $l+1$ is built from the terminal family $\mathcal{T}_{l}$
of level $l$ and the approximate root $f_{l+1}$. In this section we properly define  terminal families, 
study their properties and then prove that our construction provides a terminal family 
$\mathcal{T}_{l+1}$ of level $l+1$.
\subsubsection{Definition and first properties of terminal families}
\begin{defi}
\label{def:level}
We say that $\mathcal{T}_{l}= (\Omega_1,\ldots \Omega_{2 {\nu}_{l}}) \subset {\Omega}_{\mathcal{C}}(\mathbb{C}^2,0)$ is a
\emph{terminal family of level $l$} if the following conditions hold:
\begin{enumerate}[label=(\alph*)]
%\begin{enumerate}[(a)]
\item \label{ta} $\mathcal{T}_{l}$ is a \nice{} family of level $l$; 
%The sequence 
% $\Omega_1, \hdots, \Omega_{\tilde{\nu}_{l}}$ satisfies the
 %first four conditions in Definition \ref{def:nice} (with $n$ is replaced
 %by~$\tilde{\nu}_{l}$);
\item \label{tb} the last contact satisfies
 $\nu_{\mathcal{C}} (\Omega_{2 {\nu}_{l}}) \leq \overline{\beta}_{l+1}$ when
 $0 \leq l <g$;
\item \label{tc} ${\mathcal T}_{l}$ is contained in $\cup_{\beta' \in {\mathcal E}({\mathcal C})} \Delta_{\beta', <}$.
Moreover, in the case  $0\leq l<g$, we have:
\[
  \mathcal{T}_{l}  \subset \bigg(\bigcup_{\substack{\beta' \in {\mathcal E}({\mathcal C})\\ \beta' < \beta_{l+1}}} \Delta_{\beta', <}\bigg)\bigcup\Delta_{\beta_{l+1}, <}^{0}
\]
 \end{enumerate}
\end{defi}
Condition \ref{tb} indicates that the sequence of multiplicities ``does not go too far''
and indeed that all of them are less than or equal to the multiplicity of $d f_{l+1}$
which is a form of level $l$ since it belongs to $\hat{M}_{2 {\nu}_{l}}$.  Condition
\ref{tc} implies, on one hand, that the leading variables are not greater than
$\beta_{l+1}$ for $l <g$, and that if they are $\beta_{l+1}$, then they appear with degree
$0$; and on the other hand, that the $c$-sequences are required to be \emph{strictly
  increasing}.  The last property is inherent to the study the semi-module of contacts of
holomorphic $1$-forms; notice that Expression \eqref{equ:semigroup} describes the
semigroup of the curves in ${\mathcal C}$ but
$f_0^{k_0}f_1^{k_1}\cdots f_r^{k_r} \in {\Delta}_{\beta_{r},\leq}^{k_{r}-1} \setminus
{\Delta}_{\beta_{r},<}^{k_{r}-1} $ if $r \geq 1$ and $k_{r} >0$ and thus those products do
not have strictly increasing $c$-sequences. Condition \ref{tc} is used in the inductive
steps of our construction.
%, but there is nothing intrinsic in them.

The family $\mathcal{T}_0= (dx,d f_1 )$ is terminal of level $0$ because
 $\Omega_1 \in \Delta_{n,<}$ and $\Omega_2\in \Delta_{\beta_1,<}^0$. This is our starting point.

 \begin{rem}
 \label{rem:propl}
 In a terminal family ${\mathcal T}_l$ of level $l < g$, the multiplicity
 $\nu_{\mathcal{C}} (\Omega)$ of any $1$-form $\Omega \in \mathcal{T}_l$
 belongs to $[\beta_{l+1}]_{e_l}$ or $[0]_{e_l}$ (because of their
 corresponding leading variables). This is a key property for our
 later arguments. More precisely, we have
 %Let $\mathcal{T}_l=(\Omega_1,\ldots,\Omega_{\tilde{\nu}_l})$ be a terminal family of level $l$. 
 %If $\Omega_j\in \mathcal{T}_l$, then
\begin{itemize}
\item If $\Omega$ is of level $l-1$, then
 $\nu_{\mathcal{C}} (\Omega) \in (e_l)$; 
\item Let be $\Omega$ of level $l$. Then
 $\Omega \in \cup_{\beta < \beta_{l+1}} \Delta_{\beta, <} \Leftrightarrow
%for some $\beta < \beta_{l+1}$, 
\nu_{ \mathcal{C}} (\Omega) \in (e_l)$;
\item Let be $\Omega$ of level $l$. Then
 $\Omega \in \Delta_{\beta_{l+1}, <}^{0} \Leftrightarrow \nu_{ \mathcal{C}} (\Omega) \in [\beta_{l+1}]_{e_l}$ .
\end{itemize} 
\end{rem}
Hence, a terminal family of level $l<g$ ``covers'' all the
congruences modulo $n$ in  
%the set 
$[(e_l)]_{n} \cup [\beta_{l+1} + (e_l)]_{n}$ (cf. Definition \ref{def:modulo}). 
\begin{cor}
 \label{rem:all-modn-congruences} 
 We have  $[\nu_{\mathcal{C}} ( \mathcal{T}_l) ]_{n} = [(e_l)]_{n} \cup [\beta_{l+1} + (e_l)]_{n}$ for $l<g$. 
 %contains exactly $\tilde{\nu}_l=2\nu_l$ elements.
\end{cor}
\begin{proof}
 This holds because, on one hand, there are $2 \nu_l$ congruences modulo
 $n$ in the set $(e_l) \cup \{ \beta_{l+1} + (e_l) \}$
 ($\nu_l$ from $(e_l)$ and other
 $\nu_l$ from $\beta_{l+1} + (e_l)$); and on the other hand, $[\nu_{\mathcal{C}} ( \mathcal{T}_l) ]_{n}$ 
 has $2 \nu_l$ elements by definition of terminal family. 
 %the numbers $\nu_{\Gamma} \Omega_1, \hdots, \nu_{\Gamma} \Omega_{\tilde{\nu}_l}$
 %are pairwise incongruent modulo $n$.
\end{proof}

\subsubsection{From level $l$ to level $l+1$}
\label{subsec:ll+1}
Suppose $\mathcal{T}_l= (\Omega_1, \hdots, \Omega_{2 {\nu}_{l}})$ is a terminal family of
level $l$ for $0 \leq l < g$. 
Our aim is to construct a terminal family of level $l+1$
from $\mathcal{T}_l$. This requires several stages. The first one (stage $0$) is
the construction of the initial family $\mathcal{G}_{l+1}^0$ of stage $0$ (Definition \ref{def:set-stage0}).
 \begin{lem}
 \label{lem:no3} 
   $\mathcal{G}_{l+1}^0$  is an \tidy{} family of level $l+1$. 
 \end{lem}  
 \begin{proof}
 Property \ref{cond1} holds by construction.
 
  Notice that $\nu_{\mathcal{C}} (f_{l+1}) = \overline{\beta}_{l+1}$,
 $\nu (f_{l+1}) = \nu_l$ and
 $n \leq \nu_{\mathcal{C}} (\Omega_j) \leq \overline{\beta}_{l+1}$ for
 $1 \leq j \leq 2 \nu_l$ by the properties of a terminal family of level $l$. Since the
 sequence $(\nu_{\mathcal{C}} (\Omega_{j}))_{1 \leq j \leq 2 \nu_{l}}$ satisfies property
 \ref{cond3} and
 \[
 \nu_{\mathcal{C}} (\Omega_{2 \nu_{l} +1}^{0}) - \nu_{\mathcal{C}} (\Omega_{2 \nu_{l}}^{0})
 \geq (\overline{\beta}_{l+1} +n) - \overline{\beta}_{l+1} = n,
 \]
 property \ref{cond3} 
 holds for $(\nu_{\mathcal{C}} (\Omega_{j}^{0}))_{1 \leq j \leq 2 \nu_{l+1}}$. 
 Thus $\mathcal{G}_{l+1}^0$ is \tidy{}. 
 \end{proof}
 \begin{rem}
 \label{rem:cong_conflict}
  Denote ${\mathfrak S} = \{ (a,j) : 0 \leq a < n_{l+1} \ \mathrm{and} \ 1 \leq j \leq 2 {\nu}_{l} \}$.
  Since
 \[
 \nu_{\mathcal{C}} (f_{l+1}^{a} \Omega_j) = a \overline{\beta}_{l+1} + \nu_{\mathcal{C}} (\Omega_j) \in (e_{l+1})
 \]
 for any $(a,j) \in {\mathfrak S}$,  the set 
 $ [\nu_{ \mathcal{C}} (\mathcal{G}_{l+1}^0)]_{n}$ contains at
 most $\nu_{l+1} = n/e_{l+1}$ elements. Notice also that, since
 $m \overline{\beta}_{l+1} \not \in (e_l)$ for any $1 \leq m < n_{l+1}$, we have
 \[
 \nu_{\mathcal{C}} (f_{l+1}^{a} \Omega_j) - \nu_{\mathcal{C}} (f_{l+1}^{b} \Omega_k) \not \in (n)
 \]
 if $(a,j), (b,k) \in {\mathfrak S}$, $(a,j) \neq (b,k)$ and $\nu_{\mathcal{C}} (\Omega_j)$
 and $\nu_{\mathcal{C}} (\Omega_k)$ belong both to $(e_l)$ or both to $[\beta_{l+1}]_{e_l}$. In
 particular any element of $[(e_{l+1})]_{n}$
% ${\mathbb Z}/n{\mathbb Z}$ 
 has at most two representatives in
 $(\nu_{\mathcal{C}} ({\Omega}_{j}^{0}))_{j=1}^{2 \nu_{l+1}}$. 
 %Thus $(\nu_{\mathcal{C}} (\Omega_{j}^{0}))_{j=1}^{2 \nu_{l+1}}$ defines at least $\nu_{l+1}$ classes modulo $n$. 
 Hence, $(\nu_{\mathcal{C}} (\Omega_{j}^{0}))_{j=1}^{2 \nu_{l+1}}$ defines
 exactly the $\nu_{l+1}$ classes modulo $n$ in $[(e_{l+1})]_{n}$ and each class has two representatives in it. 
\end{rem}

This property will allow us to choose, from each class modulo $n$, the
$1$-form whose representative appears later in the sequence and substitute
it for another one using the key formula
\eqref{eq:main-substitution}, in order to obtain the $2 {\nu}_{l+1}$
new forms of the next stage $1$.
There are two cases, namely $l+1 <g$ and $l+1=g$.
In the former case, this property of
having two representatives in each class modulo $n$ is stable up to a final
stage in which there is only one representative in each class modulo $n$;
the forms at this stage will be the terminal ones. In the case $l+1=g$, we will see that at some stage  the family provides a nice basis.
%(Definition \ref{def:nice}).

We construct families $\mathcal{G}_{l+1}^{0}, \mathcal{G}_{l+1}^{1}, \ldots$ 
(see subsection \ref{subsec:basic_step}) until we obtain a terminal family 
$\mathcal{G}_{l+1}^{r}$.   
The following definition summarizes all the properties we
shall make use of when passing from stage $s$ to stage $s+1$, and they will guarantee that
we can find a terminal family at some stage.
Assume a family $\mathcal{G}_{l+1}^s=(\Omega_1^{s}, \ldots, \Omega_{2 {\nu}_{l+1}}^s)$ of level $l+1$ at stage $s$ is
given, with $0\leq l < g$ and $s\geq 1$.  
Recall that $\mathcal{T}_l=(\Omega_1,\ldots, \Omega_{2 {\nu}_l})$ is a terminal family of level $l$. 
\begin{defi}
\label{def:good}
The family $\mathcal{G}_{l+1}^s$ 
 is a \emph{good family} of stage $s \geq 1$ 
(in level $l+1$) if the following properties hold: 
\begin{enumerate}[label=(\Roman*)]
 \item \label{good1} The first $2 {\nu}_{l}$ elements come from the previous terminal family ${\mathcal T}_l = (\Omega_1, \ldots, \Omega_{2 {\nu}_l})$: 
 $\Omega_{j}^{s} = \Omega_j$ for any $1 \leq j \leq 2 {\nu}_{l}$.
 \item \label{good2} $(\Omega_{1}^{s}, \hdots, \Omega_{2 {\nu}_{l+1}}^{s})$ is an \tidy{} family of level $l+1$.  
% \item \label{good5} They have leading variables $a_{\beta}$ satisfying $\beta\leq N^s(\beta_{l+1})$, and their
% $c$-sequences are increasing, and strictly so if $\beta<N^{s}(\beta_{l+1})$, i.e.:
% $\Omega_{j}^{s} \in \cup_{\beta' < N^{s}(\beta_{l+1})} \Delta_{\beta',<} \cup
% \Delta_{N^{s}(\beta_{l+1}),\leq} $ for any $1 \leq j \leq 2 {\nu}_{l+1}$.
% \end{enumerate}
% For the case $s \geq 1$, we require two more conditions.
% \begin{enumerate}[label=(\Roman*)]
% \setcounter{enumi}{3}
 \item \label{good6} There is a partition of $\{1, 2, \hdots, 2 \nu_{l+1} \}$ in two sets 
 ${\mathcal L}_s$ and ${\mathcal U}_s$ of the same cardinal
 such that $\Omega_{j}^{s} \in  \cup_{\beta' < {\mathfrak n}^{s}(\beta_{l+1})} \Delta_{\beta^{\prime},<}$ 
 if $j \in {\mathcal L}_s$
 and  $\Omega_{j}^{s} \in   \Delta_{{\mathfrak n}^{s}(\beta_{l+1}),<}^{0}$ if $j \in {\mathcal U}_s$. 
% In other words, the $c$-sequences are strictly increasing in any case and the 
 %degree of the leading variable $a_{\beta}$ is $0$ if $\beta=N^s(\beta_{l+1})$.
 %$\Omega_{j}^{s} \in \cup_{\beta' < N^{s}(\beta_{l+1})} \Delta_{\beta',<} \cup
 %\Delta_{N^{s}(\beta_{l+1}),<}^{0} $ for any $1 \leq j \leq 2 {\nu}_{l+1}$;
 \item \label{good7} We have
 \begin{equation*} \label{eq1}
 \begin{split}
 \{ [\nu_{\mathcal C}(\Omega_{j}^{s})]_{n} : j \in {\mathcal L}_s \} = [(e_{l+1})]_{n}  & \ \ \mathrm{and} \\
 \{  [\nu_{\mathcal C}(\Omega_{j}^{s})-{\mathfrak n}^{s}(\beta_{l+1})]_{n} : j \in {\mathcal U}_s \} = [(e_{l+1})]_{n} .&  
 \end{split}
 \end{equation*}
 In particular $[\nu_{\mathcal C}(\Omega_{j}^{s})]_{n} \neq [\nu_{\mathcal C}(\Omega_{k}^{s})]_{n}$ if $j \neq k$
 and both $j, k$ belong to ${\mathcal L}_s$  or both belong to ${\mathcal U}_s$.
 \end{enumerate} 
\end{defi} 
%\javier{}
%Good families in level $g$ have twice as many elements as terminal families of level $g$. This property motivates the next definition.
%\begin{defi}
%We say that a family $\mathcal{G}_{g}^s=(\Omega_1^{s}, \ldots, \Omega_{2 n}^s)$ is terminal if 
%$\{ \Omega_1^{s}, \ldots, \Omega_{n}^s \}$ is a terminal family of level $g$.
%\end{defi} \javier{}

\subsubsection{From stage $s$ to $s+1$ inside level $l+1$}
\label{subsec:stos+1}
Assume, from now on, that the families
$\mathcal{G}^{0}_{l+1}, \ldots, \mathcal{G}^{s}_{l+1}$ of level $l+1$ have been
constructed for $s\geq 0$,  $\mathcal{G}^{j}_{l+1}$ is good for $1 \leq j \leq s$  and that 
\begin{equation}
\label{cond:good_tech}
\{ \beta_{l+1} = {\mathfrak n}^{0}(\beta_{l+1}), {\mathfrak n}^{1} (\beta_{l+1}), \ldots, {\mathfrak n}^{s} (\beta_{l+1}) \} \subset (e_{l+1}) .
\end{equation}
Now, we explain how to construct a good family ${\mathcal G}_{l+1}^{s+1}$.
If  ${\mathcal G}_{l+1}^{s}$ is non-terminal, ${\mathcal G}_{l+1}^{s+1}$ is closer than ${\mathcal G}_{l+1}^{s}$  to be terminal.
Later on, we will see that one of the good families constructed is indeed terminal. 
%none is terminal (otherwise, we finish level $l+1$
%setting $\mathcal{T}_{l+1}=\{ \Omega_{1}^{s}, \hdots, \Omega_{\tilde{\nu}_{l+1}}^{s} \}$ the first time $\mathcal{G}_{l+1}^s$ is
%terminal).
\begin{pro}
\label{pro:well_def_fam}
$\mathcal{G}^{s+1}_{l+1}$ is well-defined.
\end{pro}
\begin{proof}
Consider $1 \leq k \leq 2 {\nu}_{l+1}$. There exists a unique $1 \leq j \leq 2 {\nu}_{l+1}$, with $j \neq k$, such that 
$\nu_{\mathcal{C}} (\Omega_{k}^{s}) - \nu_{\mathcal{C}} (\Omega_{j}^{s}) = dn$ for some
$d \in {\mathbb Z}$. This holds by Remark \ref{rem:cong_conflict} 
for the case $s=0$, and by Property \ref{good7} for $s \geq 1$.
Notice that ${\mathcal G}_{l+1}^{s}$ is \tidy{} by  Lemma \ref{lem:no3} and Property \ref{good2}.
Hence, we get $d \leq 0$ if $j>k$ and $d \geq 0$ if $j<k$.
Therefore, $\Omega_{k}^{s+1} = \Omega_{k}^{s}$ holds in the former case whereas
$\Omega_{k}^{s+1}$ is given by substitution \eqref{eq:main-substitution} in the latter case. 
We deduce that $\Omega_{k}^{s+1}$ is a well-defined element of ${\Omega}_{\mathcal C}(\mathbb{C}^2,0)$.
\end{proof}
\begin{rem}
\label{rem:inv_modulus} 
%Notice that 
The equality of $\mathbb{C}[[x]]$-modules
\[ \sum_{j=1}^{2 {\nu}_{l+1}} {\mathbb C}[[x]] \Omega_{j,\Gamma}^{s} = 
\sum_{j=1}^{2 {\nu}_{l+1}} {\mathbb C}[[x]] \Omega_{j,\Gamma}^{s+1} =
\sum_{j=1}^{2 {\nu}_{l+1}} \mathbb{C}[[x]] \overline{\Omega}_j \]
holds for generic $\Gamma \in \mathcal{C}$ (cf. Definition \ref{def:form_space}), by construction. 
\end{rem} 
The next remark fixes the notation for the next results.
\begin{rem}
\label{rem:next} 
Consider $1 \leq j < k \leq 2 {\nu}_{l+1}$ with 
$\nu_{\mathcal{C}} (\Omega_{k}^{s}) - \nu_{\mathcal{C}} (\Omega_{j}^{s}) = dn$ for some
$d \in {\mathbb Z}_{\geq 0}$.
Let $\beta'$ and $\beta''$ be such that $\Omega_{j}^{s} \in \Delta_{\beta', \leq}$,
$\Omega_{k}^{s} \in \Delta_{\beta'',\leq}$, and write
\begin{equation}
\label{equ:formulas}
 \begin{split}
 t \Gamma^{*} (\Omega_{j, \Gamma}^{s}) &= \left( C_{\beta', j}(\Gamma) t^{\nu_{\mathcal{C}} (\Omega_{j}^{s})} + 
 \sum_{r \geq 1} C_{{\mathfrak n}^{r}(\beta'), j}(\Gamma) t^{\nu_{\mathcal{C}} (\Omega_{j}^{s}) +{\mathfrak n}^{r}(\beta') - \beta'} \right) dt
 \\
 t \Gamma^{*} ( \Omega_{k, \Gamma}^{s}) &= \left( C_{\beta'', k}(\Gamma) t^{\nu_{\mathcal{C}} (\Omega_{k}^{s})} + 
  \sum_{r \geq 1} C_{{\mathfrak n}^{r}(\beta''), k}(\Gamma) t^{\nu_{\mathcal{C}} (\Omega_{k}^{s}) +{\mathfrak n}^{r}(\beta'') - \beta''} \right) 
  dt, 
\end{split}
\end{equation}
where $C_{\beta', j}, C_{\beta'', k} \in \mathbb{C}{[a_{\beta}]}_{\beta \in \mathcal{E}(\mathcal{C}) } \setminus \{0\}$.
We have
\begin{equation}\label{eq:main-substitution-s+1}
 \Omega_{k}^{s+1} = C_{\beta', j} \Omega_{k}^{s} - C_{\beta'', k} x^{d} \Omega_{j}^{s} .
\end{equation}
\end{rem}

The following elementary results will be used when studying the properties of stage $s+1$
given stage $s$. Consider a family $\mathcal{G}_{l+1}^s$, either the initial family of level $l+1$ if $s=0$
or a good family if $s \geq 1$.
\begin{lem} 
 \label{lem:dom}
 Assume that $[\nu_{\mathcal{C}} (\Omega^{s}_j)]_{n}=[\nu_{\mathcal{C}} (\Omega^{s}_{k})]_{n}$ for $j \neq k$, 
 and that  
 $\Omega^{s}_k \in \Delta_{{\mathfrak n}^{s}(\beta_{l+1}),<}$ and 
 $\Omega^{s}_j \in \cup_{\beta' < {\mathfrak n}^{s}(\beta_{l+1})} \Delta_{\beta',<}$. Let
 $(c_{{\mathfrak n}^{r}(\beta_{l+1})})_{r \geq 0}$ be the $c$-sequence of $\Omega^{s}_k$. Then: 
\[
 \Omega_{\max (j,k)}^{s+1} \in \Delta_{{\mathfrak n}^{s+1}(\beta_{l+1}),<}^{0},\,\,\,
 \Omega_{\min (j,k)}^{s+1}=\Omega_{\min (j,k)}^s
\] 
and the $c$-sequence of $\Omega_{\max(j,k)}^{s+1}$ is equal to %\marginpar{OJO a los superindices!}
\[
 ((c_{{\mathfrak n}^{s+1}(\beta_{l+1})})^{-1} (c_{{\mathfrak n}^{s+r}(\beta_{l+1})}))_{ r \geq 1} .
\] 
  In particular,
\[ 
 \nu_{\mathcal{C}} (\Omega_{\max (j,k)}^{s+1}) - \nu_{\mathcal{C}} (\Omega_{\max (j,k)}^{s}) = 
 {\mathfrak n}^{s+1}(\beta_{l+1}) - {\mathfrak n}^{s}(\beta_{l+1}) 
\]
holds. 
\end{lem}
\begin{proof}
 This follows straightforwardly from equations  \eqref{equ:formulas} and \eqref{eq:main-substitution-s+1}.
%
%
%
 %except the condition on the coefficients
 %whose leading variable is a Puiseux exponent.
 %Denote
 %\[  t \Gamma^{*} (\Omega_{\max (j,k), \Gamma}^{s+1}) = \left(   
 %\sum_{r \geq 1} C_{{\mathfrak n}^{r+s} (\beta_{l+1})} 
 %t^{\nu_{\Gamma} (\Omega_{\max(j,k)}^{s}) +
 %  {\mathfrak n}^{r+s} ( \beta_{l+1}) -  {\mathfrak n}^{s} (\beta_{l+1})} \right) dt.
 %\]
 %We use the notations of Equation \eqref{equ:formulas} where
 %$\beta'' = {\mathfrak n}^{s} (\beta_{l+1})$. By definition,
 %\[
  % C_{\beta} =
  %C_{\beta'}^{1} C_{\beta}^{2} -
 % C_{{\mathfrak n}^{s}(\beta_{l+1})}^{2} C_{\beta + \beta' - {\mathfrak n}^{s} (\beta_{l+1})}^{1}
 %\]
 %for $\beta > {\mathfrak n}^{s} (\beta_{l+1})$.  If $\beta = \beta_m$, then $C_{\beta}^{2}$ is
 %multiple of $a_{\beta}$ by Definition \ref{def:max}. Moreover,
 %$\beta_m + \beta' - {\mathfrak n}^{s} (\beta_{l+1})$ is an exponent smaller than $\beta_m$ and does
 %not belong to $(e_{m-1})$.  Since such a exponent does not exist, we deduce
 %$C_{\beta + \beta' - {\mathfrak n}^{s} (\beta_{l+1})}^{1} \equiv 0$ and hence $C_{\beta}$ is multiple
 %of $a_{\beta}$.
\end{proof}  
 \begin{lem}
 \label{lem:non_dom}
 Suppose that
 $[\nu_{\mathcal{C}} (\Omega_{j}^{0})]_{n}=[\nu_{\mathcal{C}}
 (\Omega_k^0)]_n$ for $j<k$, and that $\Omega_{j}^{0}$ and $\Omega_{k}^{0}$
 belong to $\tilde{\Delta}_{\beta_{l+1},\leq}$. % or $\cup_{m=0}^{n_{l+1}-2} \Delta_{\beta_{l+1},\leq}$. 
 Denote by
 $(c_{{\mathfrak n}^{r}(\beta_{l+1}),j})_{r \geq 0}$ and
 $(c_{{\mathfrak n}^{r}(\beta_{l+1}),k})_{r \geq 0}$ the $c$-sequences of
 $\Omega_j^{0}$ and $\Omega_k^{0}$ respectively. If
 $c_{{\mathfrak n}(\beta_{l+1}),j}\neq c_{{\mathfrak n}(\beta_{l+1}),k}$ and
 \[
 ((c_{{\mathfrak n}(\beta_{l+1}),j} - c_{{\mathfrak n}(\beta_{l+1}),k})^{-1} (c_{{\mathfrak n}^{r}(\beta_{l+1}),j} -
 c_{{\mathfrak n}^{r}(\beta_{l+1}),k}))_{r \geq 1}
 \] is an increasing sequence in ${\mathbb Q}^{+}$, then
 $\Omega_{k}^{1} \in \Delta_{{\mathfrak n}(\beta_{l+1}),<}^{0}$ and the above sequence is the
 $c$-sequence of $\Omega_{k}^{1}$.    In particular,
\begin{equation}
\label{equ:step1}
 \nu_{\mathcal{C}} (\Omega_{k}^{1}) - \nu_{\mathcal{C}} (\Omega_{k}^{0}) = {\mathfrak n}(\beta_{l+1}) - \beta_{l+1}
\end{equation}
holds.  
\end{lem}
\begin{proof}
%Note that $\cup_{m=0}^{n_{l+1}-2} \Delta_{\beta_{l+1},\leq} \subset \tilde{\Delta}_{\beta_{l+1},\leq}$ 
%by Proposition \ref{pro:q0} and so 
%$\Omega_j^{0}, \Omega_k^{0} \in  \tilde{\Delta}_{\beta_{l+1},\leq}$.
With the notation \eqref{equ:formulas}, we have
 \[
 C_{{\mathfrak n}^{r}(\beta_{l+1}),\delta} = \tilde{c}_{\delta} c_{{\mathfrak n}^{r}(\beta_{l+1}),\delta}
 a_{{\mathfrak n}^{r}(\beta_{l+1})} R_{\delta}(a_{\beta_1}, \hdots, a_{{\mathfrak p}(\beta_{l+1})})
 a_{\beta_{l+1}}^{m_\delta} + Q_{{\mathfrak n}^{r}(\beta_{l+1}), \delta}
\]
for $\delta\in \{j,k\}$ and $r \geq 0$, following Definition \ref{def:max}, where
$Q_{\beta_{l+1}, j} \equiv Q_{\beta_{l+1}, k} \equiv 0$ by hypothesis. 
%since $\Omega_j^{0}, \Omega_k^{0} \in  \tilde{\Delta}_{\beta_{l+1},\leq}$.

If we now denote 
\[
  t \Gamma^{*} (  \Omega_{k, \Gamma}^{1}  )
  =
  \left(  \sum_{r \geq 1} C_{{\mathfrak n}^{r} (\beta_{l+1})} (\Gamma)
    t^{\nu_{\Gamma} (\Omega_{k}^{s}) + {\mathfrak n}^{r} (\beta_{l+1}) - \beta_{l+1}}
  \right) dt,
\]
then 
\[
  C_{{\mathfrak n}^{r} (\beta_{l+1})}
  = ( c_{{\mathfrak n}^{r}(\beta_{l+1}),k} - c_{{\mathfrak n}^{r}(\beta_{l+1}),j} )
  [\tilde{c}_j \tilde{c}_k a_{\beta_{l+1}}^{m_j+m_k+1} R_j R_k]
  a_{{\mathfrak n}^{r} (\beta_{l+1})} + Q_r,
\]
for $r \geq 1$, where
$Q_r \in \mathbb{C}{[a_{\beta}]}_{\beta_1 \leq \beta < {\mathfrak n}^{r}(\beta_{l+1})}$, 
and the result follows.
\end{proof}
We want to show that $\mathcal{G}^{s+1}_{l+1}$ is good. % if $\mathcal{G}^{s}_{l+1}$ is good.
%and non-terminal. 
To this end, we need to distinguish the case $s=0$ and the rest. The former is the most complicated step of our construction. 
 \begin{defi}
   We define the initial sets
   ${\mathcal L}_1$ and ${\mathcal U}_1$ as follows:
 \begin{itemize}
        \item Let $1 \leq j \leq 2 \nu_l$ such that
          $\Omega_j \in \Delta_{\beta^{\prime}, <}$ with $\beta^{\prime} <
          \beta_{l+1}$. Then $j \in {\mathcal L}_1$ and $j + 2 r \nu_l \in {\mathcal U}_1$
          for $1 \leq r < n_{l+1}$.
        \item Let $1 \leq j \leq 2 \nu_l$ such that
          $\Omega_j \in \Delta_{\beta^{\prime}, <}$ with $\beta^{\prime} =
          \beta_{l+1}$. Then $ j + 2 (n_{l+1}-1) \nu_l \in {\mathcal U}_1$ and
          $j + 2 r \nu_l \in {\mathcal L}_1$ for $0 \leq r \leq n_{l+1}-2$.
 \end{itemize}
\end{defi}
Recall that the sets $M_j$ and $\hat{M}_j$ were defined in page \pageref{def:Mj}.
 \begin{pro}
\label{pro:aux1}
The family ${\mathcal G}_{l+1}^{1}$ satisfies
\begin{itemize}
\item ${\mathcal U}_1 = \{ 1 \leq k \leq 2 \nu_{l+1} : \Omega_{k}^{1} \neq \Omega_{k}^{0} \}$.
\item  $\Omega_{j}^{1} \in \cup_{\beta^{\prime} \leq \beta_{l+1}} \Delta_{\beta^{\prime},<} \cap \hat{M}_{j}$ 
for any $j \in {\mathcal L}_1$.
\item  $\Omega_{k}^{1} \in \Delta_{{\mathfrak n}(\beta_{l+1}),<}^{0} \cap \hat{M}_{k}$ for any $k \in {\mathcal U}_1$.
\item $\nu_{\mathcal{C}} (\Omega_{k}^{1}) - \nu_{\mathcal{C}} (\Omega_{k}^{0}) = {\mathfrak n}(\beta_{l+1}) - \beta_{l+1}$
for any $k \in {\mathcal U}_1$.
\item $\{ [\nu_{\mathcal C}(\Omega_{j}^{1})]_{n} : j \in {\mathcal L}_{1} \} = [(e_{l+1})]_{n}$.
\item $\{ [\nu_{\mathcal C}(\Omega_{k}^{1}) - {\mathfrak n}(\beta_{l+1})]_{n} : k \in {\mathcal U}_{1} \} = [(e_{l+1})]_{n}$.
\end{itemize}
 \end{pro} 
\begin{proof} 
        %Consider $1 \leq k \leq 2 \nu_{l+1}$.  We have $\Omega_{k}^{0} = f_{l+1}^{m_k} \Omega_b$.
Suppose that $\nu_{\mathcal{C}} (\Omega_{k}^{0}) - \nu_{\mathcal{C}} (\Omega_{j}^{0}) = dn$ for
some $d \in {\mathbb Z}_{\geq 0}$ and $1 \leq j <  k \leq 2 \nu_{l+1}$.
First, we want to see that $j \in {\mathcal L}_1$ and $k \in {\mathcal U}_1$. 
As $\Omega_{j}^{0} = f_{l+1}^{m_j} \Omega_a$ and
$\Omega_{k}^{0} = f_{l+1}^{m_k} \Omega_b$, we deduce that
\[
(m_{k} -m_{j}) \overline{\beta}_{l+1} + ( \nu_{\mathcal{C}} (\Omega_b) - \nu_{\mathcal{C}} (\Omega_a))
= dn.
\]
Moreover, $j < k$ implies $m_{j} \leq m_{k}$, and since $0 \leq m_j, m_k < n_{l+1}$, 
Remark \ref{rem:propl} implies
that this last equality can only happen in two cases
\begin{itemize}
\item[(a)] $m_{k}=m_{j}+1$, $ \nu_{\mathcal{C}} (\Omega_b) \in (e_l)$ and
$\nu_{\mathcal{C}} (\Omega_a) \in [\beta_{l+1}]_{e_l}$ or
\item[(b)] $m_k=n_{l+1} -1$, $m_j=0$,
$ \nu_{\mathcal{C}} (\Omega_b) \in [\beta_{l+1}]_{e_l}$ and
$ \nu_{\mathcal{C}} (\Omega_a) \in (e_l)$.
\end{itemize}
Since any $[x]_{n} \in [(e_{l+1})]_{n}$ has exactly two representatives in $(\nu_{\mathcal C}(\Omega_{j}^{0}))_{1 \leq j \leq 2 \nu_{l+1}}$
by Remark \ref{rem:cong_conflict}, we deduce
${\mathcal U}_1 = \{ 1 \leq r \leq 2 \nu_{l+1} : \Omega_{r}^{1} \neq \Omega_{r}^{0} \}$.

Next, let us show $\Omega_{j}^{1} \in \cup_{\beta^{\prime} \leq \beta_{l+1}} \Delta_{\beta^{\prime},<} \cap \hat{M}_{j}$, 
$\Omega_{k}^{1} \in \Delta_{{\mathfrak n}(\beta_{l+1}),<}^{0} \cap \hat{M}_{k}$
and equation \eqref{equ:step1}.  

\strut

\textbf{Case (a).} If $m_k=m_j+1$, then $\Omega_a \in \Delta_{\beta_{l+1},<}^{0}$ and
$\Omega_b \in \Delta_{\beta,<}$ for some $\beta < \beta_{l+1}$ by Remark
\ref{rem:propl}. Denote by $(c_{{\mathfrak n}^{r}(\beta_{l+1})})_{r \geq 0}$ the $c$-sequence of
$\Omega_a$. Since $\beta < \beta_{l+1}$, notice that
$\Omega_{k}^{0} = f_{l+1}^{m_{j}+1} \Omega_b$ belongs to
$\tilde{\Delta}_{\beta_{l+1}, \leq}^{m_j} \setminus \Delta_{\beta_{l+1},<}^{m_j}$ and its $c$-sequence is
$1, m_{j}+1, \hdots, m_{j}+1, \hdots$ by Corollary \ref{cor:f_dominates_omega}. 
We have 
$\Omega_{j}^{1} = \Omega_{j}^{0} \in \tilde{\Delta}_{\beta_{l+1},<}^{m_j} \cap \hat{M}_{j}$  and that
%Proposition \ref{pro:q0} and Lemma \ref{lem:prod}.
its $c$-sequence   is
%and $\Omega_{k}^{0}$ are
\[
1, c_{{\mathfrak n}(\beta_{l+1})} + m_{j}, c_{{\mathfrak n}^{2}(\beta_{l+1})} + m_{j}, \hdots 
% \ \mathrm{and} \ 1,
% m_{j}+1, \hdots, m_{j}+1, \hdots \
\] 
%respectively, by Corollaries \ref{cor:f_omega} and \ref{cor:f_dominates_omega}\pedro{}. 
by Corollary \ref{cor:f_omega}. 
Since
$(c_{{\mathfrak n}^{r}(\beta_{l+1})})_{r \geq 0}$ is strictly increasing, we deduce that
$\Omega_{k}^{1} \in \Delta_{{\mathfrak n}(\beta_{l+1}),<}^{0} \cap \hat{M}_{k}$, that its
$c$-sequence is
\[
((c_{{\mathfrak n}(\beta_{l+1})} - 1)^{-1} (c_{{\mathfrak n}^{r}(\beta_{l+1})} - 1))_{ r \geq 1}
\] 
and equation \eqref{equ:step1} by Lemma
\ref{lem:non_dom}. 
%We finally get
%$\Omega_{j}^{1} \in \Delta_{\beta_{l+1},<} \cap \hat{W}_{j}$ by Lemma \ref{lem:prod}, which finishes this case.  

\strut

\noindent\textbf{Case (b).} That is, $m_k=n_{l+1}-1$, $m_j=0$. We know that
$\Omega_{j}^{0} \in \Delta_{\beta,<}$ for some $\beta < \beta_{l+1}$ and
$f_{l+1}^{n_{l+1} -1} \Omega_b \in \Delta_{\beta_{l+1},<}^{n_{l+1} -1}$. Lemma
\ref{lem:dom} implies that
$\Omega_{k}^{1} \in \Delta_{{\mathfrak n}(\beta_{l+1}),<}^{0} \cap \hat{M}_{k}$   and equation (\ref{equ:step1}).  

\strut

It remains to show the last two properties. We have
                $\nu_{\mathcal C} (\Omega_{j}^{0}) \in (e_{l+1})$ for any
                $1 \leq j \leq 2 \nu_{l+1}$.  This equality and equation (\ref{equ:step1})
                imply $\nu_{\mathcal C} (\Omega_{j}^{1}) \in (e_{l+1})$ for any
                $j \in {\mathcal L}_{1}$ and
                $\nu_{\mathcal C} (\Omega_{k}^{1}) - {\mathfrak n}(\beta_{l+1}) \in
                (e_{l+1})$ for any $k \in {\mathcal U}_{1}$.  By construction, we have
\[ [\nu_{\mathcal C}(\Omega_{j}^{1})]_{n}  = [\nu_{\mathcal C}(\Omega_{j}^{0})]_{n}   \neq 
[\nu_{\mathcal C}(\Omega_{j^{\prime}}^{0})]_{n}  = [\nu_{\mathcal C}(\Omega_{j^{\prime}}^{1})]_{n} \]
if $j \neq j^{\prime}$
and $j, j^{\prime} \in {\mathcal L}_{s}$.  Thus $\{ [\nu_{\mathcal C}(\Omega_{j}^{1})]_{n} : j \in {\mathcal L}_{1} \}$ and $[(e_{l+1})]_{n}$ are sets of
cardinal $\nu_{l+1}$ such that the former is included in the latter and as a consequence they coincide.
Analogously, $[\nu_{\mathcal C}(\Omega_{k}^{0})]_{n}  \neq [\nu_{\mathcal C}(\Omega_{k^{\prime}}^{0})]_{n}$ holds if $k \neq k^{\prime}$ and 
$k, k^{\prime} \in {\mathcal U}_1$. This property together with  equation (\ref{equ:step1}) imply that 
$\{ [\nu_{\mathcal C}(\Omega_{k}^{1}) - {\mathfrak n}(\beta_{l+1})]_{n} : k \in {\mathcal U}_{1} \}$ 
is a set of $\nu_{l+1}$ elements contained in $[(e_{l+1})]_{n}$  and hence they coincide.  
\end{proof} 
% \begin{rem}
 %\label{rem:positive} 
 %Let $\Omega_{k}^{0} = f_{l+1}^{r} \Omega_a$ and assume that there is no $\Omega_{j}^{0}$ with
 %$1 \leq j \leq 2 {\nu}_{l+1}$, $j < k$ and
 %$[\nu_{\mathcal{C}} (\Omega_j^{0})]_{n} = [\nu_{\mathcal{C}} (\Omega_k^{0})]_{n}$. 
 %Then either $\Omega_{a} \in \Delta_{\beta_{l+1},<}^{0}$ or $\Omega_{a} \in \Delta_{\beta,<}$
 %for some $\beta < \beta_{l+1}$.
 %In the former case, we have $\Omega_{k}^{1} = \Omega_{k}^{0} \in \Delta_{\beta_{l+1},<}^{r}$ by Corollary \ref{cor:f_omega}. 
 %In the latter case, we have $r=0$; otherwise 
 %there would exist
 %$\Omega_{j}^{0} = f_{l+1}^{r-1} \Omega_b$ with $1 \leq j < k$,
 %$\Omega_{b} \in \Delta_{\beta_{l+1},<}^{0}$ and
 %$[\nu_{\mathcal{C}} (\Omega_j^{0})]_{n} = [\nu_{\mathcal{C}} (\Omega_k^{0})]_{n}$ by Corollary
 %\ref{rem:all-modn-congruences}. Thus $\Omega_{k}^{1} = \Omega_{k}^{0} \in \Delta_{\beta,<}$ holds.
 %In any case, we obtain $\Omega_{k}^{1}= \Omega_{k}^{0} \in \cup_{\beta \leq \beta_{l+1}} \Delta_{\beta,<}$. 
 %\end{rem}
The step from stage $s$ to stage $s+1$ for $s>0$ is easier. Suppose that the families
$\mathcal{G}_{l+1}^{1},\ldots,\mathcal{G}_{l+1}^s$, with
$\mathcal{G}_{l+1}^r=(\Omega_j^r:j=1,\ldots, 2 {\nu}_{l+1})$ are good 
%and non-terminal,
for $1 \leq r \leq s$   and Property \eqref{cond:good_tech} holds.  
 \begin{defi}
 We define ${\mathcal U}_{s+1} = \{ 1 \leq k \leq 2 \nu_{l+1} : \Omega_{k}^{s+1} \neq \Omega_{k}^{s} \}$ and 
 ${\mathcal L}_{s+1}$ as its complement in $\{ 1, \ldots, 2 \nu_{l+1}\}$.
 \end{defi}
 \begin{pro}
 \label{pro:auxm}
        % Assume $N^{s} (\beta_{l+1}) \in  (e_{l+1})$. 
 The family  ${\mathcal G}_{l+1}^{s+1}$ satisfies
 \begin{itemize}
        \item Both ${\mathcal L}_{s+1}$ and ${\mathcal U}_{s+1}$ have $\nu_{l+1}$ elements.
 \item  $\Omega_{j}^{s+1} \in \cup_{\beta^{\prime} < {\mathfrak n}^{s+1} (\beta_{l+1})} \Delta_{\beta^{\prime},<} \cap \hat{M}_{j}$ 
 for any $j \in {\mathcal L}_{s+1}$.
 \item  $\Omega_{k}^{s+1} \in \Delta_{{\mathfrak n}^{s+1}(\beta_{l+1}),<}^{0} \cap \hat{M}_{k}$ for any $k \in {\mathcal U}_{s+1}$.
 \item $\nu_{\mathcal{C}} (\Omega_{k}^{s+1}) - \nu_{\mathcal{C}} (\Omega_{k}^{s}) = 
 {\mathfrak n}^{s+1}(\beta_{l+1}) - {\mathfrak n}^{s} (\beta_{l+1})$
 for any $k \in {\mathcal U}_{s+1}$.
 \end{itemize}
 \end{pro}
 \begin{proof}
 As $\mathcal{G}_{l+1}^{s}$ is good,
 $\Omega_{j}^{s} \in \cup_{\beta' \leq {\mathfrak n}^{s}(\beta_{l+1})} \Delta_{\beta',<}$ for any
 $1 \leq j \leq 2 {\nu}_{l+1}$.  Since ${\mathfrak n}^{s}(\beta_{l+1}) \in (e_{l+1})$ by hypothesis, any element  of $[(e_{l+1})]_{n}$ has
 exactly two representatives in $(\nu_{\mathcal C}(\Omega_{j}^{s}))_{1 \leq j \leq 2 \nu_{l+1}}$ by  property \ref{good7} for $s$. 
 Thus ${\mathcal L}_{s+1}$ and ${\mathcal U}_{s+1}$ have both $\nu_{l+1}$ elements.
 %There are no three forms in
 %$\mathcal{G}_{l+1}^{s}$ whose $\mathcal{C}$-orders are congruent
 %modulo $n$ by property \ref{good7} 
 %for $s$.
 
 Suppose  $\nu_{\mathcal{C}} (\Omega_{k}^{s})- \nu_{\mathcal{C}} (\Omega_{j}^{s}) = dn$ for some
 $d \in {\mathbb Z}_{\geq 0}$ and $1 \leq j < k \leq 2 {\nu}_{l+1}$.
 As $\nu_{\mathcal{C}} (\Omega_{j}^{s}) - \nu_{\mathcal{C}} (\Omega_{k}^{s}) \in (n)$,
 the $1$-forms $\Omega_{j}^{s}, \Omega_{k}^{s}$ have different leading
 variables, also by property \ref{good7} 
 %of good sequences 
 for $s$. As a consequence,
 \[
   \Omega_{j}^{s+1} = \Omega_{j}^{s} \in \cup_{\beta' < {\mathfrak n}^{s+1}(\beta_{l+1})}
   \Delta_{\beta',<} \quad \mathrm{and} \quad \Omega_{k}^{s+1} \in \Delta_{{\mathfrak
       n}^{s+1}(\beta_{l+1}),<}^{0} \cap \hat{M}_k
 \]
 by Lemma \ref{lem:dom}.  The last
 property also follows from Lemma \ref{lem:dom}.
\end{proof}  

We can now show that  being good is a hereditary property.
\begin{pro}
 \label{pro:good-to-good}
 Given $\mathcal{G}^{0}_{l+1}, \ldots, \mathcal{G}^{s}_{l+1}$, assume that $\mathcal{G}^{j}_{l+1}$ is good for any
 $ 1 \leq j \leq s$ and 
 %$N^{s} (\beta_{l+1}) \in  (e_{l+1})$.  
 Property \eqref{cond:good_tech} holds. Then $\mathcal{G}^{s+1}_{l+1}$ is good.
 %non-terminal. 
 % 
% \[ \{ \beta_{l+1}, N(\beta_{l+1}), \ldots, N^{s} (\beta_{l+1}) \} \subset  (e_{l+1}). \]
% Then $\mathcal{G}^{s+1}_{l+1}$ is good.
\end{pro}
%The condition $\{ N(\beta_{l+1}), \ldots, N^{s} (\beta_{l+1}) \} \subset  (e_{l+1})$ is always satisfied as we will see later on. It is included to make the following proof
%simpler.
\begin{proof} 
 Property \ref{good1} in Definition \ref{def:good} and property \ref{cond1} of \tidy{} families hold by construction. 
% Properties \eqref{good1} and \eqref{good2} in Definition \ref{def:good} holds by construction.

 Property \ref{good6} for $\mathcal{G}_{l+1}^{s+1}$ (for the leading variables
 and properties of the $c$-sequences) follows for $s=0$ applying
 Proposition \ref{pro:aux1}. 
 %and Remark \ref{rem:positive}.
 For $s>0$,
 apply Proposition \ref{pro:auxm} 
 %and Lemma \ref{lem:non_dom}, 
 using that $\mathcal{G}_{l+1}^s$ is good.

 Let us show  Property \ref{good7} for $\mathcal{G}_{l+1}^{s+1}$. It is a consequence of Proposition \ref{pro:aux1} if $s=0$ so we can assume $s \geq 1$.
 %Since $N^{s}(\beta_{l+1}) \in (e_{l+1})$, it follows that $\nu_{\mathcal C} (\Omega_{j}^{s}) \in (e_{l+1})$ for any $1 \leq j \leq 2 \nu_{l+1}$ by 
 %Property \ref{good7} for $\mathcal{G}_{l+1}^{s}$. 
 By construction,    we obtain 
 \[ \{ [\nu_{\mathcal C}(\Omega_{j}^{s+1})]_{n} : j \in {\mathcal L}_{s+1} \}  = \{ [\nu_{\mathcal C}(\Omega_{j}^{s})]_{n} : j \in {\mathcal L}_s \} = [(e_{l+1})]_{n} \]
 and 
 \[ \{ [\nu_{\mathcal C}(\Omega_{k}^{s}) ]_{n} : k \in {\mathcal U}_{s+1} \}  = 
\{ [\nu_{\mathcal C}(\Omega_{k}^{s}) ]_{n} : k \in {\mathcal U}_{s} \}  . \]
 This property, Proposition \ref{pro:auxm} and Property \ref{good7} for $\mathcal{G}_{l+1}^{s}$  imply
 \[ \{ [\nu_{\mathcal C}(\Omega_{k}^{s+1}) - {\mathfrak n}^{s+1}(\beta_{l+1})]_{n} : k \in {\mathcal U}_{s+1} \}  = 
 \{ [\nu_{\mathcal C}(\Omega_{k}^{s})  - {\mathfrak n}^{s}(\beta_{l+1}) ]_{n} : k \in {\mathcal U}_{s} \}  , \]
 concluding the proof of Property \ref{good7} for $\mathcal{G}_{l+1}^{s+1}$.

 Since ${\mathfrak n}^{s}(\beta_{l+1}) \in (e_{l+1})$, 
 it follows that $\nu_{\mathcal C} (\Omega_{j}^{s}) \in (e_{l+1})$ for any $1 \leq j \leq 2 \nu_{l+1}$ by 
 Property \ref{good7} for $\mathcal{G}_{l+1}^{s}$.  Since
\[
  \nu_{\mathcal{C}} (\Omega_{k}^{s}) - \nu_{\mathcal{C}} (\Omega_{j}^{s}) \in (e_{l+1}) \ \mathrm{and} \ 
  \nu_{\mathcal{C}} (\Omega_{k}^{s}) \geq  \nu_{\mathcal{C}} (\Omega_{j}^{s})
\] 
for all $1 \leq j \leq k \leq 2 {\nu}_{l+1}$, we obtain 
 \begin{equation}
 \label{equ:folga}
 \nu_{\mathcal{C}} (\Omega_{k}^{s}) - \nu_{\mathcal{C}} (\Omega_{j}^{s}) \geq
 \delta_{j,k} e_{l+1}
 \end{equation}
 where $\delta_{j,k} =0$ if $ \nu_{\mathcal{C}} (\Omega_{k}^{s}) - \nu_{\mathcal{C}} (\Omega_{j}^{s}) \in (n)$ and otherwise $\delta_{j,k}=1$.
 Since ${\mathfrak n}^{s+1}(\beta_{l+1}) - {\mathfrak n}^{s}(\beta_{l+1}) \leq e_{l+1}$, we get
 \begin{equation}
 \label{equ:alp_aux}
 \nu_{\mathcal{C}} (\Omega_{k}^{s+1}) \geq \nu_{\mathcal{C}} (\Omega_{k}^{s})
 \quad \mathrm{and} \quad 
 \nu_{\mathcal{C}} (\Omega_{j}^{s+1}) \leq \nu_{\mathcal C} (\Omega_{j}^{s}) + e_{l+1}
 \end{equation}
 and therefore equations \eqref{equ:folga} and \eqref{equ:alp_aux} imply
 \begin{equation}
 \label{equ:alpha}
 \nu_{\mathcal{C}} (\Omega_{k}^{s+1}) - \nu_{\mathcal{C}} (\Omega_{j}^{s+1}) \geq 0
 \end{equation}
 if $1 \leq j \leq k \leq 2 {\nu}_{l+1}$ and $\delta_{j,k} = 1$. Assume $\delta_{j,k}=0$.  
 Since
 \[ 
   \nu_{\mathcal{C}} (\Omega_{k}^{s+1}) > \nu_{\mathcal{C}} (\Omega_{k}^{s})
   \quad \mathrm{and} \quad 
 \nu_{\mathcal{C}} (\Omega_{j}^{s+1}) = \nu_{\mathcal{C}} (\Omega_{j}^{s}),
\]
%if $ \delta_{j,k}=0$, 
Equation \eqref{equ:folga} implies equation \eqref{equ:alpha}. 
%if $\delta_{j,k}=0$.  
Therefore, property \ref{cond3} holds, so that $\mathcal{G}_{l+1}^{s+1}$ is \tidy{}, and we are done. 
 \end{proof}

 Finally, we prove that after a finite number of stages one obtains a terminal family.
 \begin{pro}
 \label{pro:from-l-to-next}
 Suppose that $(\Omega_1, \hdots, \Omega_{2 {\nu}_{l}})$ is a terminal family of level
 $l$. Then there exists $s \geq 1$ such that
 $\mathcal{G}_{l+1}^s=(\Omega_{j}^{s}: j=1,\ldots, 2 {\nu}_{l+1})$ is  terminal
 of level $l+1$, with $ \Omega_{j}^{s} = \Omega_j$ for any
 $1 \leq j \leq 2 {\nu}_{l}$. Moreover, if $l+1 < g$ then
 ${\mathfrak n}^{s} (\beta_{l+1})=\beta_{l+2}$.
 \end{pro}
 \begin{proof}
 We need to verify the three properties of a terminal family (see Definition
 \ref{def:level}) for some $s \geq 1$. Consider a good family $\mathcal{G}_{l+1}^{s}$ of level $l+1$.
 Since any good family is \tidy{} by Lemma \ref{lem:no3} and Definition \ref{def:good}, it follows that 
 $\mathcal{G}_{l+1}^{s}$ satisfies condition \ref{ta}
 for terminal families if and only if satisfies condition \ref{cond4} in Definition \ref{def:nice}.% \ref{cond4}. 
 %in Definition \ref{def:nice}. 

 We divide the verification of this and the remaining properties in two cases: $l+1<g$ and $l+1=g$.
 
 \strut
 
 \textbf{Case $l+1<g$}. Let $s_{0}$ be such that ${\mathfrak n}^{s_0}(\beta_{l+1})=\beta_{l+2}$. 
 We can construct good families ${\mathcal G}_{l+1}^{1}, \ldots, {\mathcal G}_{l+1}^{s_{0}}$ by applying 
 $s_0$ times Proposition \ref{pro:good-to-good}. Suppose $s < s_0$. 
 Since ${\mathfrak n}^{s} (\beta_{l+1}) \in (e_{l+1})$,  it follows that 
  \begin{equation*}
  \left\{ [\nu_{\mathcal{C}} (\Omega_j^s)]_{n}\ |\ \Omega_j^s \in \mathcal{G}_{l+1}^s \right\} = [(e_{l+1})]_{n}
 \end{equation*}
 by Property \ref{good7} for $\mathcal{G}_{l+1}^{s}$
 and hence $\mathcal{G}_{l+1}^{s}$ is non-terminal because it is non-\nice{}.

 We claim that ${\mathcal G}_{l+1}^{s_0}$ is a terminal family. Since $\beta_{l+2} \not \in (e_{l+1})$, we deduce that 
 all the classes
 modulo $n$ in $\{[\nu_{\mathcal{C}} (\Omega_j^{s_0})]_n\ |\ j=1,\ldots, 2 {\nu}_{l+1}\}$
 are different by Property \ref{good7}. Therefore,   ${\mathcal G}_{l+1}^{s_0}$  is \nice{}, %gives condition \ref{cond4} 
 %in Definition \ref{def:nice}, 
 and thus
 condition \ref{ta} in the definition of \emph{terminal family} \ref{def:level} holds.

 By definition, the maximum ${\mathcal C}$-order of a $1$-form in a terminal family
 $\mathcal{T}_l$ of level $l$ is $\overline{\beta}_{l+1}$. Since a $1$-form $\Omega_k^{0}$
 of stage $0$ and level $l+1$ is of the form $f_{l+1}^{a} \Omega$ where
 $0 \leq a < n_{l+1}$ and $\Omega \in \mathcal{T}_l$, we obtain
 \[
   \nu_{\mathcal{C} } (\Omega_k^{0}) \leq \overline{\beta}_{l+1}+ (n_{l+1} -1)
   \overline{\beta}_{l+1} = n_{l+1} \overline{\beta}_{l+1}
 \]
 for any $1 \leq k \leq 2 {\nu}_{l+1}$. As ${\mathfrak n}^{s_0}(\beta_{l+1})=\beta_{l+2}$, the construction gives
 \[
   \nu_{\mathcal{C}} (\Omega_k^{s_0}) \leq \nu_{\mathcal{C}} (\Omega_k^0) +
   (\beta_{l+2}-\beta_{l+1})
   \leq n_{l+1} \overline{\beta}_{l+1} + (\beta_{l+2}-\beta_{l+1}) =
   \overline{\beta}_{l+2}
 \]
 for any $1 \leq k \leq 2 {\nu}_{l+1}$ by equation \eqref{equ:rec_beta}. Thus
 ${\mathcal G}_{l+1}^{s_0}$ satisfies condition \ref{tb}.
 %of terminal families. 
 Finally,
 condition \ref{tc} 
 %of terminal families 
 is a consequence of property \ref{good6}.
 
 \strut

 %of good families. 
 \textbf{Case $l+1=g$}. Since $e_{g}=1$, we obtain good families ${\mathcal G}_{g}^{1}, {\mathcal G}_{g}^{2}, \ldots$ 
 by iterative applications of Proposition \ref{pro:good-to-good}.
 Let ${\mathcal G}_{g}^{s}$ be a good family.
 Condition \ref{tb} is trivially satisfied.
 Let $\mathcal{B}_s$ be the set of congruences $[x]_n$ modulo
 $n$ such that there exist $\Omega^{s}_j, \Omega_k^s$ with
 $[\nu_{\mathcal{C}} (\Omega_j^s)]_n=[\nu_{\mathcal{C}} (\Omega_k^s)]_{n}=[x]_n$ and
 $j < k \leq n$.  %Assume $s \geq 1$. 
 We claim that $\mathcal{G}_{g}^{s}$ is terminal if and only if
 $\mathcal{B}_s = \emptyset$. Let us show the non-trivial implication, assume  $\mathcal{B}_s = \emptyset$. 
 Since this implies 
  \begin{equation*}
 \# \left\{ [\nu_{\mathcal{C}} (\Omega_j^{s})]_{n}
 \ |\
 1 \leq j \leq n \right\} = n,
 \end{equation*}
 ${\mathcal G}_{g}^{s}$ is separated.
 Condition \ref{tc} is again a consequence of  property
 \ref{good6}.
 %for $s \geq 1$.  Assume $s=0$. Since $\Omega_{2 \nu_{g-1} +1}^{0} = f_{g} dx$ satisfies 
 %$\nu_{\mathcal C} (\Omega_{2 \nu_{g-1} +1}^{0}) \in \beta_{g} + (e_{g-1}) $, it is contained in 
 %$\{ [\nu_{\mathcal{C}} (\Omega_j)]_{n}\ |\ \Omega_{j} \in  {\mathcal T}_{g-1}\}$. 
 %As a consequence,  $\mathcal{B}_{0} = \emptyset$
 %implies $n = 2 \nu_{g-1}$, i.e. $n_{g}=2$. In such a case, we obtain 
 %$\tilde{\mathcal T}_{g} = {\mathcal T}_{g-1} = \{ \Omega_{1}^{0}, \hdots, \Omega_{n}^{0} \}$
 %and  Condition \ref{tc} follows.

  Let us conclude by proving that there is some $s_0 \geq 1$ such that  $\mathcal{B}_{s_{0}} = \emptyset$.
  In this case, ${\mathfrak n}^s(\beta_{g})=\beta_{g}+s$ since $e_g =1$. 
  The construction implies
 \begin{equation*}
 \begin{split}
 \{ [\nu_{\mathcal{C}} (\Omega_j^{s+1})]_{n}\ |\ j=1,\ldots, n\} &=\\
 \{ [\nu_{\mathcal{C}}(\Omega_j^{s})]_n\ &|\ j=1,\ldots,n \} \cup
 \{ [x]_n + [1]_n\ |\ [x]_n\in \mathcal{B}_s\}.
 \end{split}
 \end{equation*}
 A finiteness argument shows that there must be an $s_0 \geq 1$ such that $\mathcal{B}_{s_0}=\emptyset$, and we are done.
 \end{proof}

 \subsection{The terminal family of level $g$}
 Thus, beginning with the $1$-forms $\Omega_1 = dx$ and $\Omega_2 = d f_1$ of level $0$,
 we obtain a terminal family of $1$-forms $\mathbf{\Omega}=(\Omega_1, \hdots, \Omega_{2n})$ of level
 $g$ by applying $g$ times Proposition \ref{pro:from-l-to-next}.  
 The next result implies Theorem \ref{teo:nice} by Remark \ref{rem:nice_suff}. 
 \begin{teo} 
 \label{teo:exists_nice}
 The family $\mathbf{\Omega}=(\Omega_1, \hdots, \Omega_{n})$ is a nice basis for ${\mathcal C}$.
 \end{teo}
 \begin{proof}
 Consider $\Gamma \in {\mathcal C}$.
 % By Weierstrass preparation theorem, there exists 
% an irreducible $f \in {\mathbb C} \{x,y\}$ such that $\Gamma = \{f=0\}$ and $f$ is a polynomial of the form 
 %$f = y^{n} + \sum_{l=0}^{n-1}  h_{l} (x) y^{l}$. 
 Given $\hat{f} \in {\mathbb C}[[x,y]]$, there exists
 $R = \sum_{l=0}^{n-1} \kappa_{l}(x) y^{l}$ in ${\mathbb C}[[x]][y]$ such that
 $\hat{f} - R \in (f_{g+1, \Gamma})$ (cf. subsection \ref{subsec:roots}) by the Weierstrass
 division theorem. It follows that
 \begin{equation}
 \label{equ:weierstrass}
 \hat{\Omega} \cn{2}  =  {\mathcal M}_{2 \nu_{g}} + {\mathcal I}_{\Gamma} 
 \end{equation}
 for any $\Gamma \in {\mathcal C}$, cf. Definitions \ref{def:modules} and \ref{def:ideal_g}.
 
%   Consider a terminal family $\mathcal{T}_{g} = \{ \Omega_1, \hdots, \Omega_{2n} \}$ of level $g$.
 Assume, from now on,  that $\Gamma \in {\mathcal C}$ is generic. Since 
\[ \sum_{l=1}^{2n} \mathbb{C}[[x]] \Omega_{l, \Gamma}   =  {\mathcal M}_{2 \nu_{g}} \]
by equation \eqref{equ:submodules},
we deduce $\hat{\Omega} \cn{2}  = \sum_{l=1}^{2n} \mathbb{C}[[x]] \Omega_{l, \Gamma} + {\mathcal I}_{\Gamma}$ by equation \eqref{equ:weierstrass}. 
Since $\mathbf{\Omega}$ is separated,  
 it is easy to see that 
 \[ \Omega_{j, \Gamma}  \in  \sum_{l=1}^{n} \mathbb{C}[[x]] \Omega_{l, \Gamma} + {\mathcal I}_{\Gamma} \] 
 for any $n < j \leq 2n$ and hence 
\[ \sum_{l=1}^{n} \mathbb{C}[[x]] \Omega_{l, \Gamma} + {\mathcal I}_{\Gamma} = 
\sum_{l=1}^{2n} \mathbb{C}[[x]] \Omega_{l, \Gamma} + {\mathcal I}_{\Gamma}  = \hat{\Omega} \cn{2}   \]
for generic $\Gamma \in \mathcal{C}$.  Therefore, $({\Omega}_{1}, \ldots, {\Omega}_{n})$ is a nice basis for ${\mathcal C}$. 
 \end{proof}
\begin{pro}
\label{pro:omega_max}
Let $\mathbf{\Omega}=(\Omega_1,\ldots,\Omega_n)$ be a nice basis for ${\mathcal C}$
and $1\leq k \leq n$. Then
\[
 \nu_{\mathcal{C}} (\Omega_k) = \max \{ \nu_{\mathcal{C}} (\Omega') : \Omega' \in \hat{M}_{k} \} .
\]
\end{pro}
\begin{proof}
 By construction, given $\Omega'\in \hat{M}_k$ there exists $P \in \mathbb{C}{[a_{\beta}]}_{\beta \in \mathcal{E}(\mathcal{C})} \setminus \{0\}$
 such that $P \Omega'$ is of the form
 \[ c_k \Omega_k + \sum_{j=1}^{k-1} c_{j} \Omega_j , \] 
 where $c_k \in \mathbb{C}{[a_{\beta}]}_{\beta \in \mathcal{E}(\mathcal{C}) } \setminus \{0\}$ and 
 $c_j \in \mathbb{C}{[a_{\beta}]}_{\beta \in \mathcal{E}(\mathcal{C})} [[x]]$ for any $1 \leq j < k$.
 Since $\nu_{\mathcal{C}} (c_b \Omega_b) - \nu_{\mathcal{C}} (c_a \Omega_a) \not \in (n)$ for all
 $1 \leq a < b \leq k$, it follows that
\[ \nu_{\mathcal{C}} (\Omega') = \nu_{\mathcal{C}} (P \Omega') = 
 \min_{1 \leq j \leq k, \ c_j \not \equiv 0} \nu_{\mathcal{C}} (c_j \Omega_j) \leq \nu_{\mathcal{C}} (c_k \Omega_k) = \nu_{\mathcal{C}} (\Omega_k) . \]
Since $\Omega_k \in \hat{M}_k$, we are done.
\end{proof}
\begin{rem}
The analogue result also holds for terminal families of any level.
\end{rem}
\begin{cor}
\label{cor:ind_ord}
The $n$-tuple $( \nu_{\mathcal{C}} (\Omega_1), \ldots, \nu_{\mathcal{C}} (\Omega_n))$ does not depend on the nice basis 
$(\Omega_1, \ldots, \Omega_n)$ for $\mathcal{C}$. 
\end{cor}
\begin{rem}
 Consider a nice basis $(\Omega_1, \ldots, \Omega_n)$ for $\mathcal{C}$.
 Let $l < g$. We have $d f_{l+1} \in \hat{M}_{2 \nu_l} \cap  \Delta_{\beta_{l+1}, <}^{0}$ and
 $\nu_{\mathcal{C}} (d f_{l+1} ) = \overline{\beta}_{l+1}$. Note that 
 $\nu_{\mathcal{C}} (\Omega_{2 \nu_{l}}) \leq \overline{\beta}_{l+1}$ by the properties of the
 $1$-forms of level $l$ and Corollary \ref{cor:ind_ord}. Since
 \[
 \nu_{\mathcal{C}} (\Omega_{2 \nu_{l}}) = \max \{ \nu_{\mathcal{C}} (\Omega' ): \Omega' \in
 \hat{M}_{2 \nu_{l}} \} \geq \nu_{\mathcal{C}} (d f_{l+1} ) = \overline{\beta}_{l+1}
 \] by
 Proposition \ref{pro:omega_max}, we obtain
 $\nu_{\mathcal{C}} (\Omega_{2 \nu_{l}}) = \overline{\beta}_{l+1}$. Thus, we could replace
 $\Omega_{2 \nu_{l}}$ with $d f_{l+1}$ in the nice basis.
\end{rem}
%\begin{rem}
%\label{rem:algo}
%Since $\nu_{\mathcal{C}}(f_l) = \overline{\beta}_l$ for $1 \leq l \leq g$, it is easy to build an algorithm that provides
%$\Lambda_{\Gamma}$, for generic $\Gamma \in \mathcal{C}$, and whose input is $n, \beta_1, \ldots, \beta_g$, or 
%equivalently  $n, \overline{\beta}_1, \ldots, \overline{\beta}_g$, by Remark \ref{rem:rich}.
%\end{rem}
\section{Case $E \neq \emptyset$} \label{sec:E-non-empty}
We now study the equisingularity class $\mathcal{C}$ of a pair $(\Gamma, E)$ where
$\Gamma$ is a germ of irreducible curve and $E \subset \{xy=0\}$ is a nonempty normal
crossings divisor.  We give a procedure to construct a basis of the module of
K\"ahler differentials preserving $E$ on any generic element in ${\mathcal C}$ which will
allow us to provide a constructive proof of Genzmer's formula \cite{Genzmer-moduli-2020} for the dimension of the
moduli space. This procedure is quite similar to the one for $E=\emptyset$, so that we shall only indicate the differences.

\subsection{Setting for the case $E \neq \emptyset$}
In order to study $\Lambda_{\mathcal C}^E$ (Definition \ref{def:class}), we need to
adapt the technique of the preciding section, as the previously defined $\Omega_1$ and
$\Omega_2$ no longer need to leave $E$ invariant.
% \footnote{No se puede asumir que $\Gamma$ sea transversal a $x=0$, porque esta propiedad no es preservada por
% explosiones. El metodo tiene que ser empleado tanto para un {\it suitable pair} $(\Gamma, E)$ como para el transformado $(\Gamma, E_1)$ para encontrar  bases 
% paralelas.}
%\pedro{}We consider the same local system of coordinates $(x,y)$ as in the previous case, recalling that $\Gamma$ is assumed not tangent to $x=0$.\pedro{}
\begin{defi}\label{def:adapted-levels}
Let $E = \{x^{\delta_x} y^{\delta_y}=0\}$ be a non-empty normal-crossings divisor
where $\delta_x, \delta_y \in \{0,1\}$.  We define:
\begin{itemize}
\item If $E = \{x=0\}$, then
\begin{equation*}
\overline{\Omega}_{1} = dx,  \quad \overline{\Omega}_2=x d y.
\end{equation*}
\item Else $\{ y=0 \} \subset E$, and then
\begin{equation*}
\overline{\Omega}_1 = x^{\delta_x}dy,\quad \overline{\Omega}_2 = ydx.
\end{equation*}
\end{itemize}
Given $\overline{\Omega}_1$ and $\overline{\Omega}_2$, we define for $j\geq 1$:
\begin{equation*}
\overline{\Omega}_{2j+1} = y^j\overline{\Omega}_1, \quad
\overline{\Omega}_{2j+2} = y^j\overline{\Omega}_2.
\end{equation*}
Finally, $M_{0}^{E}=\{0\}$ and, recursively:
\begin{equation*}
M_{j+1}^{E}=\mathbb{C}[a_{\beta}]_{\beta\in\mathcal{E}(\mathcal{C})}[[x]] \overline{\Omega}_{j+1}
+ M_{j}^{E},\
\hat{M}_{j+1}^{E}=(\mathbb{C}[a_{\beta}]_{\beta\in\mathcal{E}(\mathcal{C})}
\setminus \{0\})
\overline{\Omega}_{j+1} + M_{j}^{E}.
\end{equation*}
\end{defi}
\begin{defi}
The set of $1$-forms of $E$-level $l$ is now defined as $M_{2 {\nu}_l}^{E}$ for
$l\in\{0,\ldots,g\}$. Let $\mathcal{I}_{E,\Gamma}$ be the set of $1$-forms
$\Omega\in \hat{\Omega}_E(\mathbb{C}^2,0)$ such that $\Gamma^{\ast}\Omega\equiv 0$. 
\end{defi}
We need to refine the definition of \emph{\tidy} and \emph{\nice} family:
\begin{defi}\label{def:E-tidy-nice}
A family $\mathbf{\Omega}=(\Omega_1,\ldots,\Omega_{r})$ in
$\Omega_{E,\mathcal{C}}(\mathbb{C}^{2},0)$, with $r \in \{ 2 \nu_{l}, \min (2 \nu_{l}, n) \}$,  
is \emph{$E$-\tidy{}} of level $0\leq l \leq g$
for $\Gamma$ if:
\begin{itemize}
        \item[$E$-(i)] $\Omega_j\in \hat{M}_{j}^{E}$ for $1\leq j \leq r$;
          %this is
%the same as the previous condition replacing $\hat{W}_j$ with $\hat{W}_{j}^{E}$;
%\item[$E$-(ii)] $\nu(\Omega_j) = \nu(\overline{\Omega}_{j})$ for
%$1\leq j \leq \overline{\nu}_j$. Notice that in this case,
%\begin{equation*}
%\nu(\overline{\Omega}_{2j+1}) =  \quad  \nu(\overline{\Omega}_{1}) + j  \quad \mathrm{and} \quad 
%\nu(\overline{\Omega}_{2j+2}) = \nu(\overline{\Omega}_{2}) + j;
%\end{equation*}
%\item[$E$-(ii)]
%$\nu_C(\Omega_k)-\nu_C(\Omega_j)\geq
%(\nu(\overline{\Omega}_k)-\nu(\overline{\Omega}_j))n$ for
%$1\leq j \leq k \leq \tilde{\nu}_l$; the same as in Definition \ref{def:nice}.
\item[$E$-(ii)]
$\nu_C(\Omega_k) \geq \nu_C(\Omega_j)$ for
$1\leq j \leq k \leq r$.% the same as in Definition \ref{def:nice}.
\end{itemize}
Moreover, $\mathbf{\Omega}$ is \emph{$E$-\nice{}} of level $l$ if it is \tidy{} and
\begin{itemize}
\item[$E$-(iii)]
$\nu_{\mathcal{C}}(\Omega_1),\ldots,\nu_{\mathcal{C}}(\Omega_{\min(2 \nu_{l},n)})$
define pairwise different classes modulo $n$.
\end{itemize}
Finally, for $l=g$ and $r=n$, $\mathbf{\Omega}$ is an \emph{$E$-nice  basis} if, moreover,
\begin{itemize}
\item[$E$-(iv)] $\hat{\Omega}_{E}(\mathbb{C}^2,0)=
\mathbb{C}[[x]]\Omega_{1,\Gamma}+\cdots+\mathbb{C}[[x]]\Omega_{n,\Gamma}+
\mathcal{I}_{E,\Gamma}$ for generic $\Gamma\in \mathcal{C}$.
\end{itemize}
      \end{defi}
      Notice that the only modification is to require that $E$ is invariant, in $E$-(i) and $E$-(ii).

      The definitions of \emph{terminal} and \emph{good} need also adapting to $E$. For technical reasons which will become clear shortly, instead of starting the construction at level $l=0$, we need to set an \emph{initial level} $l_0$, with $l_0\in\{0,1\}$, and we assume, for the moment, that there exists, for that $l_0$, a family satisfying the following conditions:
\begin{defi}\label{def:E-terminal}
A family
$\mathcal{T}_l= (\Omega_1,\ldots,\Omega_{2 {\nu}_l}) \subset
\Omega_{E,\mathcal{C}}(\mathbb{C}^2,0)$, with $l_0 \leq l \leq g$, is \emph{$E$-terminal} if
\begin{itemize}
\item[$E$-(a)] $\mathcal{T}_l$ is an $E$-\nice{} family of level $l$;
\item[$E$-(b)] On one hand
$\nu_{\mathcal{C}}(\Omega_{2 {\nu}_{l_{0}}})\leq n  + \overline{\beta}_{l_0+1}$ if $l_0 < g$, 
and on the other,
\begin{equation*}
\nu_{\mathcal{C}}(\Omega_{2 {\nu}_{l}}) -
\nu_{\mathcal{C}}(\Omega_{2 {\nu}_{l-1}}) \leq
\overline{\beta}_{l+1}-\overline{\beta}_l, \ \ \mathrm{if} \ l_0 <  l < g,
\end{equation*}
(this inequality the same as in Definition \ref{def:level} except for the initial level
$l_0$).
\item[$E$-(c)] The same as (c) in Definition \ref{def:level}:   $\mathcal{T}_{l} \subset \cup_{\beta' \in {\mathcal E}({\mathcal C})} \Delta_{\beta', <}$.
Moreover, we have
$\mathcal{T}_{l}  \subset \Delta_{\beta_{l+1}, <}^{0} \cup \cup_{\beta' \in {\mathcal E}({\mathcal C}), \ \beta' < \beta_{l+1}} \Delta_{\beta', <}$
for $0 \leq l < g$. 
\end{itemize}
      \end{defi}      
      \begin{defi}
        Given an $E$-terminal family $\mathcal{T}_l= (\Omega_1,\ldots,\Omega_{2 {\nu}_l})$ of level $0 \leq l < g$, we define 
        an $E$-initial family ${\mathcal G}_{l+1}^{0} = (\Omega_{1}^{0}, \hdots, \Omega_{2 \nu_{l+1}}^{0})$  of stage $0$ (in level $l+1$) 
        by means of equation \eqref{equ:stage}, exactly as in the case $E=\emptyset$.  
        %by the formula $\Omega_{2 a \nu_{l} + j}^{0} =f_{l+1}^{a} \Omega_{j}$ for $0 \leq a < n_{l+1}$ and $1 \leq j \leq 2 \nu_{l}$.
      \end{defi}
Let $\mathcal{G}_{l+1}^s=(\Omega_1^s,\ldots,\Omega^{s}_{2 {\nu}_{l+1}})$ be given with 
$\Omega_i^s\in \Omega_{E,\mathcal{C}}(\mathbb{C}^2,0)$, for $0\leq l <g$ and $s\geq 1$.

\begin{defi}\label{def:E-good}
  The family $\mathcal{G}_{l+1}^s$ is $E$-\emph{good} (in level $l+1$, but we omit this)
  if the following properties hold:
  \begin{enumerate}[label=(\Roman*)]
  \item \label{E-good1} The first $2 {\nu}_{l}$ elements form an $E$-terminal family
    of level $l$. That is, there is a terminal family ${\mathcal T}_l = (\Omega_1, \ldots, \Omega_{2 {\nu}_l})$ of level $l$,
    such that $\Omega_{j}^{s} = \Omega_j$ for any $1 \leq j \leq 2 {\nu}_{l}$.
  \item \label{E-good2} $(\Omega_{1}^{s}, \hdots, \Omega_{2 {\nu}_{l+1}}^{s})$ is
    $E$-\tidy{} of level $l+1$.
  \item \label{E-good-rest} Finally, conditions (III) and (IV)   for good families hold
    for $\mathcal{G}^s_{l+1}$ (Definition \ref{def:good}).
\end{enumerate}
\end{defi}
\subsection{The method for the case $E \neq \emptyset$}
With the above definitions, all the results in subsection \ref{subsec:stos+1} carry over
adding the prefix $E$ to each \tidy{}, \nice{}, terminal and good family, as well as for
a nice basis, as they are only dependent on the notion of leading variable, which has not
changed, and the conditions we have adapted have no bearing on the computations modulo
$n$. However, we need to bootstrap the process at level $l_0$, which is the foundational
step and is still missing. This is what we carry out in what follows.

When $E=\emptyset$, the terminal family of level $0$ was ${\mathcal T}_0 = ( dx, d f_1 )$. From
this family, applying the results in subsections \ref{subsec:ll+1} and \ref{subsec:stos+1}
we ended up with the terminal family ${\mathcal T}_g$.

When $E \neq \emptyset$, the construction is totally similar: given an $E$-terminal family
$\mathcal{T}_l$ of level $l_0 \leq l<g$, we define
${\mathcal G}_{l+1}^{0} = (\Omega_1^0,\ldots,\Omega^{0}_{2 {\nu}_{l+1}})$ as in Definition \ref{def:set-stage0}.
%by means of Equation \eqref{equ:stage}, exactly as in the case $E=\emptyset$.  
We have
\[
  \nu_{\mathcal C} (\Omega_{2 {\nu}_l +1}^{0}) \geq
  \nu_{\mathcal C} (f_{l+1} dx)  =  n  + \overline{\beta}_{l+1}
  \geq 
  \nu_{\mathcal C} (\Omega_{2 {\nu}_l})  = \nu_{\mathcal C} (\Omega_{2 {\nu}_l}^{0})
\]
and analogously we obtain
$ \nu_{\mathcal C} (\Omega_{2 k {\nu}_l +1}^{0}) \geq \nu_{\mathcal C} (\Omega_{2 k {\nu}_l}^{0})$ 
for any $k \geq 1$.  As a consequence, ${\mathcal G}_{l+1}^{0}$  is an \tidy{} family. 
%satisfies $E$-(ii). The family ${\mathcal G}_{l+1}^{0}$ satisfies the other properties of
%$E$-good families of level $l+1$ by construction. 
Then we just apply the results in
subsection \ref{subsec:stos+1} to obtain an $E$-terminal family of level $l+1$. By
repeating this process, we end up with an $E$-terminal family of level $g$ and hence an
$E$-nice basis.
\begin{rem}
\label{rem:suffice}
Thus, we only need to find an $E$-terminal family of some level $l_0$. Actually, it
suffices to find an $E$-good family ${\mathcal G}_{l_0}^{1}$ (with $E$-(I) being empty)
since in such a case the results in subsection \ref{subsec:stos+1} provide an $E$-terminal
family of level $l_0$.
\end{rem} 
The rest of this section is devoted to providing this initial $E$-terminal family. There are three cases to consider: $\delta_y=0$, so that $E=\{x=0\}$; $\delta_y=1$ with $\beta_0\in (n)$; and $\delta_y=1$ with $\beta_0\not\in(n)$.
\subsubsection{Case $E=\{x=0\}$}
In this case, we set:
\begin{equation*}
\Omega_1= dx,\quad
\Omega_2 = xd f_1.
\end{equation*}
which satisfy $\nu (\Omega_1)=0$, $\nu_{\mathcal C}(\Omega_1)=n$, $\nu (\Omega_2)=1$,
$\nu_{\mathcal C}(\Omega_2)=\beta_1+n$.  The family $\mathcal{T}_0= (\Omega_1,\Omega_2 )$
is clearly $E$-terminal of level $l_0=0$ and thus, we obtain an $E$-nice basis $\mathbf{\Omega}$ by Remark \ref{rem:suffice}.
\subsubsection{Case $\delta_y =1$ and $\beta_0 \in (n)$}
\label{subsub:deltay}
Notice that $\beta_0\in (n)$ implies $\beta_0<\beta_1$. We have $\overline{\Omega}_1 = x^{\delta_x} dy$ and $\overline{\Omega}_2 = y dx$ and hence
$\nu_C(\overline{\Omega}_1) \in (n)$, $\nu_C(\overline{\Omega}_2) \in (n)$.  We
set $l_0=0$, $\Omega_{1}^{0} = \overline{\Omega}_1$ and 
$\Omega_{2}^{0} = \overline{\Omega}_2$. Applying
Lemma \ref{lem:non_dom}, with $\beta_{l+1} = \beta_0$,  
%and Proposition \ref{pro:aux1},
we obtain
$\Omega_{1}^{1} = \Omega_{1}^{0}$ and $\Omega_{2}^{1} \in \Delta_{{\mathfrak n}(\beta_0), <}^{0}$
where $\nu_{\mathcal C} (\Omega_{2}^{1}) = n + {\mathfrak n}(\beta_0)$.  The family
${\mathcal G}_{0}^{1} = ( \Omega_{1}^{1}, \Omega_{2}^{1} )$ is an $E$-good family of
level $0$.  Moreover, it satisfies $\nu (\Omega_{1}^{1})=\delta_x$ and
$\nu (\Omega_{2}^{1})=1$.  Remark \ref{rem:suffice} guarantees that we can find an $E$-nice basis starting with $\mathcal{G}_0^1$. 
%
%By sucessive applications of Lemma \ref{lem:dom}
%and Proposition \ref{pro:auxm}, 
%we obtain  $\Omega_{2}^{s} \in \Delta_{\beta_1, <}^{0}$
%and $\nu_{\mathcal C} (\Omega_{2}^{s}) = n + \beta_1$ for some $s \in \mathbb{N}$.
%We define $\Omega_1 = \Omega_{1}^{s} = \overline{\Omega}_1$, $\Omega_2 = \Omega_{2}^{s}$ and $l_0 = 0$.
%Again $\mathcal{T}_0=\{\Omega_1,\Omega_2\}$ is $E$-terminal of level $l_0=0$ and we proceed as in the 
%previous case. 
\subsubsection{Case $\delta_y =1$ and $\beta_0 \not \in (n)$}
\label{subsub:notinn}
Notice that $\beta_0 = \beta_1$ in this case. Recall Definition
\ref{def:adapted-levels} for $\overline{\Omega}_1, \overline{\Omega}_2$ and
$\overline{\Omega}_j$. As $\overline{\Omega}_1 = x^{\delta_x} dy$ and
$\overline{\Omega}_2 = y dx$, we have
$[\nu_C(\overline{\Omega}_1)]_n = [\nu_C(\overline{\Omega}_2)]_n = [\beta_1]_n$, so that
$\{ \overline{\Omega}_{1}, \overline{\Omega}_{2} \}$ is not an $E$-\nice{} family of level
$0$. Using Lemma \ref{lem:non_dom} and Proposition \ref{pro:aux1} we obtain
$\Omega_{1}^{1} = \overline{\Omega}_{1}$ and
$\Omega_{2}^{1} \in \Delta_{{\mathfrak n}(\beta_1), <}^{0}$ with
$\nu_{\mathcal C} (\Omega_{2}^{1}) = n + {\mathfrak n}(\beta_1)$. This family
${\mathcal G}_{0}^{1} = ( \Omega_{1}^{1}, \Omega_{2}^{1} )$ is $E$-\nice{} of level $0$
but does not satisfy $E$-(b) or $E$-(c) in Definition \ref{def:E-terminal}, so it is not
an $E$-terminal family of level $0$.  This is why, in this case, we need to start at
$l_0 =1$. By Remark \ref{rem:suffice}, we just need to provide an $E$-good family of level
$1$.

Consider the $E$-\tidy{} family
${\mathcal G}_{1}^{0} = ( \Omega_{1}^{0}, \hdots, \Omega_{2 {\nu}_1}^{0} )$ given by
$\Omega_{j}^{0} = \overline{\Omega}_j$ for $1 \leq j \leq 2 {\nu}_1$, which satisfies
$\nu (\Omega_{j}^{0}) = \lfloor \frac{j + \delta_x}{2} \rfloor$ for
$1 \leq j \leq 2 {\nu}_1$.  Note that
\begin{equation}
  \label{equ:level1} 
  [\nu_C(\overline{\Omega}_{2j+1})]_n =  [\nu_C(\overline{\Omega}_{2j+2})]_n
  =  [(j+1)\beta_1]_n. 
\end{equation}
if $1 \leq 2j+1 < 2j+2 \leq 2 {\nu}_1$. All congruences
$[\beta_1]_n, [2 \beta_1]_n, \hdots, [\nu_1 \beta_1]_n$ are pairwise different, so the unique
situation in which two $1$-forms have congruent ${\mathcal C}$-orders is described in
equation \eqref{equ:level1}.  Since $dy \in \tilde{\Delta}_{\beta_1, <}^{0}$, we have
$\Omega_{2j+1}^{0} \in \tilde{\Delta}_{\beta_1, <}$ by Corollary \ref{cor:f_omega}. 
We also have $\Omega_{2j+2}^{0} \in \tilde{\Delta}_{\beta_1, \leq}$ by Corollary \ref{cor:f_dominates_omega}.
We apply Lemma \ref{lem:non_dom} and Proposition
\ref{pro:aux1} to obtain
\[
  \Omega_{2j+1}^{1} = \Omega_{2j+1}^{0} \in \Delta_{\beta_1, <}   \quad
  \mathrm{and} \quad
  \nu_{\mathcal C} (\Omega_{2j+1}^{1}) = n \delta_x + (j+1) \beta_1
\] 
for $2j+1 \leq 2 {\nu}_1$; we also get 
\[
  \Omega_{2j+2}^{1} \in \Delta_{{\mathfrak n}(\beta_1), <}^{0}, \ \
  \nu (\Omega_{2j+2}^{1}) = j+1 \
  \ \mathrm{and} \ \
  \nu_{\mathcal C} (\Omega_{2j+2}^{1}) = n + j \beta_1 + {\mathfrak n}(\beta_1)
\]
for $2j+2 \leq 2 {\nu}_1$.  As a consequence, the family
${\mathcal G}_{1}^{1} = ( \Omega_{1}^{1}, \hdots, \Omega_{2 {\nu}_1}^{1} )$ is
$E$-good of level $1$ (with $E$-(I) being empty). Note that ${\mathcal L}_1$ is the subset
of odd numbers of $\{1, \ldots, 2 \nu_1\}$ whereas ${\mathcal U}_1$ is its subset of even numbers.
Moreover, we obtain 
$\nu (\Omega_{j}^{1}) = \lfloor \frac{j + \delta_x}{2} \rfloor$ for
$1 \leq j \leq  2 {\nu}_1$ by construction. As stated above, Remark \ref{rem:suffice} ends our
argument.

\subsection{Existence of a nice $E$-basis}
The previous discussion provides the proof for the next result.
\begin{teo}\label{teo:nice-E-basis}
  In any case, starting with the $E$-initial family of level
  $l_0\in \left\{ 0,1 \right\}$, the procedure of section \ref{sec:terminal} provides an
  $E$-terminal family and hence a nice $E$-basis
  $\mathbf{\Omega}=(\Omega_1,\ldots,\Omega_n)$ for the equisingularity class
  ${\mathcal C}$ of $(\Gamma, E)$.
\end{teo}
\begin{rem}
The analogue of Proposition \ref{pro:omega_max} holds in the general case, i.e. 
\[
  \nu_{\mathcal{C}} (\Omega_k) =
  \max \{ \nu_{\mathcal{C}} (\Omega') : \Omega' \in \hat{M}_{k}^{E} \} .
\]
The proof is the same.
\end{rem}
%
%\begin{rem}
%  The set $\Lambda_{\mathcal C}^{E}$ (Definition \ref{def:class}) depends on
%  $n, \beta_0, \beta_1, \hdots, \beta_g$ and indeed it is straightforward to write an
%  algorithm whose input is $(\delta_x, \delta_y, n, \beta_0, \hdots, \beta_g)$ and whose
%  output is $\Lambda_{\mathcal C}^{E}$.  The argument is analogous to the one in Remark
%  \ref{rem:algo}.   
%\end{rem}
%
\subsection{An explicit algorithm}
As a byproduct of our work, we 
present Algorithm \ref{alg:orders} in pseudo-code which, 
given the sequence $(n, \beta_1,\ldots,\beta_g)$ of the multiplicity and the Puiseux exponents of an equisingularity class ${\mathcal C}$, 
returns the orders $a(1),a(2),\ldots,a(n)$ corresponding to the nice (Apery) basis, $a(i)=\nu_{{\mathcal C}}(\Omega_i)$, for $i=1,\ldots,n$.
In particular, the algorithm provides the generic semimodule $\Lambda_{\mathcal C}$ and gives Theorem \ref{teo:algorithm}.
Indeed, the powerful properties of our construction allow to describe $\Lambda_{\mathcal C}$ without calculating explicitly a nice basis. 
Notice how, despite the
elaborate nature of the previous arguments, 
the algorithm is straightforward and is coding very simple. 
Note that the value $b(i)$ is the leading variable of $\Omega_{i}$.
The algorithm provides ${\Lambda}_{\mathcal C}$ but also gives 
${\Lambda}_{\mathcal C}^{E}$ for $E = \{x^{\delta_x} y^{\delta_y}=0\}$
by making the following minor changes:
\begin{itemize}
\item For $\delta_{y}=0$, we consider $a (1) = b (1) = n$, $a (2) = n \delta_{x}+ \beta_{1}$ and $b (2) = \beta_{1}$;
\item For $\delta_{y} =1$, we consider 
$a (1) = n \delta_{x} + \beta_{0}$, $b(1)=\beta_{0}$, $a (2) = n+ \beta_{1}$ and $b (2) = \beta_{1}$.
\end{itemize}
\begin{algorithm}[h!]
  \caption{Computation of the orders of an Apery basis}
  \label{alg:orders}
  \begin{algorithmic}
    \State{\textsf{Input}: $\beta_0=n,\beta_1,\ldots,\beta_g$}
    \State{\textsf{Values corresponding to $\mathcal{T}_0$}:}
    \State $a(1)\leftarrow n,\,b(1)\leftarrow n,\,
    a(2)\leftarrow \beta_1,\, b(2)\leftarrow \beta_1$
    \For{$l=0,\ldots, g-1$}
    \State{\textsf{Compute the orders corresponding to $\mathcal{G}^0_{l+1}$:}}
    \For{$k=1,\ldots, n_{l+1}$}
    \For{$s=1,\ldots, 2\nu_l$}
    \State{$a(2k\nu_l+s)\leftarrow \overline{\beta}_{l+1}k+a(s)$}
    \State{$b(2k\nu_l+s)\leftarrow \beta_{l+1}$}
    \EndFor
    \EndFor
    \State{$d\leftarrow \beta_{l+1},\,c\leftarrow 1$}
    \While{$d<\beta_{l+2} \land  c\neq 0$}
    \State{\textsf{Compute the orders corresponding
        to $\mathcal{G}^{r+1}_{l+1}$ from those of $\mathcal{G}^{r}_{l+1}$: } }
    \State{$c\leftarrow 0$}
    \While{$\exists k < s\, \big| (a(s)-a(k) \in (n)  \land  \max (b(s), b(k)) =d )$}
    \State{$a(s)\leftarrow a(s)+ {\mathfrak n} (d)- d$}
    \State{$b(s)\leftarrow {\mathfrak n} (d)$}
    \If{$s\leq n$}
    \State{$c\leftarrow c+1$}
    \EndIf
    \EndWhile
    \State{$d\leftarrow \mathfrak{n}(d)$}
    \EndWhile
    \EndFor
    \State{\textsf{Output: }$a(1),\ldots,a(n)$\textsf{, the orders of the nice  (Apery) basis}}
  \end{algorithmic}
\end{algorithm}
\section{Families of equisingular curves} \label{sec:families}
In order to reason using the resolution of singularities of a generic curve in an
equisingularity class, we need to study the relation between an analytic family of curves
and another family in the blow-up of the generic element. It turns out that these families
are intimately related and a nice basis for the former will essentially provide a nice
basis for the latter. The study of this relation is the purpose of this section.
%, not necessarily the same as in the previous sections.
\subsection{Suitable pairs}
Fix a local system of coordinates $(x,y)$ in $(\mathbb{C}^2,0)$.
From now on, we shall use the expression \emph{pair} to refer to a pair
$(\Gamma, E)$ where $\Gamma$ is a germ of irreducible curve at $0 \in \mathbb{C}^{2}$ (or
at any point in a non-singular complex analytic surface) and $E$ is a germ of normal
crossings divisor.
\begin{defi} \label{def:blow} Given a pair $(\Gamma,E)$, let $\pi_1$ be the blow-up of the
  origin $P_0 =(0,0)$ of $\mathbb{C}^{2}$ and let $E_1 = \pi_{1}^{-1} (E)$ and $P_1$ be
  the point in the strict transform $\Gamma_1$ of $\Gamma$ that belongs to the divisor
  $\pi_{1}^{-1} (0,0)$. We consider ${\Gamma}_1$ as a germ of irreducible curve defined in
  a neighborhood of ${P}_{1}$.
%Denote by  $\tilde{E} = \pi_{1}^{-1} (E)$. 
%in the neighborhood of $\tilde{P}$.
\end{defi}
The following is the key notion which will enable us to relate a nice basis before
blow-up with another one after.
\begin{defi}
  We say that the pair $(\Gamma, E)$ is {\it suitable} if
\begin{itemize}
\item  $E \subset  \{xy=0 \}$,
\item $x=0$ is not the tangent cone of $\Gamma$  at the origin and
\item if $E = \{y=0\}$ then $y=0$ is the tangent cone of $\Gamma$ at the origin.
\end{itemize}
A pair $(\Gamma, E)$ which is not suitable will be called
\emph{unsuitable}.
\end{defi}
\begin{rem}
\label{rem:suit}
Let ${\mathcal C}$ be the equisingularity class of a pair $(\Gamma, E)$ (with
$E \subset \{xy=0\}$). The property of being suitable only depends on the order of the
coordinates $(x,y)$. Indeed, $(\Gamma,E)$ is suitable in coordinates
$({\rm x}_{0}, {\rm y}_{0})$, where $({\rm x}_{0}, {\rm y}_{0}) = (x,y)$ if $(\Gamma,E)$
is suitable in coordinates $(x,y)$; otherwise, it is suitable in coordinates $({\rm x}_{0}, {\rm y}_{0}) = (y,x)$.
\end{rem}
Consider the equisingularity class ${\mathcal C}$ of a suitable pair $(\Gamma_{\circ}, E)$ in coordinates $({\rm x}_0, {\rm y}_0) = (x,y)$. Any
curve $\Gamma \in \mathcal{C}$ has a Puiseux parametrization of the form
\begin{equation}
\label{equ:par_gamma}
  \Gamma (t) = \left( t^n, \sum_{\beta \geq \beta_0} a_{\beta, \Gamma} t^{\beta} \right)
\end{equation}
where $n = \nu (\Gamma)$ depends just on ${\mathcal C}$ and so we can write 
$n = \nu ({\mathcal C})$,  $a_{\beta_0, \Gamma} \neq 0$ and $\beta_{0} \geq n$.  That
parametrization will be called a {\it suitable parametrization} of $\Gamma$, leaving $E$
implicit. Consider coordinates $x=x_1$, $y= x_1 y_1$ in the first chart of the blow-up
$\pi_1$ of the origin. Since $x=0$ is not the tangent cone of $\Gamma$, its strict
transform ${\Gamma}_1$ has a parametrization
\begin{equation}
\label{equ:par_gamma1}
  {\Gamma}_1 (t) = \left( t^n, \sum_{\beta \geq \beta_0} a_{\beta, \Gamma} t^{\beta - n} \right)
\end{equation}
in coordinates $(x_1, y_1)$.
%which we rename $(x,y)$ from now on, for simplicity. 
We now
relate the main invariants of $(\Gamma,E)$ to those of $(\Gamma_1, \pi_1^{-1}(E))$ in a
neighborhood of $P_1$.

\subsubsection{Approximate roots} \label{subsec:app_root}
Take $1 \leq j \leq g$. Given an approximate root 
\[
  f_{j, \Gamma} = y^{\nu_{j-1}} + a_{\nu_{j-1} -1}(x) y^{\nu_{j-1} -1} + \hdots + a_0 (x)
\]
of $\Gamma \in \mathcal{C}$, the $j$th-approximate root ${f}_{1,j, \Gamma_{1}}$ of ${\Gamma}_1$ in the
local system of coordinates above satisfies
\[
  {f}_{1,j,\Gamma_{1}} = \frac{f_{j, \Gamma} \circ \pi_1}{x_1^{\nu_{j-1}}} .
\]
Let ${\mathcal C}_1$ be the equisingularity class of $({\Gamma}_1, E_1)$.  Since by construction,
$\nu (\Gamma_{1})$ is the multiplicity of $\Gamma_1$ at $P_1$,
and on the other hand the values $\nu_{\Gamma_1} (f_{1,j, \Gamma_{1}})$ depend just on
${\mathcal C}_{1}$, we shall denote them $\nu({\mathcal C}_{1})$ and
$\nu_{{\mathcal C}_1} ({f}_{1,j}) $ respectively, for simplicity.  Certainly, 
\[
  \nu_{{\mathcal C}_1} ({f}_{1,j}) =
  \nu_{\mathcal C} (f_{j}) - \nu_{j-1} n =
  \overline{\beta}_j   - \nu_{j-1} n  
\]
holds.
Moreover, for any $1\leq j \leq g$, the $y$-degree of both $f_j$ and $f_{1,j}$ is equal to
$\nu_{j-1}$, and their leading variables are $\beta_j$, where we are considering coordinates $(x,y)$
and $(x_1, y_1)$ and hence the parametrizations \eqref{equ:par_gamma} and \eqref{equ:par_gamma1}. 
\subsubsection{Preparing the pair $({\Gamma}_1, {E}_1)$}
In this subsection we find coordinates $({\rm x}_1, {\rm y}_1)$ in which $\Gamma_1$ has a suitable parametrization that we relate
with the parametrization \eqref{equ:par_gamma} of $\Gamma$.

The pair $({\Gamma}_1, {E}_1)$ is not necessary suitable in coordinates $(x_1,y_1)$; for instance, if $a_{n} \neq 0$
then $(\Gamma_1,E_1)$ is not even a germ of curve at $(x_{1}, y_{1})=(0,0)$. 
%(in the coordinates $(x,y)$ given previously). 
This is easily solved by means of the change of coordinates $(x_1 , y_1^{\prime}) = (x_1, y_1 - a_{n})$.
%$(x, y) \mapsto (x, y + a_n)$. 
Thus, %we can assume that
\begin{equation}
\label{equ:suit1}
 \left( t^n, \sum_{\beta \geq \beta'} a_{\beta, \Gamma} t^{\beta - n} \right)  
\end{equation}
is a parametrization of ${\Gamma}_1$ in coordinates $(x_1 , y_1^{\prime})$, where  
$\beta'$ is the first index, greater than $n$, such that $a_{\beta', \Gamma} \neq 0$.  
Now $({\Gamma}_1, {E}_1)$  is suitable if and only if $\beta' - n \geq n$. 
If this is the case, then we define coordinates $({\rm x}_1 , {\rm y}_1)  = (x_1 , y_1^{\prime}) $ and 
\eqref{equ:suit1} provides a suitable parametrization
\[
  \left( t^n, \sum_{\beta \geq \beta' -n} b_{\beta, \Gamma_1} t^{\beta} \right)
  = \left( t^{\nu ({\mathcal C}_1)},
    \sum_{\beta \geq \beta' -n} b_{\beta, \Gamma_1} t^{\beta} \right) ,
\]
in coordinates $({\rm x}_1 , {\rm y}_1)$, where
\begin{equation} \label{equ:blow_suit}
b_{\beta - \nu ({\mathcal C}_1), \Gamma_1} =  b_{\beta-n, \Gamma_1} = a_{\beta, \Gamma}  
\end{equation}
for any $\beta \in {\mathcal E}_{E}(\mathcal{C}) \setminus \{n\}$.
\begin{rem}
  Notice that $({\Gamma}_1, {E}_1)$ is suitable, in coordinates  $(x_1 , y_1^{\prime})$,
  either for all $\Gamma \in \mathcal{C}$ or
  for none.  If there exists an exponent in ${\mathcal E}_{E}(\mathcal{C})$ greater than
  $n$ and less than $2n$ then $\beta'$ is equal to the first Puiseux exponent $\beta_1$.
  Since $a_{\beta_1}$ never vanishes on $\mathcal{C}$, it follows that
  $({\Gamma}_1, {E}_1)$ is unsuitable for any $\Gamma \in \mathcal{C}$.  If no such a
  exponent exists, clearly $({\Gamma}_1, {E}_1)$ is suitable for any
  $\Gamma \in \mathcal{C}$.
\end{rem}
 
Assume that $({\Gamma}_1, {E}_1)$ is unsuitable in coordinates  $(x_1 , y_1^{\prime})$ 
for $\Gamma \in \mathcal{C}$.  
The pair $({\Gamma}_1, {E}_1)$ is suitable
in coordinates $({\rm x}_1, {\rm y}_1) = (y_1^{\prime} , x_1)$ by Remark \ref{rem:suit}. 
Let us provide a suitable parametrization of $\Gamma_1$ in coordinates $({\rm x}_1, {\rm y}_1)$.
Notice that 
\[
  \left( \sum_{\beta \geq \beta_1} a_{\beta, \Gamma} t^{\beta - n}, t^n  \right)
\]
is a (possibly unsuitable) parametrization of ${\Gamma}_1$ in coordinates $({\rm x}_1, {\rm y}_1)$.
From this one, we can obtain,
by taking $(\beta_1-n)$-th roots of unity, the suitable parametrization
\begin{equation}\label{eq:suitable-param-above}
  \left( t^{\beta_1 -n},   \sum_{\beta \geq n} b_{\beta, \Gamma_1} t^{\beta} \right) =
  \left( t^{ \nu ({\mathcal C}_1)},   \sum_{\beta \geq n} b_{\beta, \Gamma_1} t^{\beta} \right) .
\end{equation}
with the important properties:
\begin{equation}
\label{equ:blow_non_suit1} 
b_{n, \Gamma_1} = a_{\beta_1, \Gamma}^{-\frac{n}{\beta_1 -n}}
\end{equation}
and 
\begin{equation} \label{equ:blow_non_suit2} 
  b_{\beta - \nu ({\Gamma}_1), \Gamma_1} =
  - \frac{n}{\beta_{1} -n}  a_{\beta_1, \Gamma}^{- \frac{\beta-n}{\beta_{1}-n}} a_{\beta, \Gamma} +
  Q_{\beta} (a_{\beta_1, \Gamma}^{\frac{-1}{\beta_{1} -n}},
  a_{{\mathfrak n}(\beta_{1}), \Gamma}, \hdots, a_{{\mathfrak p}(\beta), \Gamma} ),  
\end{equation}
where for any $\beta \in {\mathcal E}_{E}(\mathcal{C})$ with $\beta > \beta_1$,
$Q_{\beta}$ is a polynomial. In particular $b_{\beta - \nu ({\Gamma}_1)}$ depends on
$a_{\beta_1}, a_{{\mathfrak n}(\beta_1)}, \hdots, a_{\beta}$ and in degree
$1$ on $a_{\beta}$.
\begin{rem}
\label{rem:suitable_blow}
Let $\Gamma \in \mathcal{C}$, and consider a
suitable pair $(\Gamma, E)$ in coordinates $({\rm x}_0, {\rm y}_0) = (x,y)$. Let
\begin{equation}\label{eq:desing-Gamma}
  (\mathbb{C}^2,0) \stackrel{\pi_1}{\longleftarrow} \mathcal{X}_1
  \stackrel{\pi_2}{\longleftarrow} \cdots
  \mathcal{X}_{k-1}\stackrel{\pi_k}{\longleftarrow}\mathcal{X}_k
  \stackrel{\pi_{k+1}}{\longleftarrow}\cdots
\end{equation}
be the local blowing-up process following the infinitely near points of
$\Gamma = \Gamma_0$, and denote by $\Gamma_i$ the strict transform of $\Gamma$ in
$\mathcal{X}_{i}$.  Let $E_{0} = E$ and $E_i=\pi_i^{-1}\circ\cdots\circ\pi_1^{-1}(E)$ be
the (germ of the) exceptional divisor in each stage $i \geq 1$. Denote
$P_0 = (0,0) \in \mathbb{C}^{2}$ and let $P_i$ be the point of $\Gamma_{i}$ in $E_i$.  We
find coordinates $({\rm x}_s, {\rm y}_s)$ in which the pair $(\Gamma_s, E_s)$ is
suitable for $s \geq 1$ by  just iterating the construction that was described for
$s=1$.  For any germ of curve $\Gamma'$ at $P_s$ such that $(\Gamma^{\prime}, E_s)$ is in
the equisingularity class ${\mathcal C}_s$ of the pair $(\Gamma_s , E_s)$ there is a
suitable parametrization of $\Gamma'$ in coordinates $({\rm x}_s, {\rm y}_s)$.

In the remainder of this section, we fix the coordinates $({\rm x}_s, {\rm y}_s)$ and suitable parametrizations of 
elements in ${\mathcal C}_s$ are considered in such coordinates for $s \geq 0$.
Given $\Gamma_{s}^{\prime} \in {\mathcal C}_s$, the suitable parametrization of the strict transform $\Gamma_{s+1}^{\prime}$ of  
$\Gamma_{s}^{\prime}$ by $\pi_{s+1}$
is obtained from the suitable parametrization 
of $\Gamma_s^{\prime}$ by either equation \eqref{equ:blow_suit} or equations
\eqref{equ:blow_non_suit1}, \eqref{equ:blow_non_suit2}.
\end{rem}
%In the remainder of this section, we fix the coordinates $({\rm x}, {\rm y}) = (x,y)$ and $({\rm x}_1, {\rm y}_1)$ to study
%${\mathcal C}$ and ${\mathcal C}_1$ respectively. 
%In particular we fix the suitable parametrizations \eqref{equ:par_gamma} for $\Gamma \in {\mathcal C}$
%and $(t^{\nu ({\mathcal C}_{1})}, \sum b_{\beta} t^{\beta})$ for $\Gamma_1 \in {\mathcal C}_1$.  
%The sequence $(b_{\beta})$ is given by either Equation \eqref{equ:blow_suit} or Equations
%\eqref{equ:blow_non_suit1} and \eqref{equ:blow_non_suit2} depending on whether or not $\Gamma_1$
%is suitable in coordinates $(x_1, y_1^{\prime})$.  This point of view can be iterated: 
%by starting with the pair $(\Gamma_1, E_1)$, coordinates $({\rm x}_1, {\rm y}_1$, a suitable parametrization
%
%
%
%
% Let ${\mathcal C}$ be the equisingularity class of a suitable pair $(\Gamma', E)$.
\subsubsection{Fanning exponents}
Consider now a one-parameter analytic family
$(\Gamma^{\epsilon})_{\epsilon \in ({\mathbb C},0)}$ of curves in ${\mathcal C}$: that is,
a family $\Gamma^{\epsilon}$ admitting parametrizations
  \[
    \Gamma^{\epsilon} (t) =
    \left( t^n, \sum_{\beta \geq \beta_0} a_{\beta}(\epsilon) t^{\beta} \right)
  \]
  where $a_{\beta_0}, a_{{\mathfrak n}(\beta_0)}, \hdots$ are analytic in a common neighborhood of
  $\epsilon =0$.  Suppose that the family is non-constant and let
  \[
    \theta(\Gamma^{\epsilon}) =
    \min \{ \beta \in {\mathcal E}_{E}(\mathcal{C}): a_{\beta} \ \mathrm{is} \
    \mathrm{non-constant} \}
  \]
  be the {\it fanning exponent} of the family
  $(\Gamma^{\epsilon})_{\epsilon \in ({\mathbb C},0)}$. Given such a family, we obtain by
  pull-back a corresponding family $(\Gamma^{\epsilon}_1,E_1)$ in $P_1$ with suitable
  parametrizations described above. The next remarks
  follow from the previous discussions.
\begin{rem} \label{rem:coef_tr} \cite[Lemma 3.19]{Fortuny-Ribon-Canadian} There are two possibilities:
\begin{itemize}
\item Either $\theta(\Gamma^{\epsilon}) = n$ and the map
  $\epsilon \mapsto {\Gamma}_{1}^{\epsilon} \cap \pi_{1}^{-1} (0)$ is non-constant: in
  other words the base point of ${\Gamma}_{1}^{\epsilon}$ in $\pi_{1}^{-1} (0)$ varies in
  the family,
\item Or $\theta(\Gamma^{\epsilon}) > n$ and the first non-constant coefficient of the
  suitable parametrization of the family $({\Gamma}_{1}^{\epsilon}, {E}_1)$ is
  $b_{\theta(\Gamma^{\epsilon}) - \nu ({\mathcal C}_1)} (\epsilon)$.
\end{itemize}
\end{rem}
\begin{rem} \label{rem:coef_down_up} Fix a parametrization
  $ (t^{ \nu ({\mathcal C})}, \sum a_{\beta, {\Gamma}} t^{\beta}) $ of
  $\Gamma \in {\mathcal C}$.  Given $\beta' \in {\mathcal E}_{E}({\mathcal C})$, consider
  the family $({\Gamma}_{\beta^{\prime}}^{\epsilon})_{\epsilon \in (\mathbb{C},0)}$ of
  curves in ${\mathcal C}$ such that
  $(t^{ \nu ({\mathcal C})}, \epsilon t^{\beta'} + \sum a_{\beta, {\Gamma}} t^{\beta})$ is
  a parametrization of $\Gamma_{\beta^{\prime}}^{\epsilon}$.  The family of strict
  transforms $({\Gamma}_{1,\beta^{\prime}}^{\epsilon})_{\epsilon \in (\mathbb{C},0)}$ is a
  family of curves in ${\mathcal C}_1$ as long as $\beta' > \nu ({\mathcal C})$.  Remark
  \ref{rem:coef_tr} implies also that if $\beta' > \nu ({\mathcal C})$, then
  $\beta' - \nu ({\mathcal C}_1) \in {\mathcal E}_{E_{1}}({\mathcal C}_1)$.
\end{rem}
\begin{rem} \label{rem:coef_up_down} Take, on the other hand, a suitable curve $\Gamma_1$
  such that  $\pi_1(\Gamma_1)=\Gamma\in \mathcal{C}$, with suitable parametrization
  $ (t^{ \nu ({\mathcal C}_{1})}, \sum b_{\beta, {\Gamma}_{1}} t^{\beta})$ 
  in coordinates $({\rm x}_1, {\rm y}_1)$.
  Given
  $\beta' \in {\mathcal E}_{{E}_{1}}({\mathcal C}_{1})$, consider the family
  $({\Gamma}_{1,\beta^{\prime}}^{\epsilon})_{\epsilon \in (\mathbb{C},0)}$ of curves in
  ${\mathcal C}$ such that
  $ (t^{ \nu ({\mathcal C}_{1})}, \epsilon t^{\beta'} + \sum b_{\beta, {\Gamma}_{1}}
  t^{\beta})$ is a parametrization of ${\Gamma}_{1,\beta^{\prime}}^{\epsilon}$.  This
  defines by blow-down a family $({\Gamma}^{\epsilon})_{\epsilon \in (\mathbb{C},0)}$ of
  equisingular curves such that $\Gamma^{0} = \Gamma$.  As a consequence, the fanning
  coefficient of $(\Gamma^{\epsilon})$ satisfies
  $\theta(\Gamma^{\epsilon}) \in {\mathcal E}_{E}(\mathcal{C}) \setminus \{ \nu ({\mathcal
    C}) \}$. Even more, $\theta(\Gamma^{\epsilon}) = \beta' + \nu ({\mathcal C}_{1})$ by
  Remark \ref{rem:coef_tr}.
\end{rem}
\begin{rem}
Remarks  \ref{rem:coef_down_up} and \ref{rem:coef_up_down} imply that the map
\[ \begin{array}{ccc}
     {\mathcal E}_{E}(\mathcal{C}) \setminus \{ \nu ({\mathcal C}) \}
     & \to &  {\mathcal E}_{{E}_{1}}({\mathcal C}_{1})  \\
 \theta & \mapsto & \theta -  \nu ({\mathcal C}_{1})
\end{array} 
\]
is a bijection.
\end{rem} 
\begin{rem} \label{rem:generic} Let ${\mathcal C}$ be the equisingularity class of a
  suitable pair $(\Gamma, E)$ with $\Gamma$ generic in $\mathcal{C}$. The strict transform
  ${\Gamma}_{1}$ is a generic curve of ${\mathcal C}_{1}$.  This holds because
  $b_{\beta - \nu ({\mathcal C}_{1})}$ depends on the all coefficients
  $(a_{\beta'})_{\nu ({\mathcal C}) < \beta' \leq \beta }$ and in degree
  $1$\ on $a_{\beta}$ for any
  $\beta \in {\mathcal E}_{E}(\mathcal{C}) \setminus \{ \nu ({\mathcal C}) \}$ except when
  $\beta =n + \nu ({\mathcal C}_{1})$ and $({\Gamma}_{1}, {E}_{1})$ is unsuitable in
  coordinates $(x_1, y_1^{\prime})$.  In this case, $b_n$ is a power of the linear
  function $a_{\beta_1} = a_{n + \nu ({\mathcal C}_{1})}$.  This dependence on the highest
  index variable $a_{\beta}$ is preserved by further blow-ups, as is the genericity of the
  strict transforms.
  % As a consequence, for generic $\Gamma \in {\mathcal C}$, the curve ${\Gamma}_{1}$ is
  % generic in ${\mathcal C}_{1}$, the strict transform of ${\Gamma}_{1}$ by the blow-up
  % of ${P}_{1}$ is generic in the corresponding equisingularity class and so on.  More
  % precisely, it is easy to check that a non-zero polynomial in the variables
  % $(b_{\beta})_{\beta \in {\mathcal E}_{\tilde{E}}(\tilde{\mathcal C}) }$ induces a
  % non-zero polynomial
%\begin{itemize}
%\item in the variables $(a_{\beta})_{\beta \in  {\mathcal E}_{E}(\mathcal{C})
%  \setminus \{ \nu ({\mathcal C}) \} }$ and obtained by 
%  performing the substitutions \eqref{equ:blow_suit} if $(\tilde{\Gamma}', \tilde{E})$ is
%  suitable; 
%\item in the variables $(a_{\beta})_{\beta \in  {\mathcal E}_{E}(\mathcal{C})  \setminus
%  \{ \nu ({\mathcal C}) , \beta_1\} }$ and
% $a_{\beta_1}^{\frac{-1}{\nu (\tilde{\mathcal C})}}$, 
% obtained by  performing the substitutions  \eqref{equ:blow_non_suit1}
% and \eqref{equ:blow_non_suit2}   if $(\tilde{\Gamma}', \tilde{E})$ is unsuitable1. 
%\end{itemize}
\end{rem}
\subsection{Families of equisingular curves and desingularization} \label{subsec:blow}
Consider a curve $\Gamma \in \mathcal{C}$, a suitable pair $(\Gamma, E)$
and the setting in Remark \ref{rem:suitable_blow}. Let 
$(\Gamma^{\epsilon})_{\epsilon \in ({\mathbb C},0)}$ be a 
a  \emph{non-constant} analytic family
$(\Gamma^{\epsilon})_{\epsilon \in ({\mathbb C},0)}$ of curves with $\Gamma^0=\Gamma$ with
suitable parametrizations of the form
$\Gamma^{\epsilon}(t) = (t^{\nu ({\mathcal C})}, \sum_{\beta \geq \beta_0} a_{\beta}
(\epsilon) t^{\beta})$.  
Define $(\Gamma_0^{\epsilon}) = (\Gamma^{\epsilon})$ and
$(\Gamma_i^{\epsilon})$ the strict transform of $(\Gamma^{\epsilon})$ in
$\mathcal{X}_{i}$.   As
$(\Gamma^{\epsilon})$ is non-constant, there exists an integer
$\iota(\Gamma^{\epsilon}) \geq 1$ such that $\Gamma_i^{\epsilon}$ passes through $P_i$ for
all $0 \leq i < \iota(\Gamma^{\epsilon})$, and $\epsilon$ in a neighborhood of $0$ but
$\Gamma_{\iota}^{\epsilon}$ does not pass through $P_{\iota(\Gamma^{\epsilon})}$ for any
$\epsilon$ in a neighborhood of $0$.
\begin{defi}
  We say that $\iota(\Gamma^{\epsilon})$ is the {\it sliding
    divisor} of the family
  $(\Gamma^{\epsilon})_{\epsilon \in ({\mathbb C},0)}$.
\end{defi}
Denote $D_i = \pi_{i}^{-1} (P_{i -1})$ for $i \geq 1$. The sliding divisor
$\iota(\Gamma^{\epsilon})$ indicates the divisor $D_{\iota(\Gamma^{\epsilon})}$ where the
curves in the family $(\Gamma^{\epsilon})_{\epsilon \in ({\mathbb C},0)}$ have no common
point, hence the name (the point $\Gamma_{\iota(\Gamma^{\epsilon})}^{\epsilon}\cap D_{\iota(\Gamma^{\epsilon})}$ slides along that divisor).
\begin{rem}
  \label{rem:trace_pt}
  Let $\iota = \iota(\Gamma^{\epsilon})$ for clarity in this remark. Since the family
  $(\Gamma_{\iota -1}^{\epsilon}, E_{\iota -1})_{\epsilon \in ({\mathbb C},0)}$ is
  contained in an equisingularity class and their points
  $P_{\iota, \epsilon}=\Gamma_{\iota}^{\epsilon} \cap E_{\iota}$ 
  vary continuously in a non-constant way,  there
  is exactly one irreducible component of $E_{\iota}$ passing through $P_{\iota}$, so that
  $P_{\iota}$ is necessarily a trace point of $E_{\iota}$.
\end{rem}
\begin{rem}
  \label{rem:break_1}
  The sliding divisor $\iota(\Gamma_{\epsilon})$ is equal to $1$ if and only if the
  fanning exponent $\theta(\Gamma^{\epsilon})$ is
  equal to $\nu ({\mathcal C})$, by Remark \ref{rem:coef_tr}.
\end{rem}
%\begin{defi}
%We denote by $\zeta_i$ the multiplicity of $\Gamma_i^{0}$ at $P_i$ for $i \geq 0$.
%\end{defi}
We now proceed to show how the sliding divisor determines the fanning exponent. Let
$(\Gamma^{\epsilon})_{\epsilon \in ({\mathbb C},0)}$ be an analytic family 
%with $\Gamma^0$ generic 
in $\mathcal{C}$ and assume the notations above for $(\Gamma^{\epsilon}_i)$,
etc. Set $\Gamma_0^{\epsilon}=\Gamma^{\epsilon}$ for simplicity. 
The following definition is closely related to Corollary 3.20 of \cite{Fortuny-Ribon-Canadian}.
\begin{defi} \label{def:theta} Let ${\mathcal C}$ be the equisingularity class of a
suitable pair $(\Gamma^{0}, E)$.  We define $\mathbf{Slid}_{\mathcal C}$ as the set of
sliding divisors associated to analytic families
$(\Gamma^{\epsilon})_{\epsilon \in ({\mathbb C},0)}$ in ${\mathcal C}$. 
%Motivated by Lemma \ref{lem:break_det_fanning}, 
We define the map 
%$\hat{\theta}$
\[
\hat{\theta} (\iota) =
\nu_{P_{\iota -1}} (\Gamma_{\iota -1}^{0}) +
\sum_{j=1}^{\iota -1} \nu_{P_j} (\Gamma_{j}^{0}).
\]
for $\iota \in \mathbf{Slid}_{\mathcal C}$ (cf. Definition \ref{def:mult}).
\end{defi}
\begin{lem}
\label{lem:break_det_fanning}
Let $\iota=\iota(\Gamma^{\epsilon})$ be the sliding divisor of
$(\Gamma^{\epsilon})_{\epsilon \in ({\mathbb C},0)}$. Then
\[
  \theta(\Gamma^{\epsilon}):=\theta(\Gamma_0^{\epsilon}) = \hat{\theta} (\iota) 
\]
where an empty sum is $0$.
\end{lem}
\begin{proof}
  The equality is trivial if $\iota=1$ by Remark \ref{rem:coef_tr}, so we can assume
  $\iota > 1$. Since the sliding divisor of the family
  $(\Gamma_{\iota -1}^{\epsilon})$ is equal to one, its fanning exponent is the
  multiplicity $\nu_{P_{\iota -1}} (\Gamma_{\iota -1}^{0}) $ by Remark
  \ref{rem:break_1}. The result follows now applying Remark \ref{rem:coef_tr} to the
  families
  $(\Gamma_{\iota -2}^{\epsilon})_{\epsilon \in ({\mathbb C},0)}, \hdots,
  (\Gamma_{0}^{\epsilon})_{\epsilon \in ({\mathbb C},0)}$.
\end{proof}
Before proceeding, we study which exceptional divisors in the desingularization of a curve are candidates to sliding divisors:
\begin{lem}
\label{lem:break_set}
Let ${\mathcal C}$ be the equisingularity class of a suitable pair $(\Gamma, E)$. 
Then 
\[
  \mathbf{Slid}_{\mathcal C} =
  \{ \iota \geq 1 :
  P_{\iota} \ \mbox{\emph{is a trace point of}}\ E_{\iota} \} .
\]
\end{lem}
\begin{proof}
  Let us recall that $P_{i}$ is the $i$th-infinitely near point of $\Gamma$. Notice that
  whether or not $P_i$ is a trace point of $E_i$ does not depend on the choice of
  $\Gamma \in {\mathcal C}$.  By Remark \ref{rem:trace_pt}, if $\iota$ is a sliding
  divisor then $P_{\iota}$ is a trace point of $E_{\iota}$, which gives the direct
  inclusion.

  To prove the reverse inclusion, assume that $P_{\iota}$ is a trace point of $E_{\iota}$
  and let $\mu_{\iota -1} = \nu_{P_{\iota -1}} (\Gamma_{\iota -1})$.  Consider a suitable
  parametrization
  $(t^{\mu_{\iota -1}}, \sum_{\beta \geq \mu_{\iota -1}} a_{\beta, \Gamma} t^{\beta})$
  of $(\Gamma_{\iota -1}, E_{\iota -1})$ in coordinates $({\rm x}_{\iota -1}, {\rm y}_{\iota -1})$, 
  and define the analytic family
  $(\Gamma_{\iota -1}^{\epsilon})_{\epsilon \in ({\mathbb C},0)}$ with
  $\Gamma_{\iota -1}^{\epsilon}$ defined in the neighborhood of
  $P_{\iota -1}$ having suitable parametrization
  \[
    \Gamma^{\epsilon}_{\iota -1 } \equiv \bigg( t^{\mu_{\iota -1}}, \epsilon t^{\mu_{\iota
        -1}} + \sum_{\beta \geq \mu_{\iota -1}} a_{\beta, \Gamma^{0}} t^{\beta} \bigg).
  \]
  Obviously, $\Gamma_{\iota-1}^0=\Gamma_{\iota-1}$. Since $P_{\iota}$ is a trace point,
  the curves in the family belong to the equisingularity class of
  $(\Gamma_{\iota -1}, E_{\iota -1})$.  The family $(\Gamma_{\iota -1}^{\epsilon})$
  induces by blow-down another family
  $(\Gamma_{0}^{\epsilon})_{\epsilon \in ({\mathbb C},0)}$ defined in a neighborhood of
  the origin, and each curve $\Gamma_0^{\epsilon}$ belongs to $\mathcal{C}$.
  By construction, $\iota$ is the sliding divisor of such a family, and we are done.
\end{proof}
\begin{lem}
\label{lem:break_to_fanning_monotone}
Let ${\mathcal C}$ be the equisingularity class of a suitable pair $(\Gamma, E)$.
Consider $\iota, \iota' \in \mathbf{Slid}_{\mathcal C} $ with $\iota < \iota'$. Then
$ \hat{\theta} (\iota)< \hat{\theta} (\iota')$.
\end{lem}
\begin{proof}
  Denote $\mu_j = \nu_{P_j} (\Gamma_{j})$ for $j \geq 0$ and consider the family
  $(\Gamma_{0}^{\epsilon})_{\epsilon \in ({\mathbb C},0)}$ associated to $\iota'$ in the
  proof of Lemma \ref{lem:break_set} with $\Gamma^0_0=\Gamma$.  Since the sliding divisor
  $\iota(\Gamma_{\iota -1}^{\epsilon})$ is greater than $1$ by definition, we have,
  by Remark \ref{rem:coef_tr}:
  \[
    \theta(\Gamma_{\iota-1}^{\epsilon}) \in
    {\mathcal E}_{E_{\iota-1}}(\Gamma_{\iota -1}^{0}) \setminus \{ \mu_{\iota -1} \}
  \]
  and, in particular, $\theta(\Gamma^{\epsilon}_{\iota-1})>\mu_{\iota-1}$.
  By applying Remark
  \ref{rem:coef_tr} $\iota -1$ times, we conclude that
  \[
    \hat{\theta}(\iota^{\prime}) > \mu_{\iota-1} + \sum_{j=1}^{\iota -1} \mu_j
  \]
  Hence,
  $ \hat{\theta} (\iota)< \hat{\theta} (\iota')$ by Lemma \ref{lem:break_det_fanning}.
\end{proof}
\begin{pro}
\label{pro:bij_break_fanning}
Let ${\mathcal C}$ be the equisingularity class of a suitable pair $(\Gamma, E)$.  Then
$\hat{\theta}: \mathbf{Slid}_{\mathcal C} \to {\mathcal E}_{E}(\mathcal{C})$ is a monotone
bijection.
\end{pro}
\begin{proof}
  By Lemma \ref{lem:break_to_fanning_monotone}, it suffices to prove that $\hat{\theta}$
  is surjective. This is clear since, by Remark \ref{rem:coef_down_up}, any
  $k \in {\mathcal E}_{E}(\mathcal{C})$ is the fanning exponent of a family
  $(\Gamma^{\epsilon})_{\epsilon \in ({\mathbb C},0)}$ in ${\mathcal C}$.
\end{proof} 
\subsection{Contact orders between curves and vector fields}
Our study of the analytic moduli requires us to restrict the analytic families we
have studied to be associated to flows of vector fields. This is what we introduce in this
subsection. 

Consider the equisingularity class ${\mathcal C}$ of a suitable pair $(\Gamma,E)$. Recall
that $P_0, P_1, \hdots$ is the sequence of infinitely near points of $\Gamma$ following
the blow-ups $\pi_i$ for $i\geq 1$.  
\begin{defi}
  For any integer $i \geq 0$, we denote by ${\mathfrak X}_{E,i}$ the set of analytic
  vector fields defined in a neighborhood of $P_i$ in $\mathcal{X}_i$ that
\begin{itemize}
\item leave invariant every irreducible component of $E_i$ at $P_i$ (see subsection
  \ref{subsec:blow} for notation), and
\item are singular at $P_i$ for $i=0$.
\end{itemize}
We denote by ${\mathfrak X}_{E,i}^{\circ}$ the subset of ${\mathfrak X}_{E,i}$ consisting of
vector fields singular at $P_i$.
\end{defi}
 Given $X \in {\mathfrak X}_{E,i}$, with $X(P_i)=0$ or $i \geq 1$, such that
$\Gamma_{i}$ is not $X$-invariant, we consider the flow
$(\phi^{\epsilon})_{\epsilon \in ({\mathbb C},0)}$ of $X$ in a neighborhood of $P_i$ and
define the curve $\Gamma_{i}^{\epsilon} = \phi^{\epsilon} (\Gamma_{i})$ for
$\epsilon \in \mathbb{C}$ in a neighborhood of $0$. The pair
$(\Gamma_{i}^{\epsilon} , E_i)$ is in the same equisingularity class as
$(\Gamma_{i}, E_i)$ for $\epsilon$ in a neighborhood of $0$.  By blowing-down we obtain
the family $(\Gamma^{\epsilon})_{\epsilon \in ({\mathbb C},0)}$ induced by $X$, which is
an analytic family in the equisingularity class of $\Gamma^{}\in {\mathcal C}$.

\textbf{Notation.}  The fanning exponent of the family $(\Gamma^{\epsilon})$ just
constructed will be denoted $\theta_{X} (\Gamma^{})$ in order to stress that
$\Gamma^{\epsilon}$ is defined by $X$.

\begin{defi} \label{def:flow-fanning}
For any $i\geq 0$, we define the sets of fanning exponents associated to analytic vector fields and to singular analytic vector fields:
\[
  \begin{split}
    &\Theta_{i+1} (\Gamma^{}) =
    \{ \theta_{X} (\Gamma^{}) : X \in {\mathfrak X}_{E,i+1} \},
    \\
    &\Theta_{i}^{\circ} (\Gamma^{}) =
    \{ \theta_{X} (\Gamma^{}) : X \in {\mathfrak X}_{E,i}^{\circ}\} .
\end{split}
\]
Both are included in ${\mathcal E}_{E}({\mathcal C}) $. For completeness, we set
$\Theta_0 (\Gamma^{}) = \Theta_{0}^{\circ} (\Gamma^{}) $. We shall call them the sets of
\emph{flow-fanning exponents} and the \emph{singular flow-fanning exponents} at indices
$i+1$ and $i$, respectively.
%and $\Theta_{i}^{\circ} (\Gamma^{}) = \{ \theta (\Gamma^{}, X) : X \in {\mathfrak X}_{E,i}^{\circ} \} $.
\end{defi}
\begin{rem} \cite[Corollary 4.11 and Lemma 4.13]{Fortuny-Ribon-Canadian} 
  \label{rem:gen_dim}
  When $E=\emptyset$, the number
  $\sharp ( {\mathcal E}_{\emptyset}({\mathcal C}) \setminus \Theta_{0}^{\circ} (\Gamma))$ is the 
  dimension of
  the space of analytic classes of elements in ${\mathcal C}$ 
  for any generic $\Gamma \in {\mathcal C}$. Moreover, we have  
  \[
    \Theta_{0}^{\circ} (\Gamma) =
    \{ \nu_{\Gamma} (\omega) - \nu ({\mathcal C})  \ | \  w \in \hat{\Omega}
    (\mathbb{C}^2,0)  \ \mathrm{and} \ \omega (0)=0 \}  
  \]
  by \cite[Theorem 3.21]{Fortuny-Ribon-Canadian}.
\end{rem}
The following long result is the foundation of the enumeration of the elements in $\Lambda_{\Gamma}^{E}$ 
which will be done in the next section.
\begin{pro} \label{pro:breaking} Consider the equisingularity class ${\mathcal C}$ of a
  suitable pair $(\Gamma^{},E)$. Then:
\begin{enumerate}
\item \label{p3t} Singular flow-fanning exponents at index $i$ are flow-fanning
  exponents at index $i+1$:
  $\Theta_{i}^{\circ} (\Gamma^{}) \subset \Theta_{i+1} (\Gamma^{})$ for $i \geq 0$.
%\item \label{p8t} $\Theta_i (\Gamma^{})  \subset \mathbb{Z}_{\geq \theta(i)}$ for $i \geq 1$,
\item \label{p1t} For indices not corresponding to sliding divisors, flow-fanning
  exponents and singular flow-fanning exponents are the same:
  $\Theta_i (\Gamma^{}) = \Theta_{i}^{\circ} (\Gamma^{}) $ if
  $i \in \mathbb{Z}_{\geq 1} \setminus \mathbf{Slid}_{\mathcal C}$.
\item \label{p2t} Assume $i$ is a sliding divisor. Then flow-fanning exponents of index
  $i$ are greater than or equal to $\hat{\theta}(i)$. Moreover, $\hat{\theta}(i)$ is the
  unique non-singular flow-fanning exponent of index $i$: if
  $i\in \mathbf{Slid}_{\mathcal{C}}$, then
  $\Theta_i (\Gamma^{}) \subset \mathbb{Z}_{\geq \hat{\theta}(i)}$ and
  $\Theta_i (\Gamma^{}) \setminus \Theta_{i}^{\circ} (\Gamma^{}) = \{ \hat{\theta} (i)
  \}$.
\item \label{p8t} For any sliding divisor $i$ and $j>i$, the flow-fanning exponents of
  index $j$ are strictly greater than $\hat{\theta}(i)$:
  $\Theta_{j} (\Gamma^{}) \subset \mathbb{Z}_{> \hat{\theta}(i)}$ if
  $i \in \mathbf{Slid}_{\mathcal C}$ and $i < j$.
\item \label{p4t} Flow-fanning exponents of index $i+1$ which are not singular
  flow-fanning exponents of index $i$ are the same thing as flow-fanning exponents of
  index $i+1$ which are not of index $i$, for $i\geq 0$:
  $ \Theta_{i+1} (\Gamma^{}) \setminus \Theta_{i}^{\circ} (\Gamma^{}) = \Theta_{i+1}
  (\Gamma^{}) \setminus \Theta_{i} (\Gamma^{}) $ for $i \geq 0$.
\item \label{p5t} The set of exponents $\mathcal{E}_{E}(\mathcal{C})$ is the union of the
  flow-fanning exponents of all indices:
  ${\mathcal E}_{E}({\mathcal C}) = \cup_{i \geq 0} \Theta_{i} (\Gamma^{})$.
\item \label{p6t} Actually, one has the following union where $\sqcup$ indicates disjoint
  union:
  \[
    \mathcal{E}_{E}(\mathcal{C}) =
    \Theta_{0}^{\circ} (\Gamma^{})
    \bigsqcup
    \sqcup_{i \geq 0}
    (\Theta_{i+1} (\Gamma^{}) \setminus  \Theta_{i}^{\circ} (\Gamma^{}) ).
    \]
  \item \label{p7t} Finally, whenever $\Gamma_i$ is transverse to a nonempty $E_i$ (in
    particular $P_i$ is a trace point of $E_i$ and $\Gamma_i$ is non-singular at $P_i$),
    then all flow-fanning exponents of index $i+1$ correspond to singular flow-fanning
    exponents of index $i$. That is, if $\tau$ is the first index such that
    $\Gamma_{\tau}$ is transverse to a nonempty $E_{\tau}$, then
    $\Theta_{i+1} (\Gamma^{}) \setminus \Theta_{i}^{\circ} (\Gamma^{}) = \emptyset$ for
    $i \geq \tau$.
\end{enumerate}
\end{pro}

\begin{proof}
  Property \eqref{p3t} holds because any $X \in {\mathfrak X}_{E,i}^{\circ}$ can be lifted
  by $\pi_{i+1}$ to a vector field in $ {\mathfrak X}_{E,i+1}$ because $X$ is singular at
  $P_i$.

  Let $i \in \mathbb{Z}_{\geq 1} \setminus \mathbf{Slid}_{\mathcal C}$ be a non-sliding
  divisor. By Lemma \ref{lem:break_set}, $P_{i}$ is not a trace point of $E_i$, so that
  $E_i$ has two irreducible components at $P_i$. By definition, any
  $X \in {\mathfrak X}_{E,i}$ preserves both components, and must be singular, which gives
  Property \eqref{p1t}.

%Let us show \eqref{p2t}.
  To prove \eqref{p2t}, let $i\in \mathbf{Slid}_{\mathcal{C}}$ and
  $k \in \Theta_i (\Gamma^{})$. By definition, $k = \theta_{X}(\Gamma^{})$ for some
  $X \in {\mathfrak X}_{E,i}$. The flow of $X$ induces a family
  $(\Gamma^{\epsilon})_{\epsilon \in ({\mathbb C},0)}$ whose sliding divisor $i'$ is
  $i' = i$ when $X \not \in {\mathfrak X}_{E,i}^{\circ}$, and $i' > i$ otherwise.  In any
  case, $k = \hat{\theta} (i') \geq \hat{\theta} (i)$ by Proposition
  \ref{pro:bij_break_fanning}, which is the first assertion. We also know that
  $\Theta_i (\Gamma^{}) \setminus \Theta_{i}^{\circ} (\Gamma^{}) \subset \{ \hat{\theta}
  (i) \}$ and $ \Theta_{i}^{\circ} (\Gamma^{}) \subset \mathbb{Z}_{> \hat{\theta}(i)}$. As
  $\hat{\theta} (i)$ is equal to $\theta_{X} (\Gamma^{})$ for any
  $X \in {\mathfrak X}_{E,i} \setminus {\mathfrak X}_{E,i}^{\circ}$ it belongs
  to $\Theta_i (\Gamma^{}) $, and we have the second assertion.

  For \eqref{p8t}, take $k \in \Theta_j (\Gamma^{})$ with $i < j$ so that
  $k = \theta_{X} (\Gamma^{})$ for some $X \in {\mathfrak X}_{E,j}$. The family induced
  by the flow of $X$ on $\Gamma^{}$ has a sliding divisor $j' > i$, which gives
  $k = \hat{\theta}(j') > \hat{\theta} (i)$, and we are done.

  Consider \eqref{p4t}. The inclusion $\supset$ is clear since
  $ \Theta_{i}^{\circ} (\Gamma^{}) \subset \Theta_{i} (\Gamma^{})$.  Let us prove the
  direct inclusion $\subset$.  Take
  $k \in \Theta_{i+1} (\Gamma^{}) \setminus \Theta_{i}^{\circ} (\Gamma^{}) $. It suffices
  to show $k \not \in \Theta_{i} (\Gamma^{}) $. If it were not so, then
  $k \in \Theta_{i} (\Gamma^{}) \setminus \Theta_{i}^{\circ} (\Gamma^{})$, so that
  $i \in \mathbf{Slid}_{\mathcal C}$ and $k = \hat{\theta} (i)$ by \eqref{p1t} and
  \eqref{p2t}.  Since $ \Theta_{i+1} (\Gamma^{}) \subset \mathbb{Z}_{> \hat{\theta} (i)}$ by
  \eqref{p8t}, we obtain a contradiction.

  The inclusion $\supset$ in \eqref{p5t} is obvious. For the direct inclusion, take
  $k \in {\mathcal E}_{E}({\mathcal C})$. By Proposition \ref{pro:bij_break_fanning} it is
  of the form $k = \hat{\theta} (i)$, and thus it belongs to $\Theta_i (\Gamma^{}) $.

  Property \eqref{p6t} is the longest to prove. First, we claim that if
  $k \in \Theta_{i} (\Gamma^{}) \cap \Theta_{j} (\Gamma^{})$ with $0 \leq i < j$ then
  $k \in \Theta_{i}^{\circ} (\Gamma^{}) \cap \Theta_{i+1} (\Gamma^{})$.  By \eqref{p3t},
  it suffices to show that $k \in \Theta_{i}^{\circ} (\Gamma^{})$. This holds for $i=0$ by
  definition, so we assume $i \geq 1$. Suppose, aiming at contradiction, that
  $k \not \in \Theta_{i}^{\circ} (\Gamma^{})$. Now, by \eqref{p1t}, we get
  $i \in \mathbf{Slid}_{\mathcal C}$; by \eqref{p2t},
  $k \in \Theta_{i} (\Gamma^{}) \setminus \Theta_{i}^{\circ} (\Gamma^{}) = \{ \hat{\theta} (i) \}$; and, finally, by \eqref{p8t},
  $k \in \Theta_{j} (\Gamma^{}) \subset \mathbb{Z}_{> \hat{\theta} (i)}$, which is the desired
  contradiction. Thus, $k \in \Theta_i^{\circ}(\Gamma^{})$.

  Now if $k \in \Theta_{i} (\Gamma^{}) \cap \Theta_{j} (\Gamma^{})$ for some
  $0 \leq i < j$, then $k \in \Theta_{l}^{\circ} (\Gamma^{})$ for any $ i \leq l < j$ by
  the above claim and induction on $l$. We can now prove \eqref{p6t}. The equality follows
  from \eqref{p5t}. It remains to show that the union is disjoint.  Given
  $k \in \Theta_{j} (\Gamma^{}) \setminus \Theta_{j-1}^{\circ} (\Gamma^{})$, if it
  belonged to $\Theta_{i} (\Gamma^{}) \setminus \Theta_{i-1}^{\circ} (\Gamma^{})$ for
  $i < j$ or to $ \Theta_{0}^{\circ} (\Gamma^{}) $, then
  $k \in \Theta_{j-1}^{\circ} (\Gamma^{})$, which is impossible. Hence, the union is
  disjoint.

  Finally, Property \eqref{p7t} holds because the sets $\Theta_{i+1} (\Gamma^{})$
  and $\Theta_{i}^{\circ} (\Gamma^{})$ are both equal to $\mathbb{Z}_{\geq \hat{\theta} (i+1)}$.
\end{proof} 
\begin{defi}
  Let $\Gamma$ be a germ of irreducible curve and $E$ a germ of normal crossings divisor.
  We define
  \[
    \Lambda_{\Gamma, \circ}^{E} =
    \{ \nu_{\Gamma} (\omega) \ |
    \  w \in \hat{\Omega}_E  (\mathbb{C}^2,0)  \ \mathrm{and} \ \omega (0)=0
    \} .
  \]
\end{defi}  
\begin{pro}
\label{pro:contact_kahler}
Consider the equisingularity class ${\mathcal C}$ of a suitable pair $(\Gamma^{},E)$.
Then
 \[
   \begin{split}
     &\Theta_{i}^{\circ} (\Gamma^{})  =
     -  \nu_{P_i} (\Gamma_{i}^{})  +  \sum_{l=1}^{i}  \nu_{P_l} (\Gamma_{l}^{})   +
     \Lambda_{\Gamma_{i}^{},\circ}^{E_{i}}, \\
     &\Theta_{i+1} (\Gamma^{})  =
     \sum_{l=1}^{i} \nu_{P_l} (\Gamma_{l}^{})  + \Lambda_{\Gamma_{i+1}^{}}^{E_{i+1}}
 \end{split}
\] 
 for any $i \geq 0$.    
\end{pro}
\begin{proof}
  Denote $\mu_j = \nu_{P_j} (\Gamma_{j}^{})$ and fix $i \geq 0$.  We have
  $ \Lambda_{\Gamma_{i},\circ}^{E_{i}}= \mu_i + \Theta_{0}^{\circ} (\Gamma_{i}^{}) $
  \cite[Theorem 3.21]{Fortuny-Ribon-Canadian}.  This gives
\begin{equation} \label{equ:from_contact_to_kahler}
  \Theta_{i}^{\circ} (\Gamma^{})  =
  \sum_{l=1}^{i} \mu_l   +  \Theta_{0}^{\circ} (\Gamma_{i}^{})  =
  -\mu_{i} +
  \sum_{l=1}^{i} \mu_l   + \Lambda_{\Gamma_{i}^{},\circ}^{E_{i}}   
 \end{equation}
 by Remark \ref{rem:coef_tr}.  Notice that
 $\Lambda_{\Gamma_{i+1}}^{E_{i+1}} \setminus \Lambda_{\Gamma_{i+1}^{},\circ}^{E_{i+1}} =
 \emptyset$ if $i+1 \not \in \mathbf{Slid}_{\mathcal C}$ and
 $\Lambda_{\Gamma_{i+1}^{}}^{E_{i+1}} \setminus \Lambda_{\Gamma_{i+1}^{},\circ}^{E_{i+1}}
 = \{ \mu_{i} \}$ otherwise. As a consequence, if
 $i+1\not\in \mathbf{Slid}_{\mathcal{C}}$, we have
 \[
   \Theta_{i+1} (\Gamma^{})  =
   \Theta_{i+1}^{\circ} (\Gamma^{})  =
   \sum_{l=1}^{i} \mu_l  + \Lambda_{\Gamma_{i+1}^{},\circ}^{E_{i+1}} =
   \sum_{l=1}^{i} \mu_l  + \Lambda_{\Gamma_{i+1}^{}}^{E_{i+1}}
 \] 
 by \eqref{p1t} in Proposition \ref{pro:breaking} and equation
 \eqref{equ:from_contact_to_kahler}.  On the other hand, if
 $i+1\in \mathbf{Slid}_{\mathcal{C}}$, we have
 \[
   \Theta_{i+1} (\Gamma^{})  =
   \Theta_{i+1}^{\circ} (\Gamma^{})  \cup \{ \hat{\theta} (i+1) \}=
   \sum_{l=1}^{i} \mu_l  + (\Lambda_{\Gamma_{i+1}^{},\circ}^{E_{i+1}}  \cup \{ \mu_{i} \}) =
   \sum_{l=1}^{i} \mu_l + \Lambda_{\Gamma_{i+1}^{}}^{E_{i+1}}
 \] 
 by \eqref{p2t} in Proposition \ref{pro:breaking}, equation
 \eqref{equ:from_contact_to_kahler} and Definition \ref{def:theta}.
\end{proof}
 \begin{pro}
 \label{pro:calc_dim}
 Consider the equisingularity class ${\mathcal C}$ of a suitable pair $(\Gamma^{},E)$.  We
 have
 \begin{equation} \label{equ:param_coef}
   {\mathcal E}_{E}({\mathcal C}) \setminus \Theta_{0}^{\circ} (\Gamma^{}) =
   \bigsqcup_{0
     \leq i < \tau} (\Theta_{i+1} (\Gamma^{}) \setminus \Theta_{i}^{\circ} (\Gamma^{}) )=
   \bigsqcup_{0
     \leq i < \tau} (\Theta_{i+1} (\Gamma^{}) \setminus \Theta_{i} (\Gamma^{}) )
 \end{equation}
 where $\tau$ is the first index such that $\Gamma_{\tau}$ is non-singular and
 transverse to a non-empty $E_{\tau}$. Moreover,
 \[ \sharp ( {\mathcal E}_{E}({\mathcal C})  \setminus \Theta_{0}^{\circ} (\Gamma^{}) )   = 
  \sum_{0 \leq i < \tau}  \sharp ([\nu_{P_i} (\Gamma_{i}^{}) + \Lambda_{\Gamma_{i+1}^{}}^{E_{i+1}} ] \setminus   \Lambda_{\Gamma_{i}^{},\circ}^{E_{i}} ).  \]
  \end{pro}
\begin{proof}
  Denote $\mu_j = \nu_{P_j} (\Gamma_{j}^{})$.  Equation \eqref{equ:param_coef}
  is a consequence of
  \eqref{p6t}, \eqref{p7t} and \eqref{p4t} in Proposition \ref{pro:breaking}. We know that
\[
  \Theta_{i+1} (\Gamma^{}) \setminus \Theta_{i}^{\circ} (\Gamma^{}) =
  ( \sum_{l=1}^{i}
  \mu_l + \Lambda_{\Gamma_{i+1}^{}}^{E_{i+1}} ) \setminus
  (-\mu_{i} + \sum_{l=1}^{i}
  \mu_l + \Lambda_{\Gamma_{i}^{},\circ}^{E_{i}} )
\]
  by Proposition \ref{pro:contact_kahler}, so that
  \[
    \sharp ( \Theta_{i+1} (\Gamma^{}) \setminus  \Theta_{i}^{\circ} (\Gamma^{}) ) 
    =
    \sharp ( [ \mu_i + \Lambda_{\Gamma_{i+1}^{}}^{E_{i+1}}   ] \setminus
    \Lambda_{\Gamma_{i}^{},\circ}^{E_{i}} )
  \]
  for any $i \geq 0$, which completes the proof. 
\end{proof}
\section{Contacts of K\"ahler differentials and blow-ups} \label{sec:contacts}
We have finally arrived at the endpoint of our work. Here, we provide an application of our methods.
By studying the behavior of contacts of K\"ahler differentials under blow-up, using the results about flows of vector fields, we can give an alternative proof of Genzmer's formula \cite{Genzmer-moduli-2020}.

As above, $\Gamma$ is a germ of
irreducible curve and $E$ a germ of normal crossings divisor at $(0,0)$.  We keep the
notation of Definition \ref{def:blow}.
%We define 
%\[ \Lambda_{\Gamma, 0}^{E} = \{ \nu_{\Gamma} (\omega) \ | \  w \in \hat{\Omega}_E  (\mathbb{C}^2,0)  \ \mathrm{and} \ \omega (0)=0 \} . \]
%Consider the blow-up $\pi_1$ of the origin of $\mathbb{C}^{2}$ and let  $\tilde{P}$ be the point in the strict transform $\tilde{\Gamma}$ of $\Gamma$ that 
%belongs to the divisor  $\pi_{1}^{-1} (0,0)$. We consider $\tilde{\Gamma}$  as a germ of irreducible curve defined in a neighborhood of $\tilde{P}$. 
%Denote by  $\tilde{E}$ be the germ of $\pi_{1}^{-1} (E)$ in the neighborhood of $\tilde{P}$.
\begin{defi}
Let  ${\mathcal C}$ be the equisingularity class of $(\Gamma, E)$, we denote 
\[
  \Lambda_{\mathcal{C},\circ}^{E}  =
  \Lambda_{\Gamma',\circ}^{E} \quad \mathrm{and} \quad 
  \Lambda_{{\mathcal C}_{1}}^{{E}_{1}}  =
  \Lambda_{{\Gamma}_{1}'}^{{E}_{1}}
\]
for any generic $\Gamma' \in \mathcal{C}$. 
\end{defi}
\subsection{Main results}
Proposition \ref{pro:calc_dim} motivates us to compare $\Lambda_{\mathcal{C},\circ}^{E}$
with $\Lambda_{{\mathcal C}_{1}}^{{E_{1}}}$.  We first state the key results, and delay
their proofs to the forthcoming subsections.
\begin{defi} \label{def:sigma}
Given $n \in \mathbb{Z}_{\geq 2}$, we define
\begin{equation*}
  \sigma(n) = \left\{
    \begin{array}{l}
      \frac{(n-2)(n-4)}{4} \;\mathrm{if}\; n \equiv 0 \mod (2)\\[0.5em]
      \frac{(n-3)^2}{4} \;\mathrm{if}\; n \equiv 1 \mod (2) .
    \end{array}
  \right.
\end{equation*}
\end{defi}
The next result is the induction step required to prove Genzmer's formula. 
Recall that $E=\{ x^{\delta_x}y^{\delta_y} =0 \}$ with $\delta_x,\delta_y\in \left\{ 0,1 \right\}$.
\begin{pro}
\label{pro:blow-up-contacts}
The inclusion
$\Lambda_{\mathcal{C},\circ}^{E} \subset \nu(\mathcal{C}) +
\Lambda_{{\mathcal C}_{1}}^{{E_{1}}}$ holds, and we can count the difference:
\begin{equation}
\label{equ:blow-up-contacts}
\sharp ([\nu(\mathcal{C})  + \Lambda_{{\mathcal C}_{1}}^{{E_{1}}} ]
\setminus \Lambda_{\mathcal{C},\circ}^{E}) =
\sigma ( \nu  (\mathcal{C} ) + \delta_x + \delta_y) . 
\end{equation}
%where $l$ is the number of irreducible components of $E$.
\end{pro}
Proposition \ref{pro:blow-up-contacts} is a consequence of Proposition
\ref{pro:parallel} below. We introduce now the concept of ``parallel families'' of
$1$-forms for $\mathcal{C}$ and $\mathcal{C}_1$, as the main tool in the comparison of
$\Lambda_{\mathcal{C},\circ}^{E}$ and $\Lambda_{{\mathcal C}_{1}}^{{E_{1}}}$.
\begin{defi}
  We say that an $E$-good (resp. $E$-initial, $E$-terminal, $E$-nice basis) family
  ${\mathcal G}_{l}^{s} = (\Omega_1^{s}, \hdots, \Omega_{r}^{s})$ of level $l$
  of $\mathcal{C}$ and an ${E}_1$-good (resp. $E_1$-initial, $E_{1}$-terminal, $E_{1}$-nice basis) family
  $\tilde{\mathcal G}_{l}^{s} = (\tilde{\Omega}_1^{s}, \hdots, \tilde{\Omega}_{r}^{s})$ of the same level $l$ of  $\mathcal{C}_1$
  are {\it parallel} if the following properties hold:
\begin{itemize}
\item The leading variables of $\Omega_j^{s}$ and $\tilde{\Omega}_j^{s}$ coincide for any $1 \leq j \leq r$;
\item $\nu_{{\mathcal C}_{1}} (\tilde{\Omega}_{j}^{s}) = \nu_{{\mathcal C}} ({\Omega}_{j}^{s}) -  n  \lfloor \frac{j-1 + \delta_x + \delta_y}{2} \rfloor$ 
for any $1 \leq j \leq r$.
\end{itemize}
\end{defi}
\begin{pro}
\label{pro:parallel}
There exist an $E$-nice basis $\mathbf{\Omega} = (\Omega_1, \hdots, \Omega_n)$ for
${\mathcal C}$ and an ${E}_{1}$-nice basis
$\mathbf{\tilde{\Omega}}= (\tilde{\Omega}_1, \hdots, \tilde{\Omega}_n)$ for
${\mathcal C}_{1}$ that are parallel. Moreover,
\begin{equation}
\label{equ:mult}
\nu (\Omega_j) = \left\lfloor \frac{j-1 + \delta_x + \delta_y}{2} \right\rfloor 
\end{equation}
holds for any $1 \leq j \leq n$.
\end{pro}
\begin{cor}
  Let $\mathbf{\Omega}=(\Omega_1, \hdots, \Omega_n)$ be an $E$-nice basis. Define
  \[
    \hat{\Omega}_j  = (\pi_{1}^{*} \Omega_j)/ x_1^{\nu (\Omega_j) }
  \]
  for $1 \leq j \leq n$. Then $(\hat{\Omega}_1, \hdots, \hat{\Omega}_n)$ 
  is an $E_1$-Apery basis of ${\mathcal C}_{1}$. In other words, we have
 \[
    \hat{\Omega}_{E_{1}} ( \mathcal{X}_1, P_1) =
    \mathbb{C}[[x_1]] \hat{\Omega}_{1,\Gamma'} +
    \hdots + \mathbb{C}[[x_1]] \hat{\Omega}_{n,\Gamma'} +
    {\mathcal I}_{E_{1}, \Gamma'}
  \]
for any generic $\Gamma' \in {\mathcal C}_{1}$.  
\end{cor}
\begin{proof}
  By construction $\hat{\Omega}_j$ is a formal $1$-form defined at ${P}_{1}$ and
  preserving ${E}_{1}$ for any $1 \leq j \leq n$.  The equality
  $\nu_{{\mathcal C}_{1}} (\hat{\Omega}_j) = \nu_{{\mathcal C}_{1}} (\tilde{\Omega}_j)$
  for any $1 \leq j \leq n$ holds because $\mathbf{\Omega}$ and $\mathbf{\tilde{\Omega}}$
  are parallel by Proposition \ref{pro:parallel}. As a consequence,
  $(\hat{\Omega}_1, \hdots, \hat{\Omega}_n)$ is an Apery basis of ${\mathcal C}_{1}$.
\end{proof}
\subsection{Proof of Proposition \ref{pro:parallel}}   
Before proceeding to the core of the proof, we consider the different cases arising in it
in order to construct parallel families of the starting level $l_0$ for $\Gamma$ and $\Gamma_1$ in each
case. Once these are constructed, we can give a common argument. Assume that $(\Gamma,E)$ is a suitable pair
in coordinates $(x,y)$. We consider coordinates $(x_1,y_1)$ % = (x, y/x)$ 
to study the pair $(\Gamma_1, E_1)$ even if in such
coordinates it  could be unsuitable.
\subsubsection{Case $E = \emptyset$ or $E=\{x=0\}$}
\label{sub:noy}
In this case, $l_0=0$ for $\mathcal{C}$ and ${E}_{1} = \{x_1 = 0 \}$. A terminal family ${\mathcal T}_0$ of
level $0$ of $\mathcal{C}$ is $(\Omega_1, \Omega_2)= (dx, x^{\delta_x} df_{1})$ whereas a
terminal family $\tilde{\mathcal T}_0$ of level $0$ of ${\mathcal C}_{1}$ is
$(\tilde{\Omega}_1, \tilde{\Omega}_2)= (d x_1, x_1 d f_{1,1})$ (see subsection
\ref{subsec:app_root}).  Of course, $dx, dx_1 \in \Delta_{n, <}$ and
$d f_{1}, x_1 d {f}_{1,1} \in \Delta_{\beta_1, <}^{0}$, and also
\[
  \nu_{\mathcal C} (dx) = \nu_{{\mathcal C}_{1}} (d x_1) =
  n \quad \mathrm{and} \quad
  \nu_{\mathcal C} (x^{\delta_x} d f_{1}) -
  \nu_{{\mathcal C}_{1}}  ( x_1 d {f}_{1,1}) = n \delta_x  .
\]
Thus, ${\mathcal T}_0$ and $\tilde{\mathcal T}_0$ are parallel and equation
\eqref{equ:mult} holds for $j \in \{1,2\}$.
\subsubsection{Case $\delta_y = 1$ and $\beta_0 \in (n) \setminus \{n\}$} 
\label{sub:yn}
In this case, $l_0=0$ also, and ${E}_{1} = \{ x_1 y_1 =0\}$.  As explained in
subsubsection \ref{subsub:deltay}, ${\mathcal T}_0 = (\Omega_1, \Omega_2)$ where
\[ \Omega_1 = x^{\delta_x} dy \in \Delta_{\beta_0, <}^{0}, \
\Omega_2 \in \Delta_{\beta_1, <}^{0}, \
\nu_{\mathcal C} (\Omega_1)=n \delta_x + \beta_0 \ \mathrm{and} \ 
\nu_{\mathcal C} (\Omega_2)=n + \beta_1. \]
 Analogously,
$\tilde{\mathcal T}_0 = (\tilde{\Omega}_1, \tilde{\Omega}_2)$ where
$\tilde{\Omega}_1 = x_1 dy_1 \in \Delta_{\beta_0, <}^{0}$,
$\tilde{\Omega}_2 \in \Delta_{\beta_1, <}^{0}$,
$\nu_{{\mathcal C}_{1}} (\tilde{\Omega}_1)= \beta_0$ and
$\nu_{{\mathcal C}_{1}} (\tilde{\Omega}_2)= \beta_1$. Both families $\mathcal{T}_0$ and
$\tilde{\mathcal T}_0$ are terminal and indeed parallel, and equation \eqref{equ:mult}
holds for any $j \in \{1,2\}$.
\subsubsection{$\delta_y = 1$ and $\beta_0 =n$} 
\label{sub:yexn}
Again, we have $l_0=0$ for $\mathcal{C}$, $E= \{xy=0\}$ and ${E}_{1} = \{ x_1=0\}$.  As described in
subsubsection \ref{subsub:deltay}, ${\mathcal T}_0 = (\Omega_1, \Omega_2)$ where
$\Omega_1 = x dy \in \Delta_{n, <}^{0}$, $\Omega_2 \in \Delta_{\beta_1, <}^{0}$,
$\nu_{\mathcal C} (\Omega_1)=2n$ and $\nu_{\mathcal C} (\Omega_2)=n + \beta_1$.  We also
have $\tilde{\mathcal T}_0 = (\tilde{\Omega}_1, \tilde{\Omega}_2)$ where
$\tilde{\Omega}_1 = d x_1 \in \Delta_{n, <}$,
$\tilde{\Omega}_2 = x_1 d {f}_{1,1} \in \Delta_{\beta_1, <}^{0}$,
$\nu_{{\mathcal C}_{1}} (\tilde{\Omega}_1)= n$ and
$\nu_{{\mathcal C}_{1}} (\tilde{\Omega}_2)= \beta_1$. Once more, the families
$\mathcal{T}_0$ and $\tilde{\mathcal T}_0$ are both terminal, and parallel.
Equation \eqref{equ:mult} for $j \in \{1,2\}$ holds too.
\subsubsection{$\delta_y = 1$ and $\beta_0 \not \in (n)$}   
\label{sub:ynon}
In this case $l_0=1$ for $(\Gamma, E)$. We have $E= \{x^{\delta_x} y=0\}$ and ${E}_{1} = \{ x_1 y_1 =0 \}$.
Consider the $E$-good family ${\mathcal G}_1^{1}$ for ${\mathcal C}$ and the
${E}_{1}$-good family $\tilde{\mathcal G}_{1}^{1}$ for ${\mathcal C}_{1}$ as described in
subsubsection \ref{subsub:notinn}.  We have
\[
  \Omega_{2j+1}^{1}  =
  x^{\delta_x} y^{j} dy \in \Delta_{\beta_1, <},  \quad
  \tilde{\Omega}_{2j+1}^{1}  =
  x_1 y_{1}^{j} d y_1  \in \Delta_{\beta_1, <}
\] 
and
\[
  \nu_{\mathcal C} ( \Omega_{2j+1}^{1}  ) -
  \nu_{{\mathcal C}_{1}} ( \tilde{\Omega}_{2j+1}^{1} )  =
  n  \nu ( \Omega_{2j+1}^{1}  ) =
  n \left\lfloor \frac{(2j+1)-1 + \delta_x + \delta_y}{2} \right\rfloor
\]
if $1 \leq 2j+1 \leq 2 {\nu}_1$. We also get  $\Omega_{2j+2}^{1} , \tilde{\Omega}_{2j+2}^{1} \in \Delta_{{\mathfrak n}(\beta_1), <}^{0}$ and 
\[
  \nu_{\mathcal C} ( \Omega_{2j+2}^{1}  ) -
  \nu_{{\mathcal C}_{1}} ( \tilde{\Omega}_{2j+2}^{1} )  =
  (n + j \beta_1 + {\mathfrak n}(\beta_1)) - (n + j (\beta_1-n) + ({\mathfrak n}(\beta_1)-n)).
\] 
As a consequence 
\[
  \nu_{\mathcal C} ( \Omega_{2j+2}^{1}  ) -
  \nu_{{\mathcal C}_{1}} ( \tilde{\Omega}_{2j+2}^{1} )  =
  n  \nu ( \Omega_{2j+2}^{1}  ) =
  n \left\lfloor \frac{(2j+2)-1 + \delta_x + \delta_y}{2} \right\rfloor
\]
if $1 \leq 2j+2 \leq 2 {\nu}_1$.  Thus ${\mathcal G}_{1}^{1}$ and
$\tilde{\mathcal G}_{1}^{1}$ are parallel families (although not necessarily terminal).
The members of the family $(\Omega_{j}^{1}, \hdots, \Omega_{2 {\nu}_1}^{1})$ satisfy equation
\eqref{equ:mult}.
\subsubsection{End of the proof of Proposition \ref{pro:parallel}}
 %\begin{proof}[proof of Proposition \ref{pro:parallel} for $E= \emptyset$]
We apply the procedure of Section $1$ to build terminal families $\mathcal{T}_l$ and
$\tilde{\mathcal T}_l$ for ${\mathcal C}$ and ${\mathcal C}_{1}$ respectively for
$0 \leq l \leq g$.  We are going to show that $\mathcal{T}_l$ and $\tilde{\mathcal T}_l$
are parallel for any $0 \leq l \leq g$ by induction on $l$.

\strut

\textbf{Case $l_0=0$}. As seen above, $\mathcal{T}_0$ and $\tilde{\mathcal T}_0$ are
parallel. Assume that there exist parallel families
${\mathcal T}_l = (\Omega_1, \hdots, \Omega_{2 {\nu}_l})$ and
$\tilde{\mathcal T}_l = (\tilde{\Omega}_1, \hdots, \tilde{\Omega}_{2 {\nu}_l})$ of
level $l<g$ such that ${\mathcal T}_l$ satisfies equation \eqref{equ:mult}.  Define the
families
${\mathcal G}_{l+1}^{0} = (\Omega_{1}^{0}, \hdots \Omega_{2 {\nu}_{l+1}}^{0})$ and
$\tilde{\mathcal G}_{l+1}^{0} = (\tilde{\Omega}_{1}^{0}, \hdots
\tilde{\Omega}_{2 {\nu}_{l+1}}^{0})$ setting
$\Omega_{2 a \nu_{l}+j}^{0} = f_{l+1}^{a} \Omega_{j}$ and
$\tilde{\Omega}_{2 a \nu_{l}+j}^{0} = {f}_{1,l+1}^{a}\tilde{\Omega}_{j}$.  This gives
\begin{equation}
  \label{equ:mult0}
  \nu (\Omega_{2 a \nu_{l}+j}^{0}) =
  a \nu_l +
  \left\lfloor
    \frac{j-1 + \delta_x + \delta_y}{2}
  \right\rfloor =
  \left\lfloor
    \frac{(2 a \nu_l +j) -1 + \delta_x + \delta_y}{2}
  \right\rfloor 
\end{equation}
for $1 \leq 2 a \nu_l +j \leq 2 {\nu}_{l+1}$. Those two families are respectively
$E$-initial and ${E}_{1}$-initial for ${\mathcal C}$ and ${\mathcal C}_{1}$.  Since
${\mathcal T}_l$ and $\tilde{\mathcal T}_l$ are parallel, and $\beta_{l+1}$ is the leading
variable of $f_{l+1}^{a} \Omega_j$, and ${f}_{1,l+1}^{a} \tilde{\Omega}_j$ for
$1 \leq j \leq \tilde{\nu}_l$ and $a>0$, it follows that the leading variables of
$\Omega_{j}^{0}$ and $\tilde{\Omega}_{j}^{0}$ coincide for any
$1 \leq j \leq \tilde{\nu}_{l+1}$.  Moreover,
\[
  \nu_{{\mathcal C}} ({\Omega}_{2 a \nu_{l}+j}^{0})  -
  \nu_{{\mathcal C}_{1}} (\tilde{\Omega}_{2 a \nu_{l}+j}^{0})    =
  a (\overline{\beta}_{l+1} - (\overline{\beta}_{l+1}- \nu_l n )) +
  n \left\lfloor \frac{j-1 + \delta_x + \delta_y}{2} \right\rfloor
\]
and thus
\[
  \nu_{{\mathcal C}} ({\Omega}_{2 a \nu_{l}+j}^{0})  -
  \nu_{{\mathcal C}_{1}} (\tilde{\Omega}_{2 a \nu_{l}+j}^{0})    =
  n \left\lfloor \frac{2 a \nu_l + j-1 + \delta_x + \delta_y}{2} \right\rfloor .
\]
As a consequence, ${\mathcal G}_{l+1}^{0}$ and $\tilde{\mathcal G}_{l+1}^{0}$ are
parallel.
 
Take $1 \leq k \leq 2 {\nu}_{l+1}$, and assume $\Omega_{k}^{1} \neq \Omega_{k}^{0}$.
Then there exists $1 \leq j < k$ such that
$[\nu_{{\mathcal C}} ({\Omega}_{j}^{0})]_n = [\nu_{{\mathcal C}} ({\Omega}_{k}^{0})
]_n$. Since
\[
  [\nu_{{\mathcal C}_{1}} (\tilde{\Omega}_{j}^{0})]_n  =
  [\nu_{{\mathcal C}} ({\Omega}_{j}^{0}) ]_n  =
  [\nu_{{\mathcal C}} ({\Omega}_{k}^{0}) ]_n =
  [\nu_{{\mathcal C}_{1}} (\tilde{\Omega}_{k}^{0})]_n,
\]
we deduce that $\tilde{\Omega}_{k}^{1} \neq \tilde{\Omega}_{k}^{0}$.  The leading variable
of $\Omega_{k}^{0}$ is $\beta = \beta_{l+1}$ and, as ${\mathcal G}_{l+1}^{0}$ and
$\tilde{\mathcal G}_{l+1}^{0}$ are parallel, we conclude that the leading variables of
$\tilde{\Omega}_{k}^{1}$ and $\Omega_{k}^{1}$ are equal to ${\mathfrak n}(\beta)$ and also that
\[
  \nu_{{\mathcal C}} ({\Omega}_{k}^{1})  -
  \nu_{{\mathcal C}_{1}} (\tilde{\Omega}_{k}^{1}) =
  ( \nu_{{\mathcal C}} ({\Omega}_{k}^{0})  + m)-
  (\nu_{{\mathcal C}_{1}} (\tilde{\Omega}_{k}^{0}) + m) =
  \nu_{{\mathcal C}} ({\Omega}_{k}^{0})  -
  \nu_{{\mathcal C}_{1}} (\tilde{\Omega}_{k}^{0}),
\]
where $m = {\mathfrak n}(\beta) - \beta$. At the same time,
\[
  \Omega_{k}^{1} =
  C_{\beta', j} \Omega_{k}^{0} -
  C_{\beta, k} x^{d} \Omega_{j}^{0}  \quad
  \mathrm{and} \quad
  \tilde{\Omega}_{k}^{1} =
  \tilde{C}_{\beta', j} \tilde{\Omega}_{k}^{0} -
  \tilde{C}_{\beta, k} x^{\tilde{d}} \tilde{\Omega}_{j}^{0},
\]
(see equation \eqref{eq:main-substitution-s+1}), where
$d = \tilde{d} + (\nu (\Omega_{k}^{0}) - \nu (\Omega_{j}^{0}) )$ by the parallelism of
${\mathcal G}_{l+1}^{0}$ and $\tilde{\mathcal G}_{l+1}^{0}$, and by equation
\eqref{equ:mult0}. Notice also that $\tilde{d} \geq 0$ by construction. We deduce that
$\nu (\Omega_{k}^{0}) \leq \nu (x^{d} {\Omega}_{j}^{0})$.  Since
$\Omega_{j}^{0} \in M_{j}^{E}$ and $\Omega_{k}^{0} \in \hat{M}_{k}^{E}$, we get
$\nu (\Omega_{k}^{1}) = \nu (\Omega_{k}^{0})$.  Thus equation \eqref{equ:mult} for
$(\Omega_{1}^{1}, \hdots, \Omega_{2 {\nu}_{l+1}}^{1})$ is a consequence of equation
\eqref{equ:mult0}.  In the same way, $\tilde{\Omega}_{k}^{1} \neq \tilde{\Omega}_{k}^{0}$
implies $\Omega_{k}^{1} \neq \Omega_{k}^{0}$.
 
By repeating this argument finitely many times, we obtain that ${\mathcal T}_{l+1}$ and
$\tilde{\mathcal T}_{l+1}$ are parallel and ${\mathcal T}_{l+1}$ satisfies equation
\eqref{equ:mult}.  Thus, by induction on $l$, we obtain parallel terminal families 
${\mathcal T}_g$ and $\tilde{\mathcal T}_g$ and hence parallel nice bases 
${\bf \Omega}$ and $\tilde{\bf{\Omega}}$ such that
${\bf \Omega}$ satisfies equation \eqref{equ:mult}.

\strut 
 
\textbf{Case $l_0=1$}. In this remaining case, when $\delta_y = 1$ and
$\beta_0 \not \in (n)$, the families ${\mathcal G}_{1}^{1}$ and
$\tilde{\mathcal G}_{1}^{1}$ in subsection \ref{sub:ynon} are parallel and ${\mathcal G}_{1}^{1}$
satisfies equation \eqref{equ:mult}.  Following the same argument as above, we obtain
parallel terminal families of level $l_0=1$, ${\mathcal T}_1$ and $\tilde{\mathcal T}_1$
with equation \eqref{equ:mult} holding for ${\mathcal T}_1$.  The rest of the
proof is in the same as in the previous case.
% \end{proof}
% pedro aqui 20240413
\subsection{Proof of Proposition \ref{pro:blow-up-contacts}}
%First, let us prove the first part of Proposition \ref{pro:blow-up-contacts}.
%\begin{lem}
%Let $(\Gamma, E)$ be a suitable pair. Then
%$\Lambda_{\Gamma, 0}^{E} \subset \Lambda_{\tilde{\Gamma}}^{\tilde{E}} + \nu_{(0,0)} (\Gamma)$.
%\end{lem}
%\begin{proof}
%$\Gamma$ has a Puiseux parametrization 
%$\Gamma (t) = ( t^n, \sum_{\beta \geq \beta_0} a_{\beta} t^{\beta} )$
%where $n =  \nu_{(0,0)} (\Gamma)$. 
%Let $\omega = a(x,y) dx + b(x,y) dy  \in  \hat{\Omega}_E  (\mathbb{C}^2,0)$ with $\omega (0)=0$. The point $\tilde{P}$ is in the first
%chart $(x_1,y_1)$ of the blow-up. We define
%\[ \omega_1 := \frac{\pi_{1}^{*} \omega}{x_1} =  \frac{a(x_1,x_1 y_1) + b(x_1,x_1 y_1) y_1}{x_1} dx_1 + b(x_1,x_1 y_1) d y_1 .  \]
%By construction  $\omega_1$ is a formal $1$-form at $\tilde{P}$. 
%Since $\nu_{\Gamma} (\omega) = \nu_{\tilde{\Gamma}} (\pi_{1}^{*} \omega)$, it follows that 
%\[  \nu_{\Gamma} (\omega) = \nu_{\tilde{\Gamma}} (\omega_1) + n .\]
%In particular, any element $\nu_{\Gamma} (\omega)$ of $\Lambda_{\Gamma, 0}^{E}$ belongs to 
%$\Lambda_{\tilde{\Gamma}}^{\tilde{E}}  + \nu_{(0,0)} (\Gamma)$.
%\end{proof} 
%
%
%\subsubsection{end of the proof of Proposition \ref{pro:blow-up-contacts}}
We can assume that $(\Gamma, E)$ is a suitable pair by Remark \ref{rem:suit}.  We have
$\Lambda_{\mathcal{C},\circ}^{E} \subset \Lambda_{{\mathcal C}_{1}}^{{E}_{1}} + \nu
(\mathcal{C}) $ because $\Theta_{0}^{0} (\Gamma) \subset \Theta_{1} (\Gamma) $
by Proposition \ref{pro:breaking} and Proposition \ref{pro:contact_kahler}.

 We know that there exist parallel nice bases
${\bf \Omega}= (\Omega_1, \hdots, \Omega_n)$ and
${\bf \tilde{\Omega}}= (\tilde{\Omega}_1, \hdots, \tilde{\Omega}_n)$ by Proposition \ref{pro:parallel}.  Again, we consider
the different cases for $E$.

\textbf{Case $E = \emptyset$}. An $E$-nice basis for singular $1$-forms for ${\mathcal C}$
is clearly obtained by replacing $(\Omega_1, \Omega_2)$ with $(x dx, x df_{1})$ in
$\mathbf{\Omega}$.  We have
\[
  \nu_{{\mathcal C}} (xdx)  -\nu_{{\mathcal C}_{1}} (\tilde{\Omega}_{1}) =
  n \quad \mathrm{and} \quad
  \nu_{{\mathcal C}} (xdf_{1})  -\nu_{{\mathcal C}_{1}} (\tilde{\Omega}_{2}) =n .
\]
Since
$\nu_{{\mathcal C}} ({\Omega}_{j}) - \nu_{{\mathcal C}_{1}} (\tilde{\Omega}_{j}) = n
\lfloor \frac{j-1}{2} \rfloor$ for $j \geq 3$, it follows that
\[ \sharp ([\Lambda_{{\mathcal C}_{1}}^{{E}_{1}} + n] \setminus
  \Lambda_{\mathcal{C},\circ}^{E}) =
  \underbrace{0 +0 +0 +0 +1 +1 +2 +2 + \hdots}_{n \ \mathrm{terms}} = \sigma (n).
\]

\textbf{Case $E\equiv\{x=0\}$ or $E\equiv\{y=0\}$}. This means $\delta_x + \delta_y = 1$. An $E$-nice
basis for singular forms $1$-forms for ${\mathcal C}$ is obtained by simply replacing
$\Omega_1$ with $x \Omega_1$ in $\mathbf{\Omega}$.  Since
$\nu_{{\mathcal C}} (x \Omega_1) -\nu_{{\mathcal C}_{1}} (\tilde{\Omega}_{1}) =n$ and
$\nu_{{\mathcal C}} ({\Omega}_{j}) - \nu_{{\mathcal C}_{1}} (\tilde{\Omega}_{j}) = n
\lfloor \frac{j}{2} \rfloor$ for $j \geq 2$, we deduce
\[
  \sharp ([\Lambda_{{\mathcal C}_{1}}^{{E}_{1}} + n] \setminus
  \Lambda_{\mathcal{C},\circ}^{E}) =
  \underbrace{0 +0 +0 +1 +1 +2 +2 + \hdots}_{n \ \mathrm{terms}}=
  \sigma (n+1). 
\] 

\textbf{Case $E\equiv \{xy=0\}$}. That is, $\delta_x +\delta_{{y}}=2$. Since all the
$1$-forms that preserve $E$ are singular and
$\nu_{{\mathcal C}} ({\Omega}_{j}) - \nu_{{\mathcal C}_{1}} (\tilde{\Omega}_{j}) = n
\lfloor \frac{j+1}{2} \rfloor$ for $j \geq 1$, it follows that
\[
  \sharp ([\Lambda_{{\mathcal C}_{1}}^{{E}_{1}} + n] \setminus
  \Lambda_{\mathcal{C},\circ}^{E}) =
  \underbrace{0 +0 +1 +1 +2 +2 + \hdots}_{n \ \mathrm{terms}}= \sigma (n+2). 
\]
This completes the proof.
\subsection{The dimension formula}
Now, we prove Theorem \ref{teo:dim_gen_geo}.
\begin{proof}[Proof of Theorem \ref{teo:dim_gen_geo}]
Up to replacing $\Gamma$ with
  another curve in its equi\-si\-ngularity class ${\mathcal C}$, we can assume that $\Gamma_i$
  is generic in the equisingularity class of $(\Gamma_i, E_i)$ for any $0 \leq i < \tau$
  by Remark \ref{rem:generic} .  Now, Remark \ref{rem:gen_dim} and Propositions \ref{pro:calc_dim} and
  \ref{pro:blow-up-contacts} imply the result.
\end{proof}

  %Notice that $\sigma(2)=\sigma(3)=0$, so that we do not need to assume that $\Gamma_{\tau}$ is
  %transverse to $E_{\tau}$, just that it is non-singular. 

\bibliography{sb.bib}
\end{document}